\documentclass[a4paper, 11pt]{article}
\usepackage[top = 0.5 in, left = 0.5in, right = 0.5in, bottom = 0.7in]{geometry}
\usepackage{amsmath, amssymb, amsthm, indentfirst, placeins, float, graphicx, booktabs, verbatim, xcolor, titlesec, titling, mathrsfs, dsfont, multirow, tikz, natbib, subcaption, diagbox,  enumitem}

\usepackage{setspace}
\allowdisplaybreaks[4]
\pretitle{\begin{center}\bfseries\Large}
\posttitle{\par\end{center}}
\usepackage{caption}
\usepackage[hyperindex,breaklinks]{hyperref}
\hypersetup{linkcolor=blue, citecolor=teal, colorlinks=true}

\usepackage{mathtools}                      
\mathtoolsset{showonlyrefs=true}

\newtheorem{theorem}{Theorem}
\newtheorem{lemma}[theorem]{Lemma}

\newtheorem{proposition}[theorem]{Proposition}

\newtheorem{definition}[theorem]{Definition}

\newtheorem{remark}[theorem]{Remark}

\numberwithin{equation}{section}
\numberwithin{theorem}{section}

\newcommand\Lc{\mathcal{L}}

\newcommand\Uc{\mathcal{U}}

\newcommand{\Rb}{\mathbb{R}}

\newcommand{\dd}{\mathrm{d}}

\newcommand{\rev}[1]{{\leavevmode\color{blue}{#1}}}

\allowdisplaybreaks
\begin{document}
	
	\title{Optimal Dividend, Reinsurance, and Capital Injection for Collaborating Business Lines under Model Uncertainty}

    \author{
    Tim J. Boonen\thanks{Department of Statistics and Actuarial Science, School of Computing and Data Science, University of Hong Kong, China. Email: \texttt{tjboonen@hku.hk}.}
    \and
    Engel John C. Dela Vega\thanks{Department of Statistics and Actuarial Science, School of Computing and Data Science, University of Hong Kong, China. Email: \texttt{engel\_john.dela\_vega@mymail.unisa.edu.au}.}
    \and
    Len Patrick Dominic M. Garces\thanks{School of Mathematical and Physical Sciences, Faculty of Science, University of Technology Sydney, Australia.
    E-mail: \texttt{LenPatrickDominic.Garces@uts.edu.au}.}
    }
	
	\date{}	
	\maketitle

    \vspace*{-1.5cm} 
    
    {\setstretch{1}
    \begin{abstract}
		This paper considers an insurer with two collaborating business lines that faces three critical decisions: (1) dividend payout, (2) reinsurance coverage, and (3) capital injection between the lines, in the presence of model uncertainty. The insurer considers the reference model to be an approximation of the true model, and each line has its own robustness preference. The reserve level of each line is modeled using a diffusion process. The objective is to obtain a robust strategy that maximizes the expected weighted sum of discounted dividends until the first ruin time, while incorporating a penalty term for the distortion between the reference and alternative models in the worst‑case scenario. We completely solve this problem and obtain the value function and optimal (equilibrium) strategies in closed form. We show that the optimal dividend-capital injection strategy is a barrier strategy. The optimal proportion of risk ceded to the reinsurer and the deviation of the worst-case model from the reference model are decreasing with respect to the aggregate reserve level. Finally, numerical examples are presented to show the impact of the model parameters and ambiguity aversion on the optimal strategies.

        \medskip

	\end{abstract}
    }

    \section{Introduction}

    The optimal dividend payout problem is a classical topic in actuarial science, operational research, and mathematical finance. In his seminal work, \citet{definetti1957} proposed maximizing the dividends paid to the shareholders over the lifetime of an insurance portfolio as an alternative performance measure to minimizing the probability of ruin. Under the framework of \citet{definetti1957}, the optimal way to distribute dividends is the so-called \emph{barrier}-type strategy, in which all reserves exceeding a fixed barrier level are immediately paid out as dividends. Since then, numerous extensions and variations have been explored; see, e.g., \citet{schmidli2007book}, \citet{albrecher2009}, and \citet{avanzi2009} for comprehensive overviews. However, much of the literature has focused on single business lines. 

    We consider a multivariate extension of the classical framework of \citet{definetti1957} in which each business line has its own reserve process. The insurer's objective is to maximize the total discounted dividends paid by the business lines to the shareholders up to the portfolio's ruin time. This raises an important question: how should ruin time be defined in a multivariate setting? Several natural definitions appear in the literature. The most common definitions include \emph{first} ruin time: the first time when at least one of the reserve levels reaches or falls below zero; \emph{sum} ruin time: the first time that the total reserves of all lines reach or fall below zero; \emph{simultaneous} ruin time: the first time that all of the reserve levels reach or fall below zero simultaneously; and \emph{last} ruin time: the first time that all of the reserve levels, not necessarily simultaneously, have reached or fallen below zero. As such, the problem of minimizing the probability of ruin has been widely studied; see, e.g., \citet[Chapter XIII]{asmussenbook2010} and the references therein.

    There have been several works in the multivariate framework that focus on maximizing the total discounted dividends to shareholders until ruin time. To the best of the authors' knowledge, \citet{czarna2011} are the first to study this problem under the Cram\'{e}r-Lundberg (CL) model (i.e., a compound Poisson process). Most subsequent work has also used the CL model; see \citet{liucheung2014}, \citet{albrecher2017}, \citet{azcue2019}, \citet{azcuemuler2021}, and \citet{strietzel2022}. Reserve levels modeled by diffusion processes have also been explored in the multivariate context; see \citet{gu2018}, \citet{grandits2019saj}, \citet{yang2025}, \citet{boonen2025collaboratingbusinesslines}, and \citet{boonen2025xsofloss}.

    Managing risk across multiple lines of business is also essential for ensuring that sufficient reserves are available for dividend payouts. Reinsurance, as a form of risk control, has been explored in the multivariate framework of optimal dividend payout problems. Proportional reinsurance is the predominant type; see \citet{czarna2011}, \citet{liucheung2014}, \citet{azcue2019}, \citet{strietzel2022}, \citet{boonen2025collaboratingbusinesslines}, and \citet{yang2025}. Excess-of-loss reinsurance has also been studied in \citet{boonen2025xsofloss}.

    Capital injection is another form of risk management that allows insurers to prevent ruin. In the context of multiple business lines, capital injection can be made between lines to save one line from insolvency, provided sufficient reserves are available. This type of capital injection is referred to as \emph{collaboration}. In \citet{albrecher2017}, \citet{gu2018}, \citet{boonen2025collaboratingbusinesslines}, and \citet{boonen2025xsofloss}, the rule of collaboration permits the transfer of reserves between lines only when one is at risk of ruin and provided that the transfer does not endanger the solvent line. In contrast, \citet{grandits2019saj} allows for capital transfers but does not oblige any line to save another.

    A common feature of the above-mentioned literature is the assumption that the probabilistic model that governs the reserve processes is known. Parameters such as the drift, volatility, and even the claim intensity are assumed to be accurate. In practice, however, these quantities are estimated from historical data and are subject to statistical error and misspecification. As a result, strategies obtained under a single reference model may perform poorly. Ignoring model uncertainty may lead to overly aggressive dividend payout policies, insufficient reinsurance coverage, or weak capital injection rules. Model ambiguity has been considered in several insurance decision problems, including reinsurance with multiple reinsurers; see, e.g., \citet{cao2023reinsurance}. For a comprehensive discussion of decision-making under model uncertainty, we refer the reader to \citet{hansen:robustness}.

    Motivated by these considerations, we incorporate model uncertainty into the problem of optimal dividend, reinsurance, and capital injection between collaborating business lines. More precisely, the reference probability measure is regarded by the insurer as an approximation to the true, albeit unknown, model, which is assumed to belong to a set of alternative probability measures. The insurer then considers the strategies under the worst-case model within this set, which, consequently, promotes robustness against possible model misspecification. To prevent the insurer from accepting implausible alternative models that are ``too far" from the reference model, a discounted penalty term is introduced in the optimization criterion. This penalty is based on relative entropy, which quantifies the distance (i.e., divergence) between probability measures, and is modified to account for both a normalization factor and line-specific ambiguity aversion. The resulting formulation falls within a penalty-based robust control approach pioneered by \citet{anderson2003} and \citet{hansen2006}. More specifically, the formulation of decision-making under model uncertainty in this paper follows the homothetic robustness formulation proposed by \citet{maenhout2004}. In this formulation, the penalty term is dependent on the value function itself, which preserves the analytical tractability of the robust control problem. This approach to modeling uncertainty has been adopted in robust portfolio optimization \citep{weiyangzhuang2023,garcesshen2025} and robust reinsurance \citep{huchenwang2018,feng2021}.

    Robust formulations of the optimal dividend payout problem have been studied in the univariate setting. \citet{feng2021} consider proportional reinsurance with an Ornstein-Uhlenbeck-type reserve process. \citet{chakraborty2023} consider a Brownian risk model with a generalized reward function and a penalty term based on the Kullback-Leibler divergence. Capital injection via equity issuance has been studied in \citet{wang2025} under the standard deviation premium principle, and robust dividend strategies within the classical CL framework are investigated in \citet{feng2024}. The impact of model uncertainty in a multivariate setting remains largely unexplored.

    In this work, we consider an insurer that must simultaneously determine (i) dividend payout policies, (ii) (proportional) reinsurance coverage levels, and (iii) capital injections between collaborating business lines under model uncertainty. We apply a diffusion model for the risk exposure of each line, and we allow correlation between the lines. The goal is to maximize the weighted total discounted dividends until the first ruin time, taking into account the worst-case scenario. We incorporate a penalty term based on relative entropy, which penalizes deviations from the reference model. To the best of our knowledge, this is the first study on an optimal dividend payout problem involving reinsurance coverage and capital injection between collaborating lines while explicitly accounting for model uncertainty.

    We summarize the main contributions of this paper as follows:

    \begin{enumerate}
        \item We identify three general scenarios and derive, in closed form, the corresponding value function and optimal strategies for each scenario (see Theorems \ref{theorem 1}, \ref{theorem 2}, and \ref{theorem 3}). The scenarios are determined by whether the ambiguity aversion of each line is sufficiently large and whether the interplay between the parameters, particularly risk, ambiguity, and correlation, satisfies certain technical conditions. This is analogous to the results in \citet{feng2021}, who also identify three scenarios in the univariate setting with an Ornstein-Uhlenbeck-type reserve process. Setting $\rho=0$ in their notation corresponds to our case where one line has ceded all of its risk to the reinsurer (see Section \ref{subsec: remaining case}).

        \item We show that an optimal dividend and capital injection strategy is a barrier strategy. More precisely, we prove that there exists a barrier level $b\geq 0$ such that the optimal strategy maintains the aggregate reserve below $b$, with dividends paid exclusively by the more important line, and capital is transferred from the less important line, ensuring sufficient reserves for dividend distribution. This is consistent with the results presented in \citet{gu2018}, where they also discuss dividend payout with capital injections between collaborating lines. For sufficiently high ambiguity aversion, the optimal barrier collapses to $b=0$, implying that the optimal reinsurance and distortion strategies become irrelevant (see Theorem \ref{theorem 3}).

        \item We prove that, for sufficiently low ambiguity aversion, the optimal retained reinsurance levels are increasing with respect to the aggregate reserve level. Moreover, except in the case where one line adopts full reinsurance at all times (see Proposition \ref{prop case 2}), the two optimal reinsurance levels are simultaneously constant. The monotonicity property is consistent with existing results in the literature, for instance in \citet{hojgaard1999} for the univariate case and \citet{boonen2025collaboratingbusinesslines} for the multivariate case.

        \item We show that, for sufficiently low ambiguity aversion, the deviation of the worst-case probability measure from the reference measure decreases as the aggregate reserve level increases. The deviation remains constant and attains its maximal level as long as both lines engage in proportional reinsurance. Once one line switches to full reinsurance, the deviation strictly decreases.

        \item We further explore the effects of model uncertainty via the ambiguity-aversion parameters on the optimal strategies in our numerical illustrations in Section \ref{sec:numerical}. We observe that the deviation of the worst-case model from the reference model becomes less sensitive to further increases in the aggregate reserve level. In addition, changes in the ambiguity-aversion parameters can alter which line retains all of its risk beyond a certain reinsurance threshold. Finally, we find that strong asymmetry in the ambiguity-aversion parameters between the two lines can lead to a lower dividend barrier and a lower reserve threshold at which one line optimally retains all of its risk.
        
    \end{enumerate}

The rest of the paper is organized as follows. Section \ref{sec:model} introduces the model, the formulation of the problem, and the corresponding verification theorem. Section \ref{sec:AnalyticalSol} presents the main results, while Section \ref{sec:no mod.unc.} discusses the results under the assumption of no model uncertainty. Numerical examples are presented in Section \ref{sec:numerical}. The proofs of the main results are presented in Section \ref{sec:proofs}. Section \ref{sec:conclusion} concludes. Appendices \ref{app: verification proof} and \ref{app: theorem no model unc} provide the proofs of the verification theorem and the results concerning no model uncertainty, respectively, and Appendix \ref{sec:MU comp b} provides additional numerical results.
    
	\section{Model}\label{sec:model}

    \subsection{Problem Formulation}
	We fix a complete filtered probability space $(\Omega,\mathcal{F},\mathbb{F},\mathbb{P})$, where $\mathbb{F}:=\{\mathcal{F}_t\}_{t\geq 0}$ is a right-continuous, $\mathbb{P}$-completed filtration generated by two independent Brownian motions $B_1:=\{B_1(t)\}_{t\geq 0}$ and $B_2:=\{B_2(t)\}_{t\geq 0}$. Define two correlated Brownian motions $W_1:=\{W_1(t)\}_{t\geq 0}$ and $W_2:=\{W_2(t)\}_{t\geq 0}$ by $W_1(t) := B_1(t)$ and $W_2(t) := \rho B_1(t) + \sqrt{1-\rho^2}\, B_2(t)$,
	where $\rho\in(-1,1)$ captures the correlation between the two Brownian motions $W_1$ and $W_2$. 
    
	We consider an insurer that begins with a reference model for the reserve process under the reference probability measure $\mathbb{P}$, which serves as an approximation to the ``true" model. Moreover, the insurer has two collaborating business lines, where the risk exposure of Line $i$, with $i=1,2$, is governed by a diffusion process given by
	\begin{equation}
		\label{eq:dR}
		\dd R_i(t) = \tilde\mu_i \, \dd t + \sigma_i \, \dd W_i(t),
	\end{equation}
    where $\tilde \mu_i,\sigma_i > 0$. The insurer charges a premium for taking the risk $R_i$ via the expected-value principle with a loading factor $\alpha_i\geq 0$; that is, the premium rate for Line $i$ is given by $\tilde P_i(t) = (1 + {\alpha}_i) \tilde\mu_i.$
    
    The insurer is faced with three types of decisions regarding the operation of each line: reinsurance, dividend payout, and capital injection between the lines. Each decision is described as follows:
    \begin{enumerate}
        \item Reinsurance decision. For each line, the insurer can purchase \emph{proportional} reinsurance to transfer part of its risk $R_i$ to a reinsurer. For $i=1,2$, let $\pi_i(t) \in [0,1]$ denote the retained proportion of Line $i$ at time $t$ and $P_i(t)$ the premium for such a reinsurance policy. The premium is calculated using the same expected-value principle and relative safety loading $\alpha_i$; that is, $P_i(t) = (1 + \alpha_i) \, (1-\pi_i(t) )\, \tilde\mu_i$.
        The loadings for insurance and reinsurance are equal, which is sometimes referred to as ``cheap reinsurance".

    \item Dividend payout decision. The insurer also chooses a dividend strategy to distribute profits to the shareholders for each line. Let $C_i:=\{C_i(t)\}_{t\geq 0}$ denote the total dividends paid by Line $i$ to the shareholders. This is called an unrestricted (or unbounded) dividend-payment strategy (\citealp[see Section 3 in][or Case C in]{hojgaard1999} \citealp{jeanblanc1995optimization}). We treat $C_1$ and $C_2$ as singular-type controls. 

    \item Capital injection decision. In addition to seeking protection from outside reinsurers, the insurer can also allocate surplus from one line to another, at \emph{no cost}. Let $L_i:=\{L_i(t)\}_{t\geq 0}$ denote the \emph{cumulative} amount of capital transferred into Line $i$ from Line $3-i$; both $L_1$ and $L_2$ are nondecreasing, c\'{a}dl\'{a}g, singular-type controls.
    \end{enumerate}

	Given an admissible control $u:=(\pi_1, \pi_2, C_1, C_2, L_1, L_2)$, the reserve process of Line $i$, denoted by $X_i:=X_i^u$, follows the dynamics
	\begin{equation}
		\begin{aligned}\label{eq:SurplusProcess}
			\dd X_i(t) &= \big( \tilde P_i(t) - P_i(t) \big) \, \dd t - \pi_i(t)  \, \dd R_i(t) - \dd C_i(t) + \dd L_i(t) - \dd L_{3-i}(t) ,  \\
			&= \mu_i\pi_i(t) \dd t - \sigma_i \pi_i(t) \dd W_i(t) - \dd C_i(t) + \dd L_i(t) - \dd L_{3-i}(t),
		\end{aligned}
	\end{equation}
	where $\mu_i:=\alpha_i\tilde\mu_i$ is the adjusted mean and $X_i(0)=x_i>0$ is the initial reserve level of Line $i$.

    In the framework of model uncertainty, due to insufficient data, the insurer does not have full confidence in the reference model. The insurer is ambiguity averse and thus prefers to consider a family of unspecified alternative models, which are not too far away from the reference model. Write $\theta_i:=\{\theta_i(t)\}_{t\geq 0}$ for $i=1,2$, and set
$\theta:=(\theta_1,\theta_2)$. We assume that $\theta$ is progressively
measurable with respect to the filtration $\mathbb F$ and satisfies the
following Novikov-type condition: for any $t>0$,
\begin{equation}\label{eqn:Novikov}
\mathbb E^{\mathbb P}\!\left[
\exp\!\left(
\frac{1}{2}\int_0^t
\Big(
\theta_1^2(s)
+ \frac{\big(\theta_2(s)-\rho\,\theta_1(s)\big)^2}{1-\rho^2}
\Big)\,\dd s
\right)\right]<\infty.
\end{equation}

We define the alternative probability measure $\mathbb Q^{\theta}$ on
$\mathcal F_t$ by
\begin{equation*}
\frac{\dd \mathbb Q^{\theta}}{\dd\mathbb P}\Bigg|_{\mathcal F_t}
=
\exp\Bigg(
\int_0^t \theta_1(s)\,\dd B_1(s)
+
\int_0^t \frac{\theta_2(s)-\rho\,\theta_1(s)}{\sqrt{1-\rho^2}}\,\dd B_2(s)
-\frac{1}{2}\int_0^t
\Big[
\theta_1^2(s)
+
\frac{\big(\theta_2(s)-\rho\,\theta_1(s)\big)^2}{1-\rho^2}
\Big]\dd s
\Bigg).
\end{equation*}
By Girsanov's theorem for multidimensional Brownian motion (see, for example,
\citealp[Theorem 8.6.6,]{oksendal:sde} or \citealt[page 21]{pham:control}), for any
time $t_0>0$ the processes $\{W^{\mathbb Q^{\theta}}_i(t)\}_{t \geq 0}$, where
\begin{equation*}
W^{\mathbb Q^{\theta}}_i(t)
:= W_i(t) - \int_0^t \theta_i(s)\,\dd s,\qquad
0 \leq t \leq t_0,\ i=1,2,
\end{equation*}
form a two-dimensional Brownian motion with correlation $\rho$ under
$\mathbb Q^{\theta}$. It should also be noted that if
$\theta_1(t)=\theta_2(t)\equiv0$, then the alternative measure $\mathbb Q^{\theta}$
coincides with the reference measure $\mathbb P$.

We can then express the dynamics of the controlled reserve process of Line $i$
under $\mathbb Q^{\theta}$ as
\begin{align}
\dd X_i(t)
&= \mu_i\pi_i(t)\,\dd t
- \sigma_i\pi_i(t)\,\dd W_i(t)
- \dd C_i(t) + \dd L_i(t) - \dd L_{3-i}(t)\\
&= \mu_i\pi_i(t)\,\dd t
- \sigma_i\pi_i(t)\Big[\dd W^{\mathbb Q^{\theta}}_i(t)
+ \theta_i(t)\,\dd t\Big]
- \dd C_i(t) + \dd L_i(t) - \dd L_{3-i}(t)\\
&= \big[\mu_i\pi_i(t)-\sigma_i\theta_i(t)\pi_i(t)\big]\dd t
- \sigma_i\pi_i(t)\,\dd W^{\mathbb Q^{\theta}}_i(t)
- \dd C_i(t) + \dd L_i(t) - \dd L_{3-i}(t),
\end{align}
for $i=1,2$.

We formally define admissible strategies below.
    \begin{definition}
        The strategies $u$ and $\theta$ are said to be admissible if $u$ is adapted to the filtration $\mathbb F$, $\theta$ is progressively measurable with respect to $\mathbb F$, and they satisfy
        \begin{enumerate}[label=(\roman*)]
            \item $\pi_i(t)\in[0,1]$ for $i=1,2$ and $t\geq 0$;
            \item $C_i$ and $L_i$ are nonnegative, nondecreasing, and right-continuous with left limits, for $i=1,2$;
            \item $(\theta_1,\theta_2)$ satisfies Novikov's condition in \eqref{eqn:Novikov} for $t\geq 0$.
        \end{enumerate}
        Denote by $\Uc$ the set of all admissible $u$, and $\Theta$ the set of all admissible $\theta$. 
    \end{definition}

    In this setting, the insurer has reservations about the reference model and views it merely as an approximation to the ``true" model. Since the ``true" model (if one exists) is also in the set of alternative models, the ambiguity-averse insurer wants to find a robust decision rule that is feasible across all alternative models. A natural approach is to analyze the problem from a worst-case perspective.

	We now formulate the corresponding robust control problem. Define the (first) ruin time, $\tau^u$, by \[\tau^u := \inf\{t > 0 : \min\{X^u_1(t),X^u_2(t)\} \leq  0 \}.\]
	The goal of the insurer is to seek a robust optimal control that solves the following optimization problem:
    \begin{equation}\label{eqn: value function}
        V(x_1,x_2)=\sup_{u\in\mathcal U} \inf_{\theta \in \Theta} \mathbb E^{\mathbb Q^{\theta}}\left[\sum_{i=1}^2a_i\left(\int_0^{\tau^u}e^{-\delta t}\dd C_i(t)+\int_0^{\tau^u}e^{-\delta t}\frac{\theta_i^2(t)V(X^{u}_1(t),X^u_2(t))}{2\widetilde\beta_i}\dd t\right)\right]  ,
    \end{equation}
    where $a_i\in(0,1)$ is the relative importance of Line $i$ and satisfies $a_1+a_2=1$, $\widetilde\beta_i$ is the robustness preference parameter of Line $i$, and $\delta$ is the discount rate. Moreover, the expectation $\mathbb E^{\mathbb Q^{\theta}}$ is taken under $X^u_1(0)=x_1$ and $X^u_2(0)=x_2$. Larger $\widetilde \beta_i$
 implies a lower penalty on model distortion and therefore stronger ambiguity aversion in the worst-case formulation.

    \begin{remark}\label{remark: zero-sum}
        The first term in \eqref{eqn: value function} is similar to the objective in \cite{gu2018}, where they also consider an insurer with two collaborating lines. However, they do not allow the insurer to purchase reinsurance coverage and there is no model uncertainty. 
        
        \cite{feng2021} also consider a similar optimization problem with both terms in \eqref{eqn: value function} present in their objective function, but in the univariate framework. The second term can be interpreted as a ``penalty" with respect to the minimization problem over $\theta$, derived from the form of discounted relative entropy. 

        We can also see the optimization problem in \eqref{eqn: value function} as a zero-sum stochastic differential game whose players are (i) the insurer with control $u$ and (ii) the market with control $\theta$.
    \end{remark}

    Denote $\vartheta:=(\pi_1,\pi_2,\theta_1,\theta_2)$ and define the generator $\Lc_0^{\vartheta}(\phi)$ for some $\mathrm{C}^{2,2}$ function $\phi$ by 
	\begin{equation}\label{eq:generator}
		\begin{aligned}
			\Lc_0^{\vartheta}(\phi)&=\sum_{i=1}^2\left[\left[\mu_i\pi_i-\sigma_i\theta_i\pi_i\right]\frac{\partial \phi}{\partial x_i}+\frac{1}{2}\sigma_i^2\pi_i^2\frac{\partial^2 \phi}{\partial x_i^2}\right]+\rho\sigma_1\sigma_2\pi_1\pi_2\frac{\partial^2 \phi}{\partial x_1\partial x_2}-\delta \phi.
		\end{aligned}
	\end{equation}
    Write $\beta_i:=\frac{\widetilde \beta_i}{a_i}$ for $i=1,2$. The associated HJB equation is given by
		\begin{equation}\label{eqn:hjborig}
        \begin{aligned}
			&\sup\Bigg\{\sup_{\pi_i\in[0,1]}\inf_{\theta_i\in\mathbb R}\left[\Lc_0^{\vartheta}(V)(x_1,x_2)+V(x_1,x_2)\sum_{i=1}^2\frac{\theta_i^2}{2\beta_i}\right], a_1-\frac{\partial V}{\partial x_1}(x_1,x_2), a_2-\frac{\partial V}{\partial x_2}(x_1,x_2),\\
            &\qquad\qquad\frac{\partial V}{\partial x_1}(x_1,x_2)-\frac{\partial V}{\partial x_2}(x_1,x_2),  \frac{\partial V}{\partial x_2}(x_1,x_2)-\frac{\partial V}{\partial x_1}(x_1,x_2)\Bigg\} = 0,
            \end{aligned}
		\end{equation}
		with boundary condition $V(0,0)=0$.

Since $\frac{\partial V}{\partial x_1}-\frac{\partial V}{\partial x_2} \le 0$ and  $\frac{\partial V}{\partial x_2}-\frac{\partial V}{\partial x_1} \le 0$ hold simultaneously by \eqref{eqn:hjborig}, it follows that $\frac{\partial V}{\partial x_1} = \frac{\partial V}{\partial x_2}$. As such, there exists a univariate function, $g: x \in \Rb_+ \mapsto  \Rb$, such that 
\begin{align}
    g(x) = V(x_1, x_2), \quad \text{with } x := x_1 + x_2 \ge 0.
\end{align}
Using this representation of $g$, we have 
\begin{align*}
    g'(x) = \frac{\partial V}{\partial x_i}(x_1, x_2) \quad \text{and} \quad 
    g''(x) = \frac{\partial^2 V}{\partial x_i \partial x_j}(x_1, x_2), \quad i,j = 1, 2.
\end{align*}
Define the one-dimensional generator $\Lc^{\vartheta}(\varphi)(x):=\Lc^{\vartheta}_0(\phi)(x_1,x_2)$, where $\varphi(x)=\phi(x_1,x_2)$ with $x:=x_1+x_2$ for some $C^2$ function $\varphi$. We can then rewrite \eqref{eqn:hjborig} as
\begin{equation}\label{eqn:hjb rewritten}
        \begin{aligned}
			&\sup\Bigg\{\sup_{\pi_i\in[0,1]}\inf_{\theta_i\in\mathbb R}\left[\Lc^{\vartheta}(g)(x)+g(x)\sum_{i=1}^2\frac{\theta_i^2}{2\beta_i}\right], a_1-g'(x), a_2-g'(x)\Bigg\} = 0,
            \end{aligned}
\end{equation}
with boundary condition $g(0)=0$.
Based on the structure of the HJB equations \eqref{eqn:hjborig} and \eqref{eqn:hjb rewritten}, we conjecture that the optimal dividend-capital injection strategy follows a barrier strategy similar to that in \cite{gu2018}. We give the definition of barrier strategies and then provide the verification theorem below.

\begin{definition}\label{definition of barrier}
    Let $(C_{1,b},C_{2,b},L_{1,b},L_{2,b}):=(\{C_{1,b}(t)\},\{C_{2,b}(t)\},\{L_{1,b}(t)\},\{L_{2,b}(t)\})_{t\geq 0}$ denote the barrier strategy with a barrier $b$. For this strategy, we partition the domain of the reserve level pair $(x_1,x_2)\in\mathbb R_+^2$ into three regions (see Figure \ref{fig:capitaltransfer unbdd}). The three regions $A_i$, $i=1,2,3$, are defined as follows:
	\begin{itemize}
		\item $A_1=\{(x_1,x_2):x_2\geq 0,\, x_1>b\},$
		\item $A_2=\{(x_1,x_2):x_1\in[0,b],\, x_2>0,\, x_1+x_2>b\},$
		\item $A_3=\{(x_1,x_2):x_1\geq 0, \,x_2\geq 0,\, x_1+x_2\leq b\}.$
	\end{itemize}
    
    \begin{figure}
		\centering
		\begin{tikzpicture}[scale=0.8] 
			\draw[->] (-1,0) -- (4,0) node[right] {\(x_1\)};
			\draw[->] (0,-1) -- (0,4) node[above] {\(x_2\)};
			
			\draw (0,3) node[left] {\((0,b)\)} -- (3,0) node[below] {\((b,0)\)};
			\draw (3,0) -- (3,4) node[right] {}; 
			
			\filldraw[black] (0,3) circle (2pt); 
			\filldraw[black] (3,0) circle (2pt);
			
			\node at (3.5, 2) {$A_1$};
			\node at (2, 2) {$A_2$};
			\node at (0.75, 1) {$A_3$};

			\node[rotate=-45] at (1.5, 1.3) {\footnotesize\(x_1+x_2 = b\)};
			\node[rotate=-90] at (3.15, 3.45) {\footnotesize\(x_1 = b\)};
		\end{tikzpicture}
		\caption{Regions for dividend payout and capital injection decisions.}
		\label{fig:capitaltransfer unbdd}
	\end{figure}
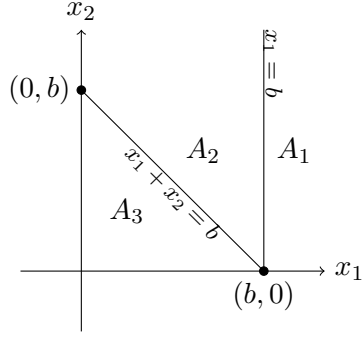
    The barrier strategy is then as follows:
    \begin{enumerate}[label=(\roman*)]
        \item If $(x_1,x_2)\in A_1$, Line 1 transfers an amount $x_1-b$ to Line 2 and we proceed to (ii). 
        \item If $(x_1,x_2)\in A_2$, Line 2 directly pays an amount $x_1+x_2-b$ as dividends and we proceed to (iii).
        \item If $(x_1,x_2)\in A_3$, Line 2 keeps the aggregate reserve reflected at $b$ whenever the aggregate reserve level reaches $b$ from below; that is, we stay in region $A_3$ until ruin. When the reserve pair reaches either axis, the insurer applies capital transfers from the other line to keep both reserve levels positive. The problem ends when the reserve pair reaches $(0,0)$.
    \end{enumerate}
\end{definition}

\subsection{Verification Theorem}

In traditional dividend problems where there is no model uncertainty, the techniques used to prove the corresponding verification theorem may not be applicable when model uncertainty is present. Since the optimization problem \eqref{eqn: value function} can be seen as a zero-sum stochastic differential game, different techniques are required to establish and prove its verification theorem. We now proceed to present the verification theorem.

\begin{theorem}[Verification Theorem]\label{thm: verification}
Let $X^{u}(t):=X^{u}_1(t)+X^{u}_2(t)$. Suppose there exist: 
\begin{enumerate}[label=(\roman*)]
    \item an increasing concave function $w(x) \in \mathrm C^2$ with $w'(x) > a_2$ for $x \in [0,b^*)$, $w'(x)=a_2$ for $x \in [b^*,\infty)$, $w(0)=0$, 
    \item a Markov control $u^*(X^{u^*}(t))=(\pi_1^*(X^{u^*}(t)),\pi_2^*(X^{u^*}(t)),C_{1,b^*},C_{2,b^*},L_{1,b^*},L_{2,b^*})\in\Uc$, where the dividend-capital injection strategy $(C_{1,b^*}, C_{2,b^*}, L_{1,b^*}, L_{2,b^*})$ is a barrier strategy, and 
    \item a Markov process $\theta^*:=(\{\theta_1^*(X^{u^*}(t))\}_{t\geq 0},\{\theta_2^*(X^{u^*}(t))\}_{t\geq 0})\in\Theta$,
\end{enumerate}
such that
\begin{align}
    &\Lc^{(\pi_1,\pi_2,\theta_1^*(x),\theta_2^*(x))}(w)(x)+\sum_{i=1}^2\frac{(\theta_i^*(x))^2}{2\beta_i}w(x) \leq 0, \quad \mbox{for all $\pi_1,\pi_2\in[0,1]$ and $x\geq 0$,}\label{proof eqn 9}\\
    &\Lc^{(\pi^*_1(x),\pi^*_2(x),\theta_1,\theta_2)}(w)(x)+\sum_{i=1}^2\frac{\theta_i^2}{2\beta_i}w(x) \geq 0, \quad \mbox{for all $\theta_1,\theta_2\in \mathbb R$ and $x\in [0,b^*]$,}\label{proof eqn 10}\\
    &\Lc^{(\pi^*_1(x),\pi^*_2(x),\theta_1^*(x),\theta_2^*(x))}(w)(x)+\sum_{i=1}^2\frac{(\theta_i^*(x))^2}{2\beta_i}w(x) = 0, \quad \mbox{for all $x\in [0,b^*]$.} \label{proof eqn 11}
\end{align}
Then,
\begin{align}
    w(x)
    &=\sup_{u\in\mathcal U} \inf_{\theta\in\Theta} \mathbb E^{\mathbb Q^{\theta}}\left[\sum_{i=1}^2a_i\left(\int_0^{\tau^u}e^{-\delta t}\dd C_i(t)+\int_0^{\tau^u}e^{-\delta t}\frac{\theta_i^2(t)w(X^{u}(t))}{2\widetilde\beta_i}\dd t\right)\right] \\
    &=\inf_{\theta\in\Theta}\sup_{u\in\mathcal U}  \mathbb E^{\mathbb Q^{\theta}}\left[\sum_{i=1}^2a_i\left(\int_0^{\tau^u}e^{-\delta t}\dd C_i(t)+\int_0^{\tau^u}e^{-\delta t}\frac{\theta_i^2(t)w(X^{u}(t))}{2\widetilde\beta_i}\dd t\right)\right] \\
    &=\mathbb E^{\mathbb Q^{\theta^*}}\left[\sum_{i=1}^2a_i\left(\int_0^{\tau^{u^*}}e^{-\delta t}\dd C_{i,b^*}(t)+\int_0^{\tau^{u^*}}e^{-\delta t}\frac{(\theta^*_i(X^{u^*}(t)))^2w(X^{u^*}(t))}{2\widetilde\beta_i}\dd t\right)\right].
\end{align}
Moreover, $u^*$ and $\theta^*$ are the optimal strategies; that is, $(u^*,\theta^*)$ is the Nash equilibrium of the zero-sum stochastic differential game.
\end{theorem}
\begin{proof}
    See Appendix \ref{app: verification proof}.
\end{proof}
Using the verification theorem, if we can find a classical solution to the HJB equation \eqref{eqn:hjborig}, then the solution is equal to the value function in \eqref{eqn: value function}. In the next section, we identify such classical solutions.

\section{Analytical Solutions}
\label{sec:AnalyticalSol}

In this section, we study the insurer's optimization problem in \eqref{eqn: value function} and obtain the optimal strategies $u^*$ and $\theta^*$, and the value function $V(x_1,x_2)=g(x)$. Due to symmetry between the two lines, we assume, without loss of generality, that $a_1\leq \frac{1}{2}$ for the rest of the paper (i.e., Line 2 is the more important line). 

First, we isolate the optimization over $\theta_i$ (distortion decision) in \eqref{eqn:hjb rewritten} and solve
\begin{equation}
    \inf_{\theta_1\in\mathbb{R},\theta_2\in\mathbb{R}} \left(\frac{\theta_1^2}{2\beta_1}+\frac{\theta_2^2}{2\beta_2}\right)g(x) - \left(\sigma_1\pi_1\theta_1+\sigma_2\pi_2\theta_2\right)g'(x),
\end{equation}
from which we obtain the candidate minimizers as
\begin{equation}\label{eqn:theta candidates}
    \widehat\theta_1 (x) = \frac{\beta_1\sigma_1\pi_1 g'(x)}{g(x)} \quad \mbox{and} \quad \widehat\theta_2 (x) = \frac{\beta_2\sigma_2\pi_2 g'(x)}{g(x)}.
\end{equation}

Using \eqref{eqn:theta candidates}, we can then solve the optimization over $\pi_i$ (reinsurance decision) in \eqref{eqn:hjb rewritten}:
\begin{align*}
        & \sup_{\pi_1\in[0,1],\pi_2\in[0,1]} \left\{\left(\frac{1}{2}\sigma_1^2\pi_1^2+\frac{1}{2}\sigma_2^2\pi_2^2+\rho\sigma_1\sigma_2\pi_1\pi_2\right)g''(x) + \left(\mu_1\pi_1+\mu_2\pi_2\right)g'(x) 
        -\frac{1}{2}\left(\beta_1\sigma_1^2\pi_1^2+\beta_2\sigma_2^2\pi_2^2\right) \frac{g'(x)^2}{g(x)} \right\},
\end{align*}
and, ignoring the constraints over $[0,1]$, we obtain the following candidate maximizers:

\begin{equation}\label{eqn: optimal pi}
\begin{aligned}
    \widehat \pi_1(x)&=\frac{(\mu_1\sigma_2-\rho\mu_2\sigma_1)g''(x)g'(x)g(x)^2-\beta_2\mu_1\sigma_2g'(x)^3g(x)}{\sigma_1^2\sigma_2\left[(\beta_1+\beta_2)g''(x)g'(x)^2g(x)-(1-\rho^2)g''(x)^2g(x)^2-\beta_1\beta_2g'(x)^4\right]},\\
    \widehat \pi_2(x)&=\frac{(\mu_2\sigma_1-\rho\mu_1\sigma_2)g''(x)g'(x)g(x)^2-\beta_1\mu_2\sigma_1g'(x)^3g(x)}{\sigma_1\sigma_2^2\left[(\beta_1+\beta_2)g''(x)g'(x)^2g(x)-(1-\rho^2)g''(x)^2g(x)^2-\beta_1\beta_2g'(x)^4\right]}.
\end{aligned}
\end{equation}

The analysis of the optimal strategies below is divided into cases involving the relative Sharpe ratio $\frac{\mu_1 / \sigma_1}{\mu_2 / \sigma_2}$, the correlation $\rho$, and the ambiguity-aversion parameters $\beta_1$ and $\beta_2$.

\subsection{The Case of $0<\frac{\rho}{1+\beta_2\frac{\gamma_1}{1-\gamma_1}}\leq \frac{\mu_1 /\sigma_1}{\mu_2/\sigma_2}\leq \frac{1+\beta_1\frac{\gamma_1}{1-\gamma_1}}{\rho}$}\label{subsec: main case}

In this section, we consider the following condition:
\begin{equation}\label{eqn: correlation bounds}
    0<\frac{\rho}{1+\beta_2\frac{\gamma_1}{1-\gamma_1}}\leq \frac{\mu_1 /\sigma_1}{\mu_2/\sigma_2}\leq \frac{1+\beta_1\frac{\gamma_1}{1-\gamma_1}}{\rho},
\end{equation}
where $\gamma_1$ is a solution to $\psi(z)=0$ on $(0,1)$ and $\psi$ is defined as 
\begin{equation}\label{eqn: defn of psi}
    \begin{aligned}
        \psi(z)&:=\widetilde A(z) - 2\delta \sigma_1^2\sigma_2^2 \widetilde B(z),\\
        \widetilde A(z)
    &:=[\mu_1^2\sigma_2^2(3\rho+1)(\rho-1)+\mu_2^2\sigma_1^2(3\rho+1)(\rho-1)+2\mu_1\mu_2\sigma_1\sigma_2(1-\rho)(2\rho^2+\rho+1)]z(z-1)^3\\
    &\qquad+[-\beta_2\mu_1^2\sigma_2^2(2\beta_1+\beta_2)-\beta_1\mu_2^2\sigma_1^2(\beta_1+2\beta_2)+2\beta_1\beta_2\mu_1\mu_2\sigma_1\sigma_2]z^3(z-1)\\
    &\qquad+[\mu_1^2\sigma_2^2(\beta_1-3\rho^2\beta_2+2\rho\beta_2+2\beta_2)+\mu_2^2\sigma_1^2(\beta_2-3\rho^2\beta_1+2\rho\beta_1+2\beta_1)\\
    &\qquad\qquad-2\mu_1\mu_2\sigma_1\sigma_2(\beta_1+\beta_2)]z^2(z-1)^2+\left[\beta_1\beta_2^2\mu_1^2\sigma_2^2+\beta_1^2\beta_2\mu_2^2\sigma_1^2\right]z^4,\\
    \widetilde B(z)
    &:=[(\beta_1+\beta_2)^2+2(1-\rho^2)\beta_1\beta_2]z^2(z-1)^2+(1-\rho^2)^2(z-1)^4+\beta_1^2\beta_2^2z^4\\
    &\qquad-2(1-\rho^2)(\beta_1+\beta_2)z(z-1)^3-2\beta_1\beta_2(\beta_1+\beta_2)z^3(z-1).
    \end{aligned}
\end{equation}
We introduce the following notations that will be used frequently in the analysis:
\begin{equation}\label{eqn: w1 and w2}
\begin{aligned}
    w_1&:=\frac{\sigma_1^2\sigma_2[(\beta_1+\beta_2)\gamma_1(\gamma_1-1)-(1-\rho^2)(\gamma_1-1)^2-\beta_1\beta_2\gamma_1^2]}{(\mu_1\sigma_2-\rho\mu_2\sigma_1)(\gamma_1-1)-\beta_2\mu_1\sigma_2\gamma_1} ,\\
    w_2&:=\frac{\sigma_1\sigma_2^2[(\beta_1+\beta_2)\gamma_1(\gamma_1-1)-(1-\rho^2)(\gamma_1-1)^2-\beta_1\beta_2\gamma_1^2]}{(\mu_2\sigma_1-\rho\mu_1\sigma_2)(\gamma_1-1)-\beta_1\mu_2\sigma_1\gamma_1} .
\end{aligned}
\end{equation}

\begin{remark}
    Condition \eqref{eqn: correlation bounds} guarantees that $w_1 >0 $ and $w_2 >0$.
\end{remark}

We also introduce the following terms:
\begin{equation}\label{eqn: N1, N2, N3}
        N_1 :=\frac{1}{2}\sigma_1^2\frac{w_0^2}{w_1^2}+\frac{1}{2}\sigma_2^2\frac{w_0^2}{w_2^2}+\rho\sigma_1\sigma_2\frac{w_0^2}{w_1w_2}, \qquad
        N_2 :=\mu_1\frac{w_0}{w_1}+\mu_2\frac{w_0}{w_2}, \qquad
        N_3 :=\frac{1}{2}\beta_1\sigma_1^2\frac{w_0^2}{w_1^2}+\frac{1}{2}\beta_2\sigma_2^2\frac{w_0^2}{w_2^2},
\end{equation}
where $w_0$ satisfies $w_0:=\inf\left\{x>0:\max\left\{\widehat \pi_1(x),\widehat\pi_2(x)\right\}=1\right\}.$
We can interpret $w_0$ as the threshold that signals the insurer to engage in zero reinsurance for at least one of the lines. We then have the following scenarios: (i) if $\widehat \pi_1(w_0)< 1= \widehat\pi_2(w_0)$, then Line 2 engages in zero reinsurance; (ii) if $\widehat \pi_2(w_0)< 1= \widehat\pi_1(w_0)$, then Line 1 engages in zero reinsurance; and (iii) if $\widehat \pi_1(w_0)=1= \widehat\pi_2(w_0)$, then both lines engage in zero reinsurance.  We also define $b^*$ by
\begin{equation}\label{eqn: defn of b*}
    b^*:=\inf \{u: g'(u)=a_2\}.
\end{equation}
We have not yet addressed the existence of $\gamma_1$, defined as a solution to $\psi(z)=0$. The following lemma gives a necessary and sufficient condition such that $\gamma_1\in(0,1)$ exists.
\begin{lemma}\label{lemma: gamma1 suff cond}
A solution to the equation $\psi(z)=0$, where $\psi$ is given by \eqref{eqn: defn of psi}, exists on $(0,1)$ if and only if 
\begin{equation}
   \delta < \frac{\mu_1^2}{2\beta_1\sigma_1^2} + \frac{\mu_2^2}{2\beta_2\sigma_2^2}. 
\end{equation}
\end{lemma}

\begin{remark}
    It should be noted that the solution of $\psi(z)=0$, if it exists on $(0,1)$, is not necessarily unique. This does not pose a problem in the analysis since the value function is also not necessarily unique.
    
    Furthermore, Lemma \ref{lemma: gamma1 suff cond} links the existence of a solution of $\psi(z) = 0$ in $(0,1)$ to the ambiguity-aversion parameters. As such, this implies that the optimal strategies are determined in part by the ambiguity-aversion parameters.
\end{remark}

\begin{theorem}\label{theorem 1}
    Suppose $\gamma_1\in(0,1)$ exists, \eqref{eqn: correlation bounds} holds, and $N_1\neq N_3$. 
    We have the following results:
    \begin{enumerate}[label=(\roman*)]
        \item $w_0=\min\{w_1,w_2\}$ and $b^*$ defined in \eqref{eqn: defn of b*} satisfies
        \begin{equation}\label{eqn: b^* theorem 1}
            b^*=\inf\{x>w_0 : v(x)^2+v'(x)=0  \},
        \end{equation}
        where
        \begin{equation}\label{eqn: v(x), K3, gamma2pm thm 1}
            \begin{aligned}
            v(x)&=\frac{N_1}{N_1-N_3}\cdot\frac{K_3\gamma_{2+}e^{\gamma_{2+}x}+\gamma_{2-}e^{\gamma_{2-}x}}{K_3e^{\gamma_{2+}x}+e^{\gamma_{2-}x}},\\
            K_3&=\frac{\left(1-\frac{N_3}{N_1}\right)\frac{\gamma_1}{w_0}-\gamma_{2-}}{\gamma_{2+}-\left(1-\frac{N_3}{N_1}\right)\frac{\gamma_1}{w_0}}e^{(\gamma_{2-}-\gamma_{2+})w_0},\\
            \gamma_{2\pm}&=\frac{-N_2\pm \sqrt{N_2^2+4\delta(N_1-N_3)}}{2N_1}.
            \end{aligned}
        \end{equation}
        \item The function $g$ defined by 
    \begin{equation}\label{eqn: g thm 1}
    g(x)=
    \begin{cases}
    \frac{a_2}{v(b^*)}\left[\frac{K_3e^{\gamma_{2+}w_0}+e^{\gamma_{2-}w_0}}{K_3e^{\gamma_{2+}b^*}+e^{\gamma_{2-}b^*}}\right]^{\frac{N_1}{N_1-N_3}}\left(\frac{x}{w_0}\right)^{\gamma_1} & \mbox{if $x<w_0$,}\\
    \frac{a_2}{v(b^*)} \left[\frac{K_3e^{\gamma_{2+}x}+e^{\gamma_{2-}x}}{K_3e^{\gamma_{2+}b^*}+e^{\gamma_{2-}b^*}}\right]^{\frac{N_1}{N_1-N_3}} & \mbox{if $w_0\leq x<b^*$,}\\
    a_2\left(x-b^*+\frac{1}{v(b^*)}\right) & \mbox{if $x\geq b^*$,}
    \end{cases}
\end{equation}
is a classical solution to the HJB equation in \eqref{eqn:hjborig} and thus is equal to the value function $V$ of the optimization problem in \eqref{eqn: value function}. Moreover, $g$ is increasing and concave as hypothesized. 
    \item The optimal reinsurance and distortion strategy is a feedback strategy of the form
    $(\pi_1^*,\pi_2^*,\theta_1^*,\theta_2^*)(t)= (\bar\pi_1,\bar\pi_2,\bar\theta_1,\bar\theta_2)(X^{u^*}(t))$ where
    \begin{equation}\label{eqn: pi-theta thm 1}
        (\bar\pi_1,\bar\pi_2,\bar\theta_1,\bar\theta_2)(x) =
        \begin{cases}
        \left(\frac{x}{w_1},\frac{x}{w_2},\frac{\beta_1\sigma_1\gamma_1}{w_1},\frac{\beta_2\sigma_2\gamma_1}{w_2}\right) & \mbox{if $x<w_0$,}\\
        \left(\frac{w_0}{w_1},\frac{w_0}{w_2},\frac{\beta_1\sigma_1w_0}{w_1}v(x),\frac{\beta_2\sigma_2w_0}{w_2}v(x)\right) & \mbox{if $w_0\leq x < b^*$,}\\
        \left(\frac{w_0}{w_1},\frac{w_0}{w_2},\frac{\beta_1\sigma_1w_0}{w_1(x-b^*+v(b^*)^{-1})},\frac{\beta_2\sigma_2w_0}{w_2(x-b^*+v(b^*)^{-1})}\right) & \mbox{if $x\geq  b^*$.}
    \end{cases}
    \end{equation}
    Moreover, the optimal dividend-capital injection strategy is $(C_{1,b^*},C_{2,b^*},L_{1,b^*},L_{2,b^*})$, which is a barrier strategy with a barrier $b^*$ defined in Definition \ref{definition of barrier}. 
    \end{enumerate}
\end{theorem}

\begin{remark}\label{remark on thm 1}
    We first highlight the significance of the assumption \eqref{eqn: correlation bounds}. Together with the definition of  $w_0$, it is guaranteed that $\bar\pi_1(x) \in [0,1]$ and $\bar\pi_2(x) \in[0,1]$. In particular, this ensures that the optimal reinsurance strategies are admissible.

    Next, we discuss the optimal reinsurance strategies $(\pi_1^*,\pi_2^*)$. If $w_0=w_1$, then Line 2 cedes a proportion $1-\frac{w_0}{w_2}$ of the risk to the reinsurer when the aggregate reserve level exceeds the reinsurance threshold level $x=w_0$. If $w_0=w_2$, Line 1 cedes a proportion $1-\frac{w_0}{w_1}$ of the risk to the reinsurer.

    For the optimal distortion controls $(\theta_1^*,\theta_2^*)$, we can see that the distortion controls are constant when both lines still cede their risk to the reinsurer (i.e., $x< w_0$). From Remark \ref{remark on g' and g'' thm 1}, $v(x)$ is decreasing, which implies that $\theta^*_1$ and $\theta^*_2$ are decreasing. This means that higher aggregate reserve levels imply that the optimal alternative model becomes closer to the reference model.

    Recall that the existence of $\gamma_1\in(0,1)$ is equivalent to $\delta < \frac{\mu_1^2}{2\beta_1\sigma_1^2} + \frac{\mu_2^2}{2\beta_2\sigma_2^2}$. The conditions $\delta < \frac{\mu_1^2}{2\beta_1\sigma_1^2} + \frac{\mu_2^2}{2\beta_2\sigma_2^2}$ and $N_1\neq N_3$ are analogous to the conditions of Case 2 in \cite{feng2021}. With the assumption that \eqref{eqn: correlation bounds} holds, a sufficient condition for $N_1\neq N_3$ is that both $\beta_1\leq 1$ and $\beta_2\leq 1$ must hold.
\end{remark}

\begin{theorem}\label{theorem 2}
    Suppose $\gamma_1\in(0,1)$ exists, \eqref{eqn: correlation bounds} holds, and $N_1 =  N_3$. 
    We have the following results:
    \begin{enumerate}[label=(\roman*)]
        \item $w_0=\min\{w_1,w_2\}$ and $b^*$ defined in \eqref{eqn: defn of b*} satisfies
        \begin{equation}\label{eqn: b^* theorem 2}
            b^*=\inf\{x>w_0 : v(x)^2+v'(x)=0  \},
        \end{equation}
        where
        \begin{equation}\label{eqn: v(x) theorem 2}
            \begin{aligned}
            v(x)&=\frac{\delta}{N_2}\left[1-e^{-\frac{N_2}{N_1}(x-w_0)}\right]+\frac{\gamma_1}{w_0}e^{-\frac{N_2}{N_1}(x-w_0)}.
            \end{aligned}
        \end{equation}
        \item The function $g$, defined by 
    \begin{equation}\label{eqn: g thm 2}
    g(x)=
    \begin{cases}
    \frac{a_2}{v(b^*)e^{v_1(b^*)}}\left(\frac{x}{w_0}\right)^{\gamma_1} & \mbox{if $x<w_0$,}\\
    \frac{a_2}{v(b^*)} e^{v_1(x)-v_1(b^*)} & \mbox{if $w_0\leq x<b^*$,}\\
    a_2\left(x-b^*+\frac{1}{v(b^*)}\right) & \mbox{if $x\geq b^*$,}
    \end{cases}
    \end{equation}
    where 
    \begin{equation}
        v_1(y)=\frac{\delta}{N_2}(y-w_0)-\frac{N_1}{N_2}\left(\frac{\gamma_1}{w_0}-\frac{\delta}{N_2}\right)\left(e^{-\frac{N_2}{N_1}(y-w_0)}-1\right)
    \end{equation}
    is a classical solution to the HJB equation in \eqref{eqn:hjborig} and thus is equal to the value function $V$ of the optimization problem in \eqref{eqn: value function}. Moreover, $g$ is increasing and concave as hypothesized. 
    \item The optimal reinsurance and distortion strategy is a feedback strategy of the form $(\pi_1^*,\pi_2^*,\theta_1^*,\theta_2^*)(t)=(\bar\pi_1,\bar\pi_2,\bar\theta_1,\bar\theta_2)(X^{u^*}(t))$, where $(\bar\pi_1,\bar\pi_2,\bar\theta_1,\bar\theta_2)(x)$ is of the same form as \eqref{eqn: pi-theta thm 1}.
    Moreover, the optimal dividend-capital injection strategy is $(C_{1,b^*},C_{2,b^*},L_{1,b^*},L_{2,b^*})$, which is a barrier strategy with a barrier $b^*$ defined in Definition \ref{definition of barrier}. 
    \end{enumerate}
\end{theorem}

\begin{remark}
    We observe that the optimal reinsurance and distortion strategies between Theorems \ref{theorem 1} and \ref{theorem 2} are similar, and thus the explanations in Remark \ref{remark on thm 1} also apply to this case, except that Theorem \ref{theorem 2}'s conditions  are analogous to the conditions of Case 1 in \cite{feng2021}.
\end{remark}

\begin{theorem}\label{theorem 3}
    Suppose $\psi(z)=0$ does not have a solution on $(0,1)$. 
    We have the following results:
    \begin{enumerate}[label=(\roman*)]
        \item $w_0$ does not exist and $b^*=0$.
        \item The function $g$, defined by $g(x)= a_2 x$, $x\geq 0$
    is a classical solution to the HJB equation in \eqref{eqn:hjborig} and thus is equal to the value function $V$ of the optimization problem in \eqref{eqn: value function}.
    \item The optimal dividend-capital injection strategy is $(C_{1,0},C_{2,0},L_{1,0},L_{2,0})$, which is a barrier strategy with a barrier $0$ defined in Definition \ref{definition of barrier}; that is, Line 2 pays all of the aggregate reserve as dividends at time $t=0$ and ruin occurs immediately. Consequently, the optimal reinsurance and distortion strategies become irrelevant. 
    \end{enumerate}
\end{theorem}

\begin{remark}
    The condition of Theorem \ref{theorem 3} is satisfied if any of the following conditions hold: (i) the discount rate $\delta$ is large enough; (ii) the adjusted means $\mu_1$ and $\mu_2$ are small enough; (iii) the volatilities $\sigma_1$ and $\sigma_2$ are large enough; or (iv) the (adjusted) robustness preference parameters $\beta_1$ and $\beta_2$ are large enough. All of these conditions imply that the insurer must immediately pay all of the reserves from both lines as dividends at time $t=0$ and ruin occurs. 
    
    When the discount rate $\delta$ is sufficiently large, the present value of future dividend payments becomes smaller, making immediate dividend payout more advantageous. If the adjusted means $\mu_i$ are low enough, the reserve process has smaller growth, justifying the need for immediate dividend payout. Moreover, if the volatilities $\sigma_i$ are high enough, the likelihood of ruin increases, which supports the immediate distribution of dividends. Lastly, if the (adjusted) robustness preference parameters (or ambiguity-aversion parameters) $\beta_i$ are large enough, it indicates a lack of trust in the reference model, signaling the insurer to pay dividends immediately.
\end{remark}   

\subsection{Remaining cases}\label{subsec: remaining case}

In this section, we discuss the remaining cases that have not been covered in Section \ref{subsec: main case}. We state the main results directly since the techniques and analysis are similar to those in the main case. Recall that we consider \eqref{eqn: correlation bounds} in Section \ref{subsec: main case}.

First, we now consider the following condition:
\begin{equation}\label{eqn: other case 2}
    \frac{\rho}{1+\beta_2\frac{\gamma_1}{1-\gamma_1}}>\frac{\mu_1 /\sigma_1}{\mu_2/\sigma_2}>0.
\end{equation}
From \eqref{eqn: optimal pi} and \eqref{eqn: other case 2}, we have $\widehat \pi _1(x) \leq 0$ and $\widehat \pi_2(x) \geq  0$ for $x\geq 0$. As such, with the constraint $\pi_i\in[0,1]$, it must hold that $\bar\pi_1(x)=0$ for all $x\geq 0$, implying that Line 1 cedes all of its risk to the reinsurer. 

\begin{proposition}\label{prop case 2}
    Define $w_0$ such that it satisfies $\widehat \pi_2(w_0)=1$ and define
    \begin{align}
    \begin{split}
        \label{eqn: prop case 2 dfn}
        N_1&:=\frac{1}{2}\sigma_2^2, \quad N_2:=\mu_2, \quad N_3:=\frac{1}{2}\beta_2\sigma_2^2, \\
        \widetilde A(z)
        &:=[\rho^2\mu_1^2\sigma_2^2+\mu_2^2\sigma_1^2(2\rho^2-1) - 2\rho^3\mu_1\mu_2\sigma_1\sigma_2]z(z-1)^3 - \mu_2^2\sigma_1^2(\beta_1^2+2\beta_1\beta_2)z^3(z-1)\\
        &\qquad + [\mu_2^2\sigma_1^2(2\beta_1+\beta_2-2\rho^2\beta_1)-\rho^2\beta_2\mu_1^2\sigma_2^2]z^2(z-1)^2 + \beta_1^2\beta_2\mu_2^2\sigma_1^2z^4.
    \end{split}
    \end{align}
    Furthermore, let $\gamma_1$ be the solution of the equation $\psi_1(z):=\widetilde{A}(z) - 2 \delta \sigma_1^2 \sigma_2^2 \widetilde{B}(z) = 0$ on $(0,1)$, where $\widetilde{A}(z)$ is given by \eqref{eqn: prop case 2 dfn} and $\widetilde{B}(z)$ is given by \eqref{eqn: defn of psi}. 
    We have the following results:
    \begin{enumerate}[label=(\roman*)]
        \item Suppose $\gamma_1\in(0,1)$ exists, \eqref{eqn: other case 2} holds, and $\beta_2 \neq 1$. Then, $w_0$ equals $w_2$ defined in \eqref{eqn: w1 and w2}, $b^*$ satisfies Theorem \ref{theorem 1}(i), and the value function equals $g$ defined in \eqref{eqn: g thm 1}. Moreover, the optimal reinsurance and distortion strategy is a feedback strategy given by $(\pi_1^*,\pi_2^*,\theta_1^*,\theta_2^*)(t)=(\bar\pi_1,\bar\pi_2,\bar\theta_1,\bar\theta_2)(X^{u^*}(t))$, where
        \begin{equation}\label{eqn: opt rein case 2}
            (\bar\pi_1,\bar\pi_2,\bar\theta_1,\bar\theta_2)(x)
            =
            \begin{cases}
                \left(0,\frac{x}{w_2},0, \frac{\beta_2\sigma_2\gamma_1}{w_2}\right) & \mbox{if $x<w_0$,}\\ 
                \left(0,1,0, \beta_2\sigma_2\gamma_1v(x)\right) & \mbox{if $w_0\leq x< b^*$,}\\
                \left(0,1,0, \frac{\beta_2\sigma_2}{x-b^*+v(b^*)^{-1}}\right) & \mbox{if $x \geq  b^*$,}
            \end{cases}
        \end{equation}
        where $v(x)$ is defined in \eqref{eqn: v(x), K3, gamma2pm thm 1}, while the optimal dividend-capital injection strategy is a barrier strategy with a barrier $b^*$.
        
        \item Suppose $\gamma_1\in(0,1)$ exists, \eqref{eqn: other case 2} holds, and $\beta_2=1$. Then, $w_0$ equals $w_2$ defined in \eqref{eqn: w1 and w2}, $b^*$ satisfies Theorem \ref{theorem 2}(i), and the value function equals $g$ defined in \eqref{eqn: g thm 2}. Moreover, the optimal reinsurance and distortion strategy is a feedback strategy given by $(\pi_1^*,\pi_2^*,\theta_1^*,\theta_2^*)(t)=(\bar\pi_1,\bar\pi_2,\bar\theta_1,\bar\theta_2)(X^{u^*}(t))$, where $(\bar\pi_1,\bar\pi_2,\bar\theta_1,\bar\theta_2)(x)$ is given by \eqref{eqn: opt rein case 2} with $v(x)$ defined in \eqref{eqn: v(x) theorem 2}, while the optimal dividend-capital injection strategy is a barrier strategy with a barrier $b^*$. 
        
        \item Suppose $\psi_1(z)=0$ does not have a solution on $(0,1)$.  Then, the value function equals $g(x)=a_2x$ and the optimal dividend-capital injection strategy is a barrier strategy with a barrier $b^*=0$. The optimal reinsurance and distortion strategies become irrelevant.
    \end{enumerate}
\end{proposition}

\begin{remark}
    Since Line 1 cedes all of its risk to the reinsurer, the insurer focuses on managing the risk of Line 2. The problem is then reduced to just one line. The optimal strategies are similar to those in \cite{feng2021}.
\end{remark}

Second, we consider the following condition:
\begin{equation}
    \label{eqn: other case 3}
    \frac{\mu_1 /\sigma_1}{\mu_2/\sigma_2}> \frac{1+\beta_1\frac{\gamma_1}{1-\gamma_1}}{\rho}>0.
\end{equation}
In this case, $\widehat \pi_2(x) \leq 0$ and $\widehat \pi_1(x) \geq 0$ for all $x\geq 0$. Thus, it follows that $\pi_2^*\equiv 0$. Moreover, $w_0$ satisfies $\widehat \pi_1(w_0)=1$. This case is symmetric to Proposition \ref{prop case 2}. Now, Line 2 cedes all of its risk to the reinsurer. The optimal strategies are obtained from Proposition \ref{prop case 2} by exchanging the indices $1$ and $2$, i.e., $(\mu_1,\sigma_1,\beta_1) \leftrightarrow (\mu_2,\sigma_2,\beta_2)$ and $(\pi_1,\theta_1)\leftrightarrow(\pi_2,\theta_2)$.

Last, we consider the case where $-1<\rho \leq 0$. In this case, it always holds that $\widehat \pi_1(x) >0$ and $\widehat \pi_2 (x) >0$ for all $x>0$. Thus, the analysis follows from the one in Section \ref{subsec: main case} using the same values for $N_1$, $N_2$, and $N_3$ defined in \eqref{eqn: N1, N2, N3}, and the same function $\widetilde A(z)$ defined in \eqref{eqn: defn of psi}.

\begin{remark}
    The ratio $\mu_i/\sigma_i$ can be interpreted as a trade-off between the adjusted mean of the risk exposure and the volatility of the risk exposure. Hence, we can interpret $\frac{\mu_1/\sigma_1}{\mu_2/\sigma_2}$ as the relative Sharpe ratio of Line 1 over Line 2. A small value of this ratio indicates a more favorable trade-off for Line 2, which aligns with the case where \eqref{eqn: other case 2} holds. On the other hand, a large value of the ratio implies a more advantageous trade-off for Line 1. However, when $\rho \leq 0$, this ratio becomes less relevant since having a non-positive correlation provides a hedging effect or diversification benefits.
\end{remark}

\section{The Case of No Model Uncertainty}\label{sec:no mod.unc.}

In the case where $\widetilde \beta_1=\widetilde\beta_2=0$ (or, equivalently, $\beta_1=\beta_2=0$), the insurer is convinced that the ``true" model is the reference model $\mathbb P$ and any deviation from $\mathbb P$ incurs an infinite penalty due to the second term in \eqref{eqn: value function}. Hence, $\mathbb Q^{\theta^*}$ should be chosen as $\mathbb P$, which yields $\theta_1^*(t)=\theta_2^*(t)\equiv 0$ to guarantee that the second term in \eqref{eqn: value function} vanishes. The optimization problem then degenerates to the problem with no model uncertainty given by:
        \begin{equation}\label{eqn: V no model uncertainty}
        V(x_1,x_2)=\sup_{u\in\mathcal U} \mathbb E^{\mathbb P}\left[\sum_{i=1}^2a_i\left(\int_0^{\tau^u}e^{-\delta t}\dd C_i(t)\right)\right].
    \end{equation}
    We call the insurer in this case ambiguity neutral.

    If the insurer is ambiguity neutral, the HJB equation \eqref{eqn:hjborig} is reduced to
    \begin{equation}\label{eqn: hjb no model unc}
        \sup\Bigg\{\sup_{\pi_i\in[0,1]}\Lc^{\vartheta}(g)(x), a_1-g'(x), a_2-g'(x)\Bigg\} = 0,
    \end{equation}
    with boundary condition $g(0)=0$. Here, $\vartheta=(\pi_1,\pi_2,0,0)$.

    We now present the results for the case where the insurer is ambiguity neutral. This serves as a ``benchmark" for comparing the results where model uncertainty is present. By a standard verification lemma \citep[see, for instance,][Section 3.3]{hojgaard1999}, a classical solution to the HJB equation \eqref{eqn: hjb no model unc} is the solution to the value function \eqref{eqn: V no model uncertainty}.

\begin{theorem}\label{theorem no uncertainty}
    Suppose $\beta_1=\beta_2=0$ and the correlation $\rho$ satisfies either $0 < \rho \leq \frac{\mu_1 / \sigma_1}{\mu_2 / \sigma_2} \leq \frac{1}{\rho}$ or $-1 < \rho \leq 0$. The following results then hold:
    \begin{enumerate}[label=(\roman*)]
        \item $w_0=\min\{w_1,w_2\}$ and $b^*$ defined in \eqref{eqn: defn of b*} is given by
        \begin{equation}\label{eqn: b^* no mod.unc.}
            b^*=w_0+\frac{1}{\gamma_{2+}-\gamma_{2-}}\ln\left(-\frac{\gamma_{2-}}{\gamma_{2+}}\right),
        \end{equation}
        where $\gamma_{2\pm}$ is defined in \eqref{eqn: v(x), K3, gamma2pm thm 1} and $\gamma_1$ is given by
            \begin{equation}
                \label{eqn: gamma1 no mod.unc.}
                \gamma_1 = \frac{2\delta \sigma_1^2\sigma_2^2(1-\rho^2)}{(\mu_1\sigma_2-\mu_2\sigma_1)^2+2(1-\rho)\mu_1\mu_2\sigma_1\sigma_2+2\delta\sigma_1^2\sigma_2^2(1-\rho^2)}.
            \end{equation}
        \item The function $g$ defined by
        \begin{align}\label{eqn: g no mod.unc.}
            g(x)=
            \begin{cases}
            \frac{2\lambda(1-\gamma_1)}{w_0}\left(\frac{x}{w_0}\right)^{\gamma_1} & \mbox{if $x<w_0$,}\\
            -\lambda\left[\gamma_{2-}e^{\gamma_{2+}(x-w_0)}+\gamma_{2+}e^{\gamma_{2-}(x-w_0)}\right]  & \mbox{if $w_0\leq x< b^*$,}\\
            a_2\left(x-b^*+\frac{N_2}{\delta}\right)  & \mbox{if $x\geq b^*$,}
            \end{cases}
        \end{align}
        where
        \begin{equation}\label{eqn: lambda no mod.unc.}
            \lambda =-\frac{a_2}{\gamma_{2+}\gamma_{2-}}\left[e^{\gamma_{2+}(b^*-w_0)}+e^{\gamma_{2-}(b^*-w_0)}\right]^{-1}
        \end{equation}
        is a classical solution to the HJB equation in \eqref{eqn: hjb no model unc} and thus is equal to the value function $V$ of the optimization problem in \eqref{eqn: V no model uncertainty}. Moreover, $g$ is increasing and concave.
        \item The optimal reinsurance strategy is a feedback strategy given by $(\pi_1^*,\pi_2^*)(t)=(\bar\pi_1,\bar\pi_2)(X^{u^*}(t))$, where
        \begin{equation}\label{eqn: pi no mod.unc.}
                (\bar\pi_1,\bar\pi_2)(x) =
                \begin{cases}
                    \left(\frac{x}{w_1},\frac{x}{w_2}\right) & \mbox{if $x<w_0$,}\\
                    \left(\frac{w_0}{w_1},\frac{w_0}{w_2}\right) & \mbox{if $x \geq w_0$.}
                \end{cases}
        \end{equation}
        Moreover, the optimal dividend-capital injection strategy is $(C_{1,b^*},C_{2,b^*},L_{1,b^*},L_{2,b^*})$, which is a barrier strategy with a barrier $b^*$ defined in Definition \ref{definition of barrier}. 
    \end{enumerate}
\end{theorem}

\begin{remark}
    The analysis for the other cases involving the correlation coefficient follows similarly as those in Section \ref{subsec: remaining case}. Moreover, the results presented in Theorem \ref{theorem no uncertainty} are analogous to the results in the univariate framework presented in \cite{hojgaard1999}.
\end{remark}

\section{Numerical Illustrations}\label{sec:numerical}


In this section, we provide numerical illustrations of the analytical solutions established above. The examples below focus on the effect of model uncertainty (via the ambiguity-aversion parameters) on the insurer's optimal strategies and how model uncertainty interacts with the riskiness of each business line's reserve process (measured by $\sigma_1$ and $\sigma_2$) in influencing the optimal strategies. Unless otherwise stated, all numerical illustrations are generated using the parameter values specified in Table \ref{tab:param}. 

\begin{table}[hbt]
    \caption{Base parameter values used for all numerical illustrations.}
    \label{tab:param}
    \centering
    \begin{tabular}{@{}cccccc@{}}
    \toprule
    $\mu_1$ & 4.00 & $\mu_2$ & 2.00 & $\rho$ & 0.60 \\
    $\sigma_1$ & 1.50 & $\sigma_2$ & 1.00 & $\delta$ & 0.50 \\
    $a_2$ & 0.70 & $\beta_1$ & 1.00 & $\beta_2$ & 1.00 \\ \bottomrule
    \end{tabular}
\end{table}

Figures \ref{fig:Th31 illustration} and \ref{fig:Th33 illustration} show the value function $g(x)$, optimal reinsurance strategies $\bar{\pi}_1(x)$ and $\bar{\pi}_2(x)$, and the optimal distortion strategies $\bar{\theta}_1(x)$ and $\bar{\theta}_2(x)$ as a function of the (total) reserve $x = X^{u^*}(t)$ under Theorems \ref{theorem 1} and \ref{theorem 2}, respectively. As expected, the value function in both cases is strictly increasing and concave with respect to the value of the underlying reserve process. To ensure the condition $N_1 = N_3$ in Theorem \ref{theorem 2} is satisfied, we set $\rho = 0$ given the parameters specified in Table \ref{tab:param}. In this case, the results in Figures \ref{fig:Th31 illustration} and \ref{fig:Th33 illustration} also illustrate the effect of a nonzero correlation when all other parameters are fixed.  In addition to the optimal distortion strategies, we also show the rate of change of the relative entropy of the worst-case probability measure $\mathbb{Q}^{\theta^*}$ with respect to $\mathbb{P}$.

In both figures, the insurer tends to linearly increase the proportion of risk retained for both lines as the underlying reserve process value increases. In the illustration of Theorem \ref{theorem 1} in Figure \ref{fig:Th31 illustration}, once the reserve process reaches a value of $w_0$ (from below), the insurer will retain all the risks associated with Line 1 and will cease ceding additional risk from Line 2 to the reinsurer. However, under the conditions of Theorem \ref{theorem 2} illustrated in Figure \ref{fig:Th33 illustration}, the threshold $w_0$ is smaller, indicating that the insurer will completely absorb the risk at a lower reserve level. In contrast to Figure \ref{fig:Th31 illustration}, the insurer in Figure \ref{fig:Th33 illustration} will retain all the risks associated with Line 2. Furthermore, the insurer will absorb a greater proportion of the reinsured line in Figure \ref{fig:Th33 illustration} (Line 1) compared to Figure \ref{fig:Th31 illustration} (Line 2). The value of $b^*$ is also smaller in Figure \ref{fig:Th33 illustration} compared to Figure \ref{fig:Th31 illustration}, indicating that the insurer will decide to pay out dividends and transfer capital from one line to another for a smaller value of the underlying reserve process.

\begin{figure}[hbt]
    \centering
    \begin{subfigure}[H]{0.30\textwidth}
        \centering
        \includegraphics[width=\linewidth]{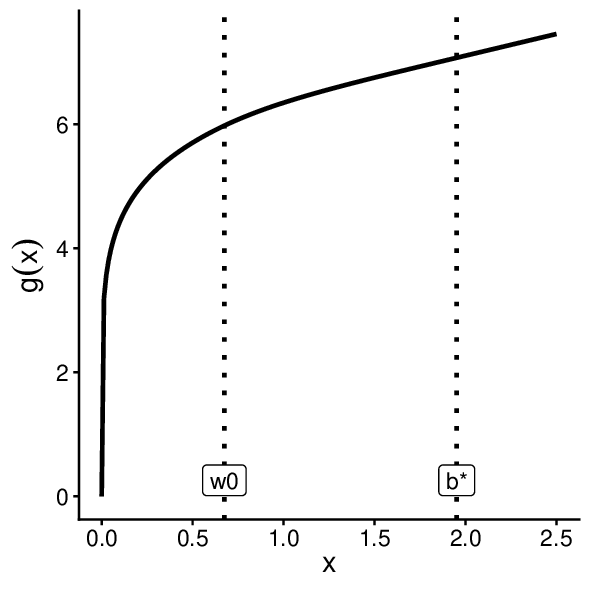} 
    \end{subfigure}
    \quad
    \begin{subfigure}[H]{0.30\textwidth}
        \centering
        \includegraphics[width=\linewidth]{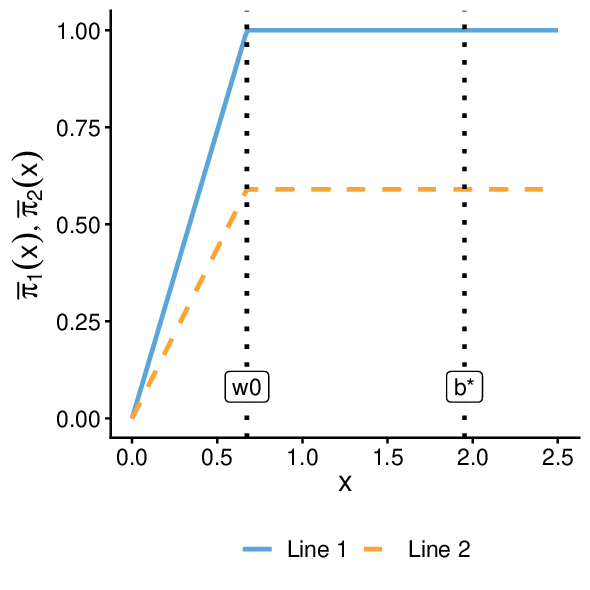}
    \end{subfigure}
    \quad
    \begin{subfigure}[H]{0.30\textwidth}
        \centering
        \includegraphics[width=\linewidth]{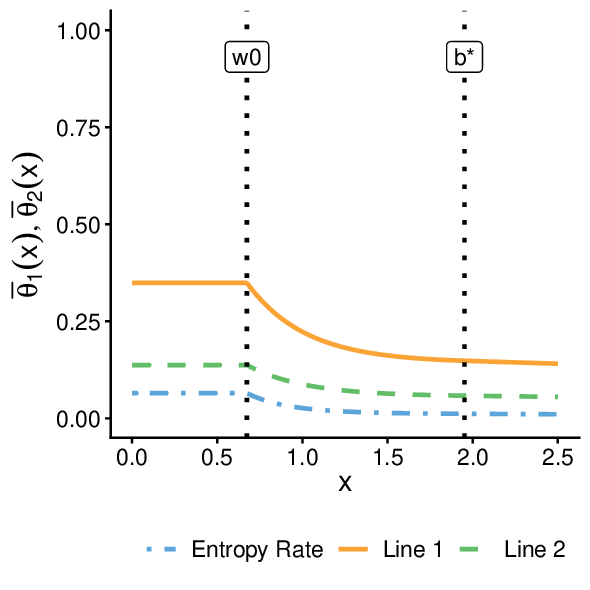}
    \end{subfigure}
    \\[-2ex]
    \caption{Value function $g(x)$, optimal reinsurance strategies $\bar{\pi}_1(x), \bar{\pi}_2(x)$, and optimal distortion strategies $\bar{\theta}_1(x), \bar{\theta}_2(x)$ as a function of the (total) reserve $x$ when $\rho = 0.60$. This parameter configuration, corresponding to Theorem \ref{theorem 1}, yields $w_0 = 0.67 < b^* = 1.95$.}
    \label{fig:Th31 illustration}
\end{figure}

\begin{figure}[hbt]
    \centering
    \begin{subfigure}[H]{0.30\textwidth}
        \centering
        \includegraphics[width=\linewidth]{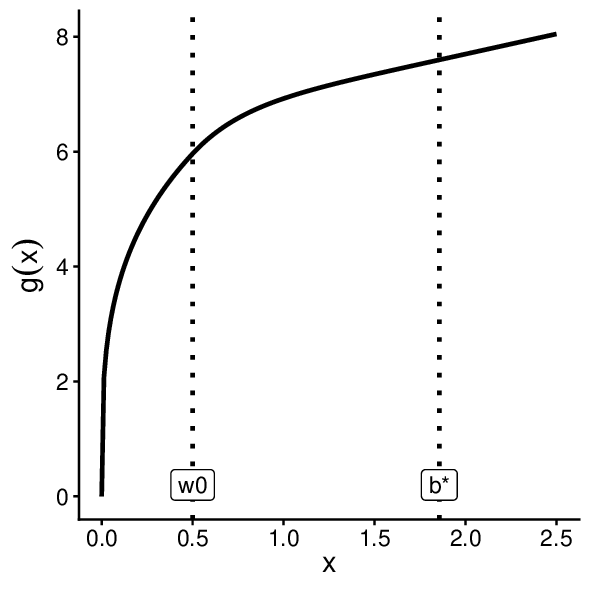} 
    \end{subfigure}
    \quad
    \begin{subfigure}[H]{0.30\textwidth}
        \centering
        \includegraphics[width=\linewidth]{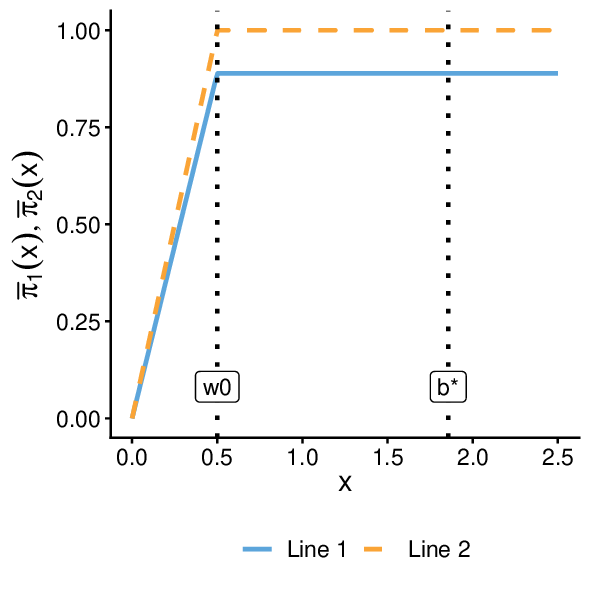}
    \end{subfigure}
    \quad
    \begin{subfigure}[H]{0.30\textwidth}
        \centering
        \includegraphics[width=\linewidth]{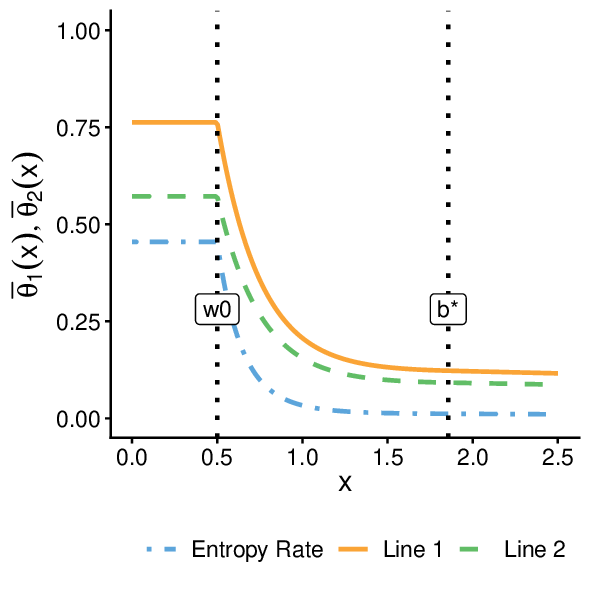}
    \end{subfigure}
    \\[-2ex]
    \caption{Value function $g(x)$, optimal reinsurance strategies $\bar{\pi}_1(x), \bar{\pi}_2(x)$, and optimal distortion strategies $\bar{\theta}_1(x), \bar{\theta}_2(x)$ as a function of the (total) reserve $x$ when $\rho = 0$. This parameter configuration, corresponding to Theorem \ref{theorem 2}, yields $w_0 = 0.50 < b^* = 1.86$.}
    \label{fig:Th33 illustration}
\end{figure}

The optimal distortion processes $\theta^*_1(t) = \bar{\theta}_1(X^{u^*}(t))$ and $\theta^*_2(t) = \bar{\theta}_2(X^{u^*}(t))$ determine how far the worst-case probability measure $\mathbb{Q}^{\theta^*}$ is from the reference probability measure $\mathbb{P}$. Since the distance, measured by the relative entropy of $\mathbb{Q}^{\theta^*}$ with respect to $\mathbb{P}$, is obtained by integrating the optimal distortion processes over time, we instead investigate the rate of change of the relative entropy, which
is given by \[\frac{\dd}{\dd t} \mathbb{E}^{\mathbb{Q}^{\theta^*}} \bigg[\log\bigg(\frac{\dd \mathbb Q^{\theta^*}}{\dd\mathbb P}\bigg|_{\mathcal F_t} \bigg)\bigg] = \frac{1}{2} \bigg[(\bar{\theta}_1(X^{u^*}(t))^2 + \frac{(\bar{\theta}_2(X^{u^*}(t)) - \rho \bar{\theta}_1(X^{u^*}(t)))^2}{1 - \rho^2} \bigg].\] The rate of change allows us to examine how fast the worst-case model $\mathbb{Q}^{\theta^*}$ is moving away from the reference model $\mathbb{P}$ with respect to the current value of the underlying reserve process.


In both Figures \ref{fig:Th31 illustration} and \ref{fig:Th33 illustration}, the rate of change is constant as long as the reserve process is less than $w_0$. Once the reserve process exceeds $w_0$ and the insurer has absorbed all the risk from at least one line, the rate of change of the relative entropy monotonically decreases as the reserve process increases. This indicates that the relative entropy becomes less sensitive to further increases in the underlying reserve value. Under the conditions of Theorem \ref{theorem 2}, the optimal distortion processes and the rate of change of relative entropy are greater than their corresponding values under the conditions of Theorem \ref{theorem 1}, indicating that a strictly positive $\rho$ can result in a worst-case model that moves faster away from the reference model (assuming all other model parameters are the same). However, these results do not quantify the actual relative entropy, or distance, between the worst‑case model and the reference model without knowledge of the trajectory of the reserve process $\{X^{u^*}(t)\}_{t \geq 0}$ under the optimal strategies.

\begin{figure}[hbt]
    \centering
    \begin{subfigure}[H]{0.35\textwidth}
        \centering
        \includegraphics[width=\linewidth]{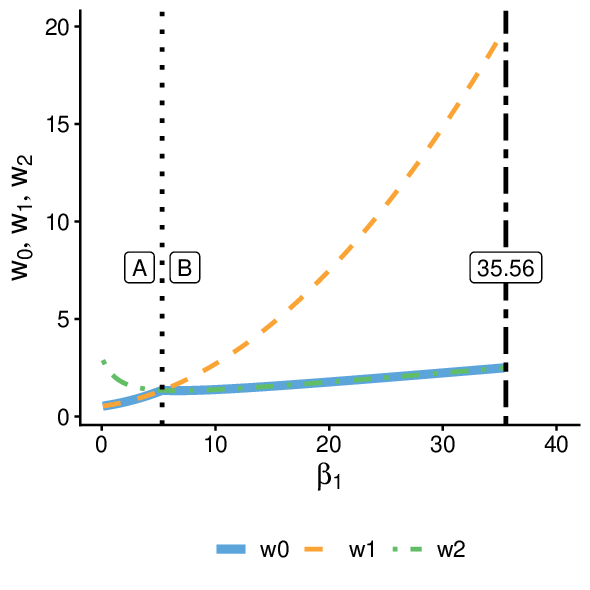} 
        \caption{$\beta_2 = 5$ fixed}
        \label{fig:beta2 fixed}
    \end{subfigure}
    \qquad
    \begin{subfigure}[H]{0.35\textwidth}
        \centering
        \includegraphics[width=\linewidth]{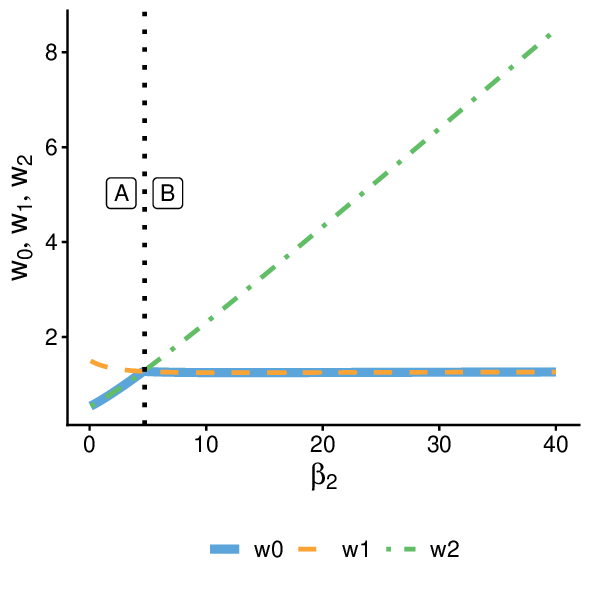}
        \caption{$\beta_1 = 5$ fixed}
        \label{fig:beta1 fixed}
    \end{subfigure}
    \caption{Reinsurance threshold $w_0$ and values of $w_1$ and $w_2$ as a function of $\beta_i$ when $\beta_{3-i} = 5$ is fixed. In Figure \ref{fig:beta2 fixed} (resp. Figure \ref{fig:beta1 fixed}), the region A to the left of the dotted graph represents the values of $\beta_1$ (resp. $\beta_2$) for which $w_2 > w_1$ (resp. $w_1 > w_2$). The region B to the right of the dotted graph represents values of $\beta_1$ (resp. $\beta_2$) for which $w_1 > w_2$ (resp. $w_2 > w_1$). When $\beta_1$ is variable, the dot-dashed graph indicates the upper bound on $\beta_1$ up to which $w_0$ exists.}
    \label{fig:w sensitivity}
\end{figure}

Figure \ref{fig:w sensitivity} illustrates the zero-reinsurance threshold $w_0$ and the values of $w_1$ and $w_2$ as a function of the (scaled) robustness preference/ambiguity-aversion parameter $\beta_i$ of Line $i$ when $\beta_{3-i} = 5$ is fixed, for $i=1,2$. Note that the value at which $\beta_i$ is fixed is chosen to illustrate the switching of $w_0$ between $w_1$ and $w_2$ as $\beta_{3-i}$ varies. In Figure \ref{fig:beta2 fixed} (resp. Figure \ref{fig:beta1 fixed}), the dotted graph represents the point at which $w_0$ switches from the value of $w_1$ to $w_2$ (resp. $w_2$ to $w_1$). In Figure \ref{fig:beta2 fixed}, when Line 1 exhibits a relatively low ambiguity aversion, we have $w_0 = w_1$. This means that when the current reserve value has exceeded the threshold $w_0$, Line 1 retains its entire risk, whereas Line 2 cedes a proportion $1 - \frac{w_0}{w_2}$ of its risk to the reinsurer.
However, when Line 1 has a high degree of ambiguity aversion, we have $w_0 = w_2$. This means that if the current reserve exceeds the threshold $w_0$, Line 2 retains its entire risk, while Line 1 cedes a proportion $1 - \frac{w_0}{w_1}$ of the line's risk. In other words, Line 1 prefers to absorb some of its risk when the degree of trust in the reference model is low. The same interpretation applies to the results shown in Figure \ref{fig:beta1 fixed}. 


Similarly, Figure \ref{fig:bstar sensitivity} illustrates the impact of model uncertainty (via the ambiguity-aversion parameters) on the optimal barrier $b^*$ for the dividend payout and capital injection strategies. In Figure \ref{fig:bstar}, as seen also in Figure \ref{fig:beta2 fixed}, there exists a nonzero $b^*$ only up to a certain value of $\beta_1$; given the parameter values assumed, the upper bound on $\beta_1$ is approximately 35.56. The behavior of $b^*$ as either $\beta_1$ or $\beta_2$ increases is notably non-monotonic (note that this behavior is also apparent in Figure \ref{fig:contour bstar} below). With $\beta_2 = 5$ fixed, the barrier increases as $\beta_1$ approaches $\beta_2$ (from below); that is, if the two lines have approximately the same degree of ambiguity aversion, the insurer is willing to pay dividends past a higher level of reserves. However, if there is a large discrepancy between the ambiguity aversion of the two lines, the insurer is likely to pay dividends at a lower reserve level. In the more extreme case of misalignment when $\beta_1$ exceeds a value of 35.56, it will be optimal for the insurer to pay out dividends immediately. As mentioned above, given the assumed parameter values, a nonzero value of $b^*$ exists for any $\beta_2$ and so it is optimal for the insurer to pay dividends only when the reserve exceeds $b^*$. We also note in Figure \ref{fig:ratio} that as $\beta_1$ approaches the upper bound, the values of $b^*$ and $w_0$ approach one another, since the ratio $w_0 / b^*$ approaches 1. This means that if Line 1 becomes more ambiguity averse (up to the computed upper bound), it is optimal for the insurer to engage in zero reinsurance for at least one line and to pay dividends when the underlying reserve value exceeds the same quantity.

\begin{figure}[hbt]
    \centering
    \begin{subfigure}[H]{0.35\textwidth}
        \centering
        \includegraphics[width=\linewidth]{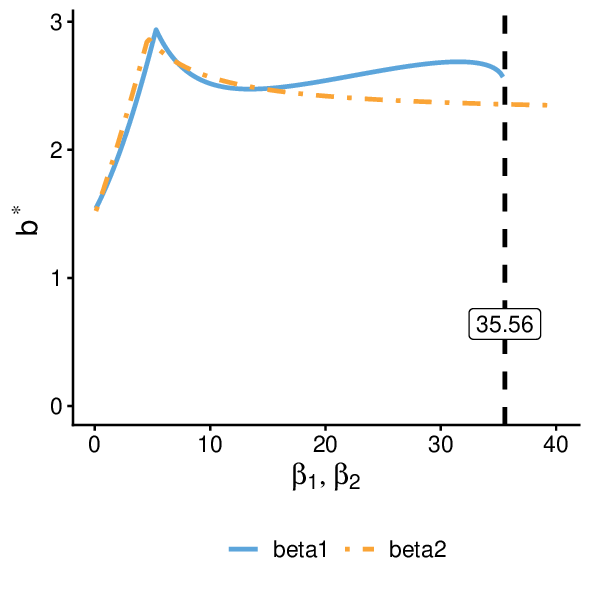} 
        \caption{$b^*$}
        \label{fig:bstar}
    \end{subfigure}
    \qquad
    \begin{subfigure}[H]{0.35\textwidth}
        \centering
        \includegraphics[width=\linewidth]{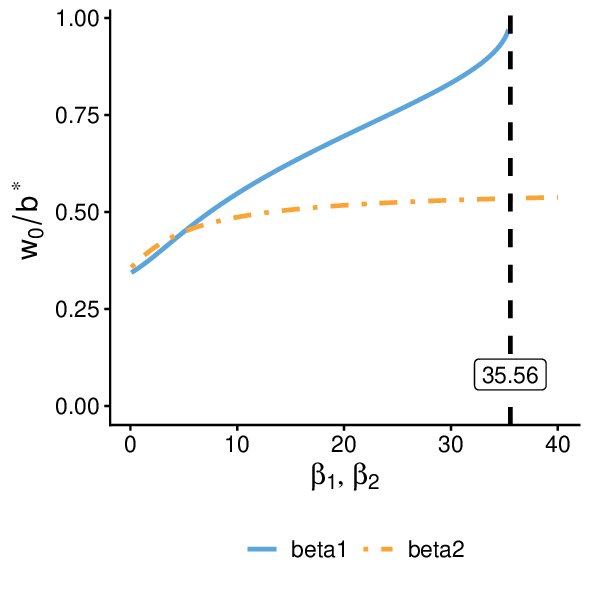}
        \caption{$w_0 / b^*$}
        \label{fig:ratio}
    \end{subfigure}
    \caption{The barrier $b^*$ corresponding to the optimal dividend-capital injection strategy and the ratio $w_0 / b^*$ as a function of $\beta_i$ when $\beta_{3-i} = 5$ is fixed, for $i=1,2$. The legend indicates which parameter is variable. When $\beta_1$ is variable, the dot-dashed graph indicates the upper bound on $\beta_1$ up to which either $w_0$ exists or $b^*$ is nonzero.}
    \label{fig:bstar sensitivity}
\end{figure}




\begin{figure}[htb]
    \centering
        \begin{subfigure}[H]{0.35\textwidth}
            \centering
            \includegraphics[width=\linewidth]{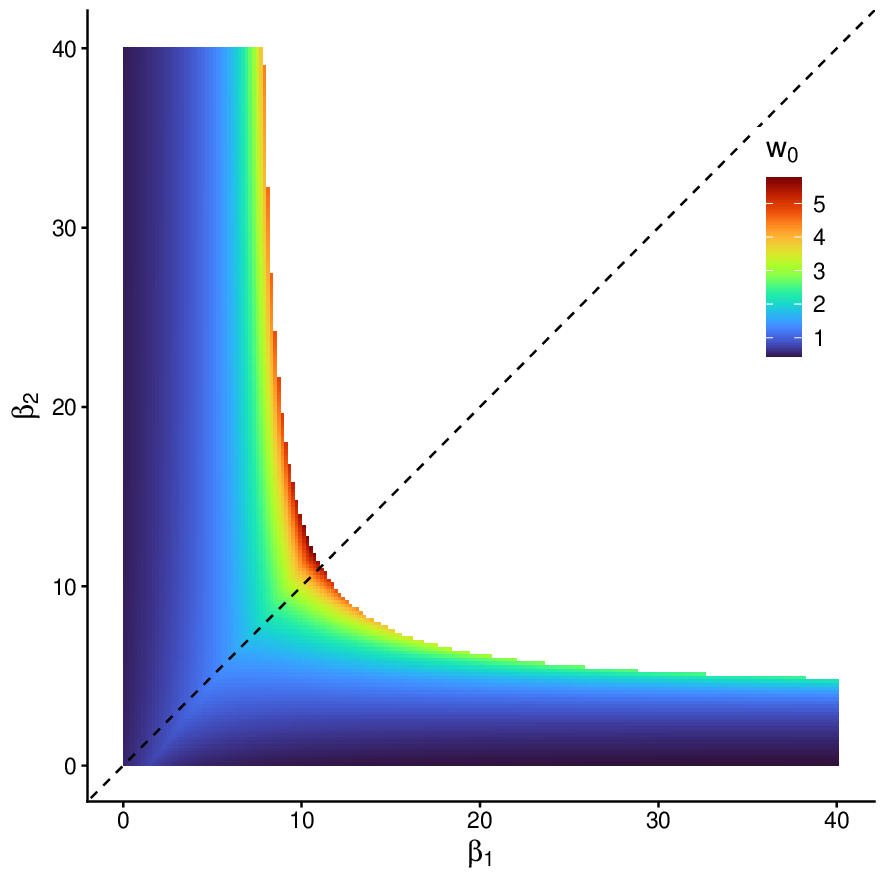} 
            \caption{$w_0$}
        \end{subfigure}
        \qquad
        \begin{subfigure}[H]{0.35\textwidth}
            \centering
            \includegraphics[width=\linewidth]{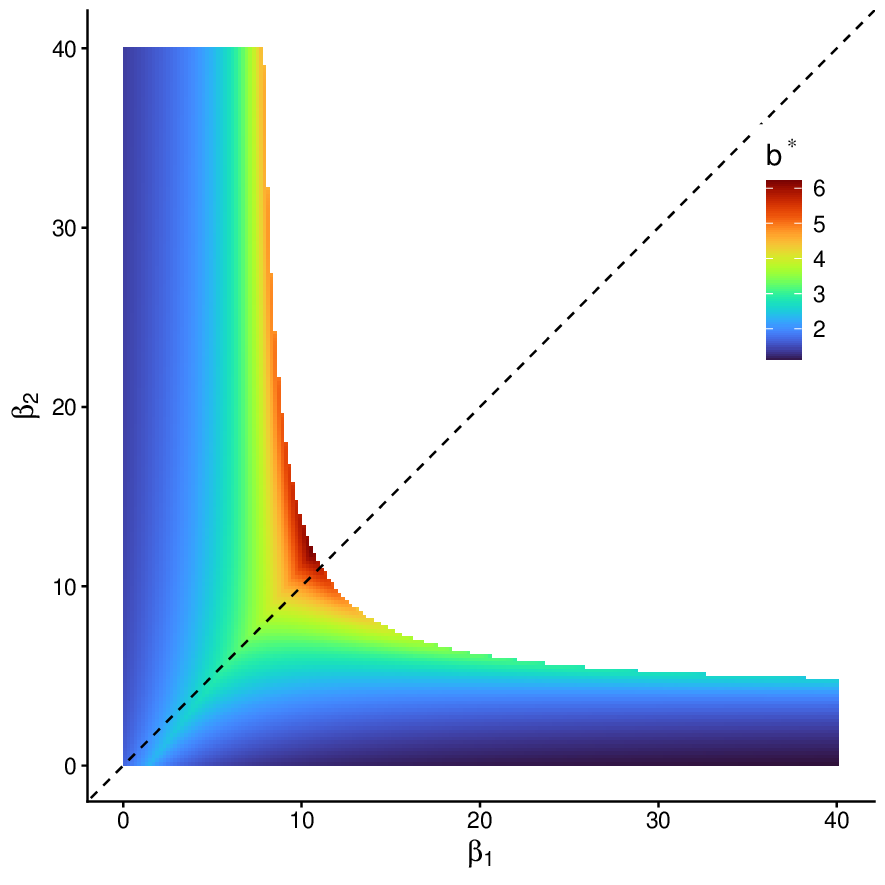}
            \caption{$b^*$}
            \label{fig:contour bstar}
        \end{subfigure}
        \caption{Contour plot of the reinsurance threshold $w_0$ and the optimal barrier $b^*$ as a function of the ambiguity-aversion parameters $\beta_1$ and $\beta_2$. Unshaded portions represent the regions where $w_0$ does not exist or where $b^* = 0$. 
        }
    \label{fig:contour}
\end{figure}

Figure \ref{fig:contour} provides insights on how the zero-reinsurance threshold $w_0$ and optimal barrier $b^*$ behave with respect to both $\beta_1$ and $\beta_2$. We note that $w_0$ tends to be insensitive to changes in the ambiguity aversion of Line $i$ if the ambiguity aversion of Line $3-i$ is low ($i=1,2$). This means that any asymmetry in the lines' ambiguity aversion, if one of the lines is not too ambiguity averse, does not significantly influence the insurer's decision to cease reinsurance for at least one line. In the same situation, as seen in Figure \ref{fig:bstar}, the optimal barrier tends to peak when the ambiguity-aversion parameters align, but slightly decreases as the discrepancy increases. However, if Line $i$ is already highly ambiguity averse and if the ambiguity aversion of Line $3-i$ increases, it becomes more likely that the insurer will retain all risks from at least one line or pay dividends immediately. Both $w_0$ and $b^*$ tend to rapidly increase when $\beta_1$ and $\beta_2$ increase together (i.e., along the $45^\circ$ line), indicating that the insurer, as a whole, will consider ceasing reinsurance or paying out dividends at a higher level of reserves if \textit{both} lines become more ambiguity averse.

\begin{figure}[htb]
    \centering
    \begin{subfigure}[]{0.3\textwidth}
        \centering
        \includegraphics[width=\linewidth]{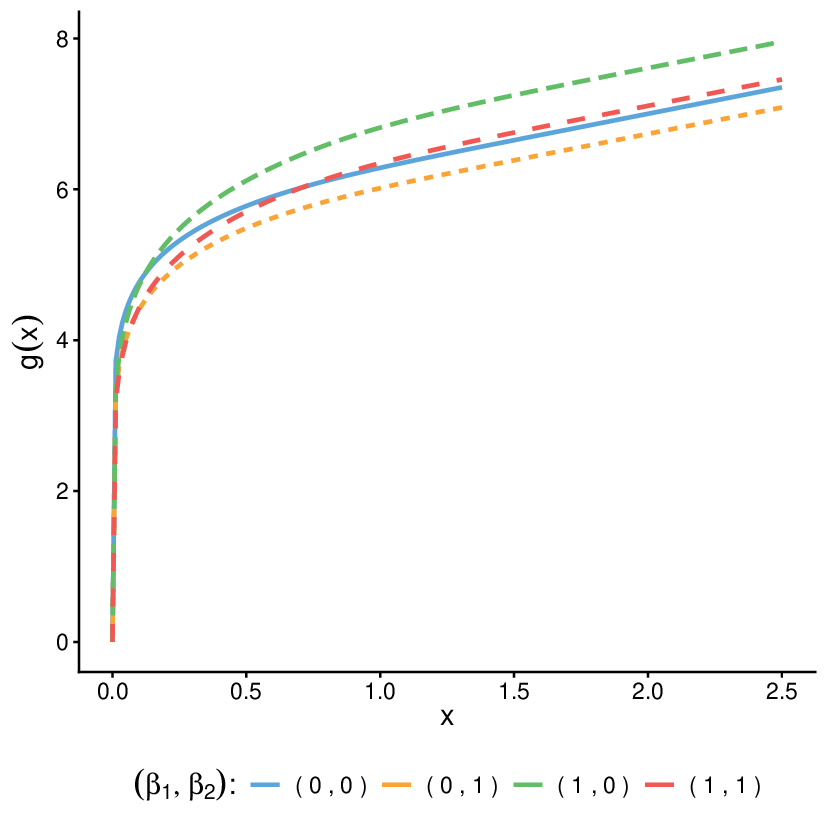} 
    \end{subfigure}
    \quad
    \begin{subfigure}[]{0.3\textwidth}
        \centering
        \includegraphics[width=\linewidth]{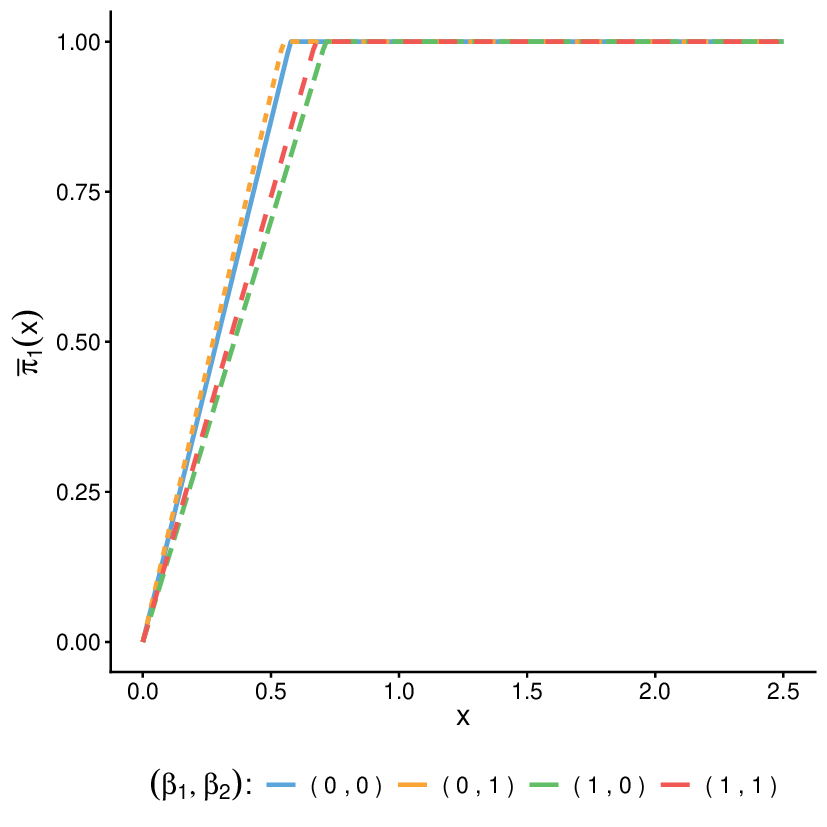}
    \end{subfigure}
    \quad
    \begin{subfigure}[]{0.3\textwidth}
        \centering
        \includegraphics[width=\linewidth]{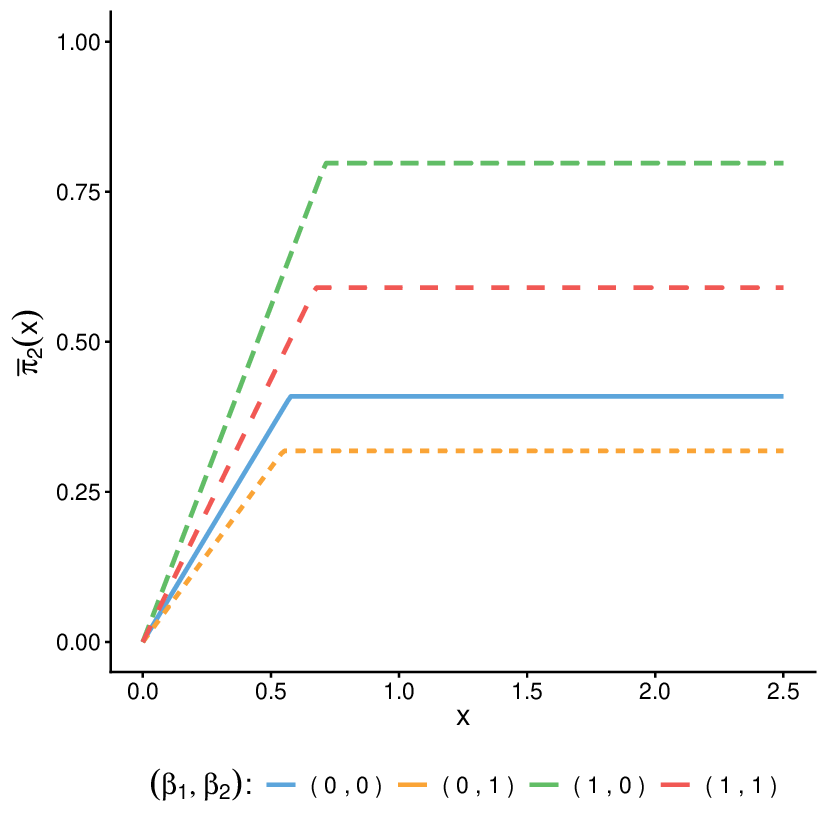}
    \end{subfigure}
    \caption{Value function $g(x)$ and optimal reinsurance strategies $\bar{\pi}_1(x), \bar{\pi}_2(x)$ for various values of $\beta_1$ and $\beta_2$, illustrating the impact of model uncertainty.}
    \label{fig:MUComp}
\end{figure}

\begin{table}[ht]
    \caption{Impact of the ambiguity-aversion parameter (and model uncertainty) on the reinsurance threshold $w_0$ and the optimal barrier $b^*$.}
    \label{tab:MUComp w0 bstar}
    \centering
    \begin{tabular}{cccc}
    \toprule
    $\beta_1$ & $\beta_2$ & $w_0$ & $b^*$\\
    \midrule
    0 & 0 & 0.5762609 & 1.634825\\
    1 & 0 & 0.7133303 & 2.118944\\
    0 & 1 & 0.5479603 & 1.641018\\
    1 & 1 & 0.6744453 & 1.951939\\
    \bottomrule
    \end{tabular}
\end{table}

Next, we examine the value function and the optimal reinsurance strategies under four scenarios: (i) both lines are ambiguity neutral; (ii) and (iii) one line is ambiguity averse and the other line is ambiguity neutral; and (iv) both lines are ambiguity averse (see Figure \ref{fig:MUComp}). These results capture the impact of ignoring model uncertainty on the insurer's welfare (measured by the value function) and reinsurance strategies. In all four scenarios, it is optimal for the insurer to internalize 
all the risk of Line 1 once the underlying reserve value reaches $w_0$ (given in Table  \ref{tab:MUComp w0 bstar}). However, we note the asymmetry in the outcomes when exactly one line is ambiguity averse. This is primarily due to the assumed values of $\mu_i$ and $\sigma_i$, so the insurer's decision to cease reinsurance or pay out dividends under model uncertainty is influenced by the specification of each line's reserve process.\footnote{We note that there is no way to know a priori which scenario will yield the highest value function based on $\mu_i$ and $\sigma_i$.} When $\mu_1 = \mu_2$ and $\sigma_1 = \sigma_2$, then the outcomes in scenarios (ii) and (iii) are identical (see Appendix \ref{sec:MU comp b}). Indeed, the results in Appendix \ref{sec:MU comp b} are consistent with those of \citet{feng2021} who consider the univariate case; due to making decisions under the worst-case scenario when both lines (i.e., the insurer as a whole) are ambiguity averse (and hence taking on a more conservative strategy), there is a decrease in the overall value function compared to when decisions are made under no ambiguity aversion. Thus, these results highlight that in the multivariate setting, the optimal strategies can differ drastically given a more granular approach to managing model risk.

\begin{figure}[htb]
    \centering
        \begin{subfigure}[H]{0.35\textwidth}
            \centering
            \includegraphics[width=\linewidth]{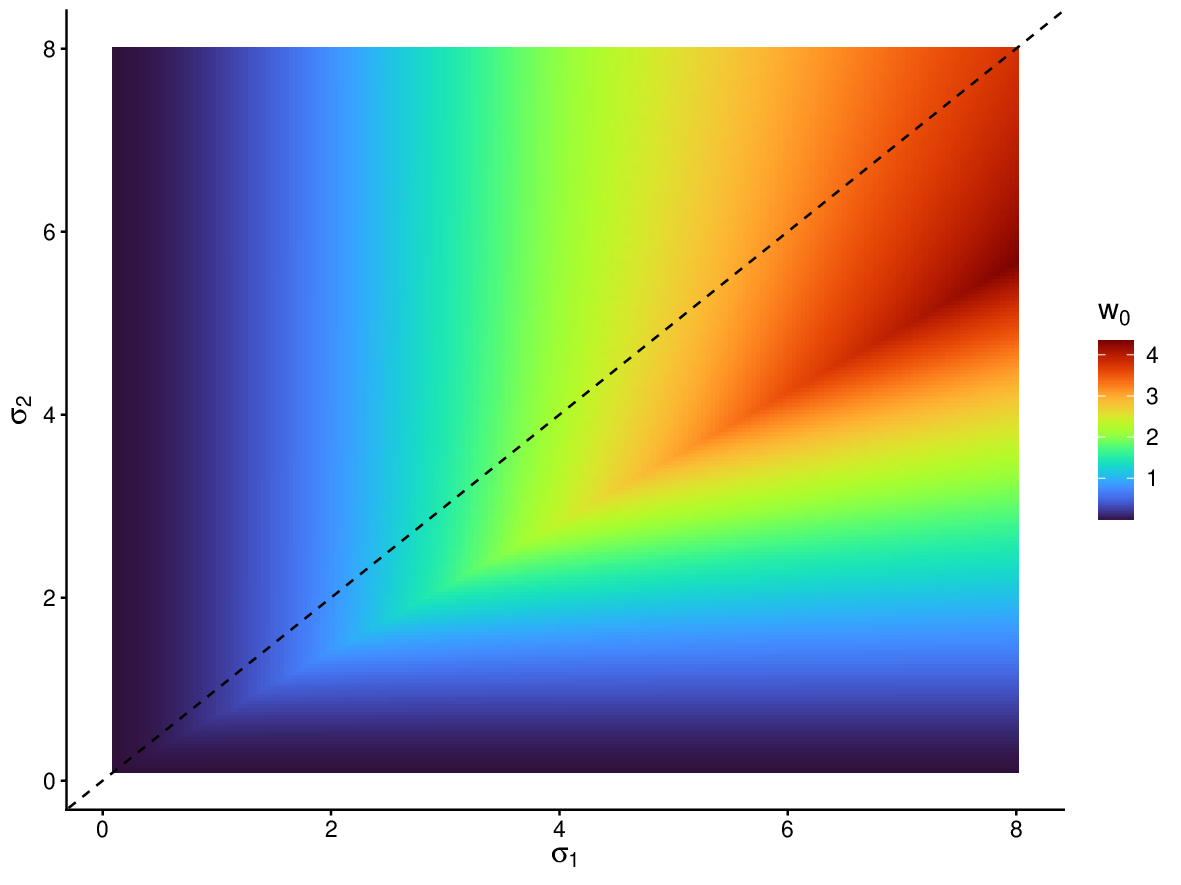} 
            \caption{$w_0$}
        \end{subfigure}
        \qquad
        \begin{subfigure}[H]{0.35\textwidth}
            \centering
            \includegraphics[width=\linewidth]{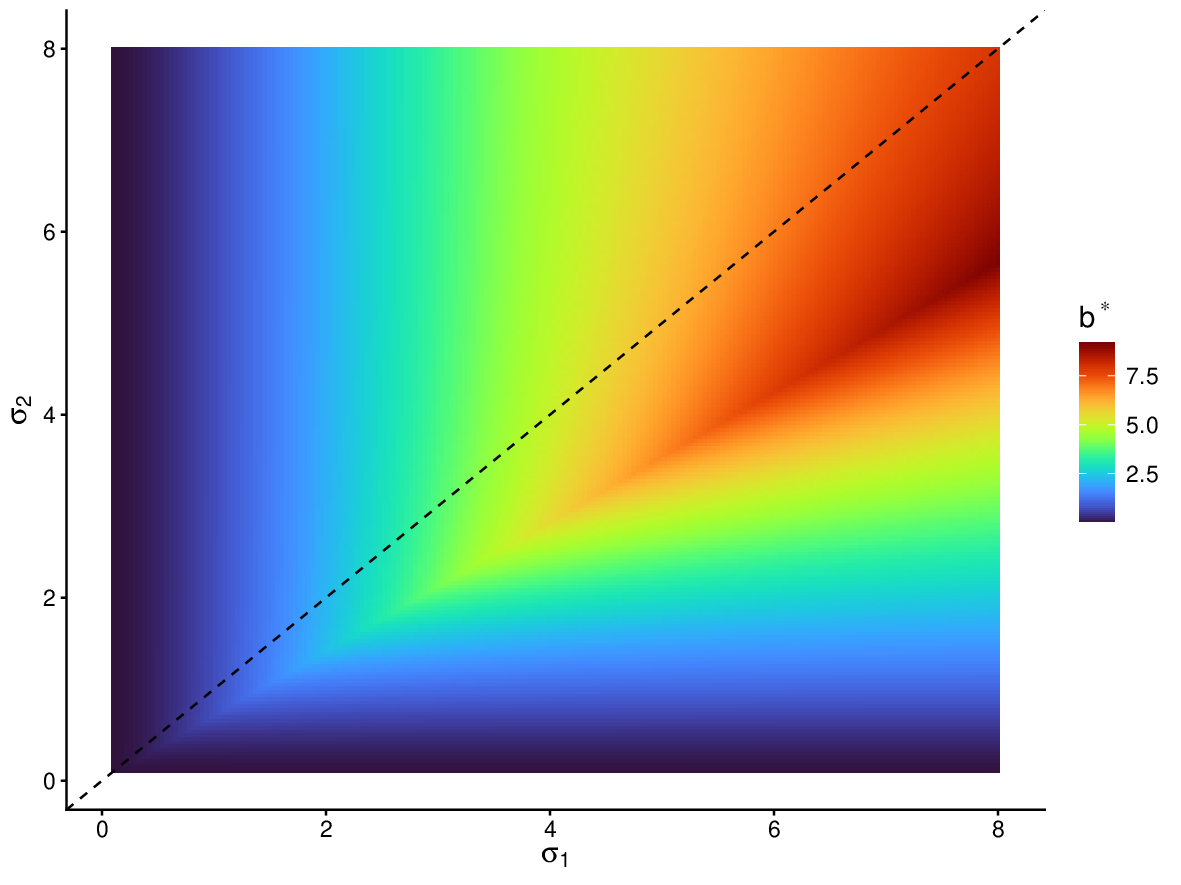}
            \caption{$b^*$}
        \end{subfigure}
        \\
        \caption{Contour plot of the reinsurance threshold $w_0$ and the optimal barrier $b^*$ as a function of $\sigma_1$ and $\sigma_2$ when $\beta_1 = \beta_2 = 0$ (i.e., both lines are ambiguity neutral). Here, we assume $\rho = 0$ to ensure the conditions of Theorem \ref{theorem no uncertainty} are satisfied. 
        }
    \label{fig:contour_sigma_NU}
\end{figure}

\begin{figure}[htb]
    \centering
        \begin{subfigure}[H]{0.35\textwidth}
            \centering
            \includegraphics[width=\linewidth]{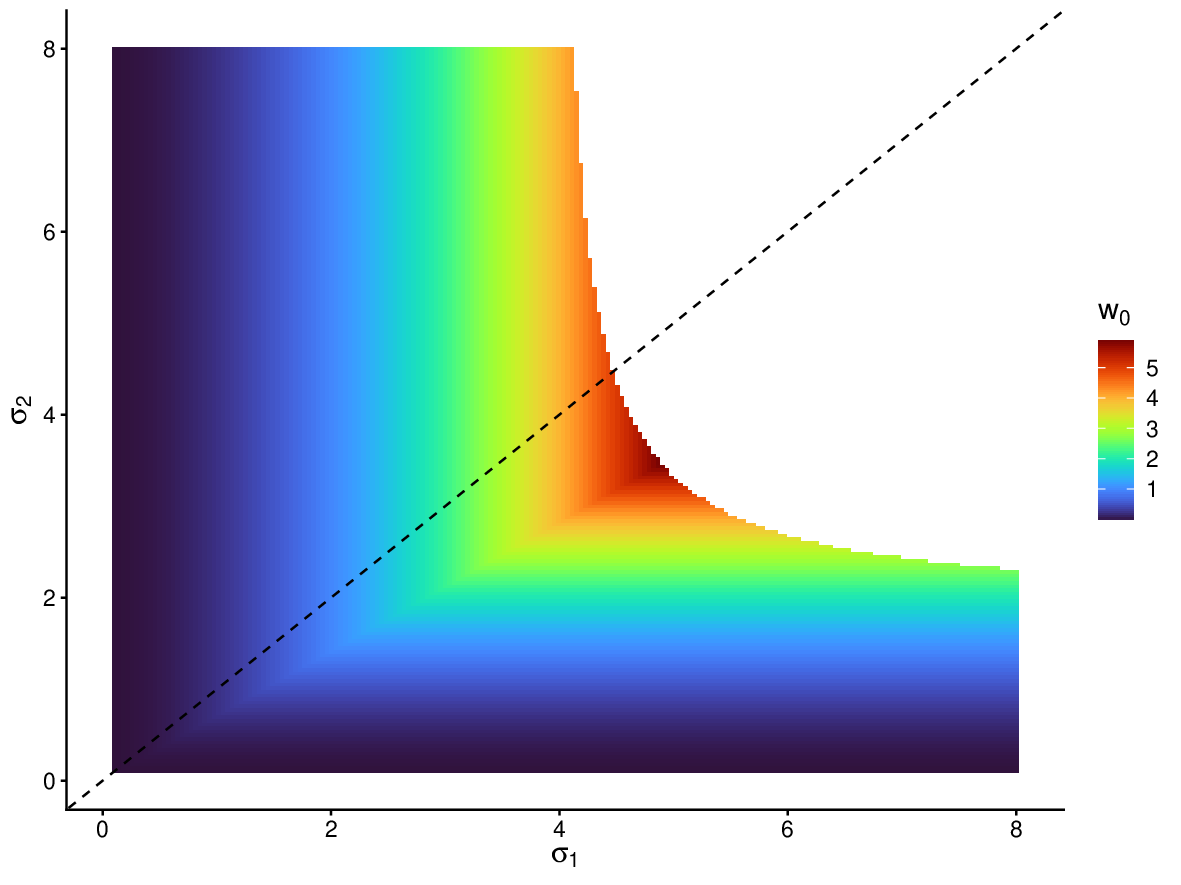} 
            \caption{$w_0$}
        \end{subfigure}
        \qquad
        \begin{subfigure}[H]{0.35\textwidth}
            \centering
            \includegraphics[width=\linewidth]{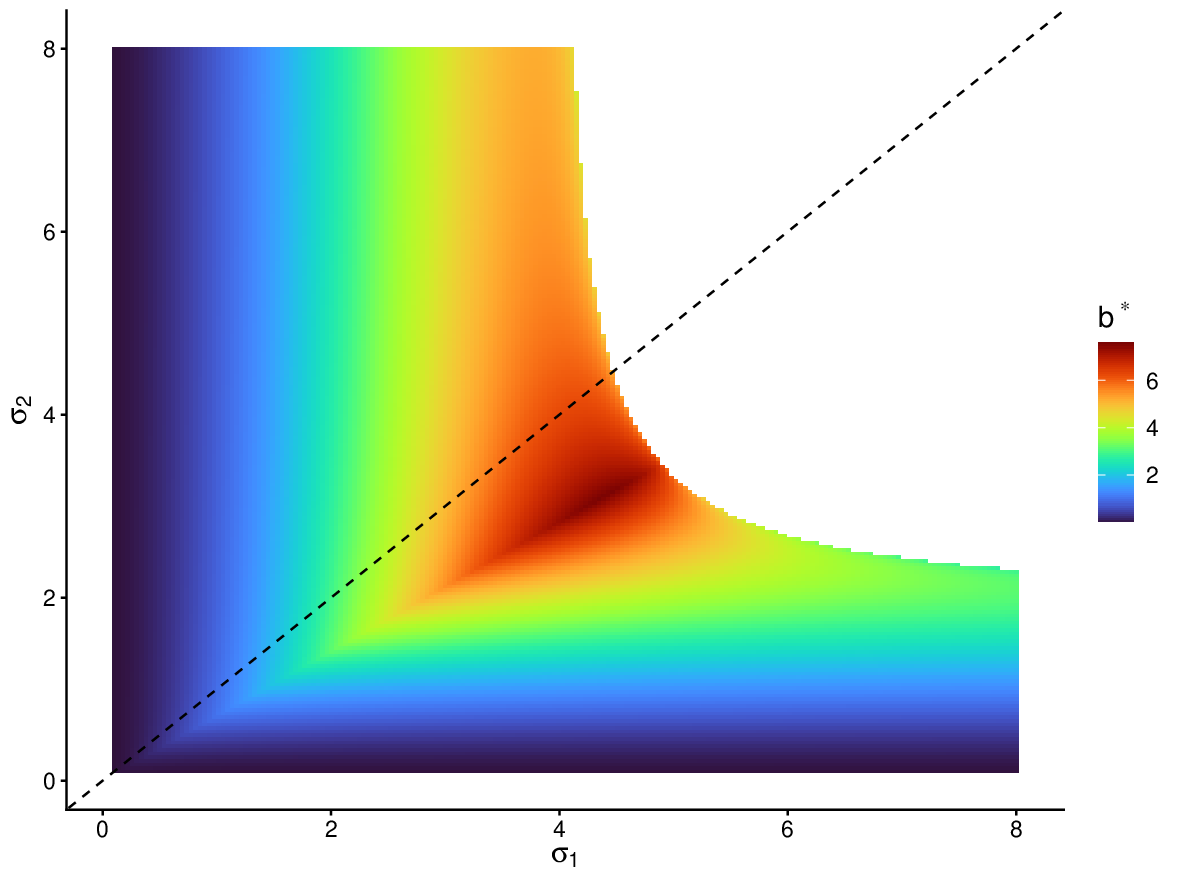}
            \caption{$b^*$}
            \label{fig:contour_sigma_bstar}
        \end{subfigure}
        \caption{Contour plot of the reinsurance threshold $w_0$ and the optimal barrier $b^*$ as a function of $\sigma_1$ and $\sigma_2$ when $\beta_1 = \beta_2 = 1$ (i.e., both lines are ambiguity averse). Here, we assume $\rho = 0$ to compare with Figure \ref{fig:contour_sigma_NU}. 
        }
    \label{fig:contour_sigma}
\end{figure}

Finally, Figures \ref{fig:contour_sigma_NU} and \ref{fig:contour_sigma} illustrate the interaction between ambiguity aversion and the riskiness of each line's reserve process (captured by $\sigma_1$ and $\sigma_2$) and how the interaction affects the reinsurance threshold $w_0$ and optimal barrier $b^*$ for the (optimal) dividend payout and capital injection strategies. When both lines are ambiguity neutral (i.e., $\beta_1 = \beta_2 = 0$), there exist nontrivial $w_0$ and $b^*$ for all values of $\sigma_1$ and $\sigma_2$ considered. Furthermore, the behavior of $w_0$ and $b^*$ with respect to $\sigma_2$ when $\sigma_1$ is fixed (say at $\sigma_1 = 2$) does not appear to be monotonic; $w_0$ and $b^*$ tend to increase as $\sigma_2$ increases from 0, but once a sufficiently high level of $\sigma_2$ is attained, the values of $w_0$ and $b^*$ tend to decrease as $\sigma_2$ increases beyond this level. In this case, high risk in both lines is managed through a combination of reinsurance, dividend payment, and capital injection. In contrast, when both lines are ambiguity averse ($\beta_1 = \beta_2 = 1$), Line 2 will immediately pay the aggregate reserve as dividends at time $t = 0$ and ruin will occur immediately after for sufficiently high levels of risk (refer to the blank region in Figure \ref{fig:contour_sigma_bstar}). Likewise, optimal reinsurance becomes irrelevant when the risk in both lines is sufficiently high. Thus, when the insurer exhibits some degree of distrust in the reference model, high risk in both lines’ reserve processes becomes a signal to pay dividends immediately.

\section{Proof of Main Results}\label{sec:proofs}

\subsection{Proof of Lemma \ref{lemma: gamma1 suff cond}}

    (Part I ``$\Longleftarrow$'') We first prove the ``if" direction. Suppose $\delta < \frac{\mu_1^2}{2\beta_1\sigma_1^2} + \frac{\mu_2^2}{2\beta_2\sigma_2^2}$, which can be rewritten as $\beta_2\mu_1^2\sigma_2^2+\beta_1\mu_2^2\sigma_1^2-2\delta\beta_1\beta_2\sigma_1^2\sigma_2^2>0$. We obtain $\psi(0)=-2\delta\sigma_1^2\sigma_2^2(1-\rho^2)^2<0$ and $\psi(1)=\beta_1\beta_2(\beta_2\mu_1^2\sigma_2^2+\beta_1\mu_2^2\sigma_1^2-2\delta\beta_1\beta_2\sigma_1^2\sigma_2^2)>0$, where $\psi(z)$ is defined in \eqref{eqn: defn of psi}. The intermediate value theorem implies the existence of the solution of \eqref{eqn: gamma1 equation} on $(0,1)$. 

    (Part II ``$\Longrightarrow$'') We prove the ``only if" direction using its contrapositive; that is, we show that \eqref{eqn: gamma1 equation} does not have a solution on $(0,1)$ if $\delta \geq \frac{\mu_1^2}{2\beta_1\sigma_1^2} + \frac{\mu_2^2}{2\beta_2\sigma_2^2}$. Suppose now that $\delta \geq \frac{\mu_1^2}{2\beta_1\sigma_1^2} + \frac{\mu_2^2}{2\beta_2\sigma_2^2}$, which is equivalent to $\beta_2\mu_1^2\sigma_2^2+\beta_1\mu_2^2\sigma_1^2-2\delta\beta_1\beta_2\sigma_1^2\sigma_2^2\leq 0$. It must be noted that 
    \begin{equation}
        0\geq \beta_2\mu_1^2\sigma_2^2+\beta_1\mu_2^2\sigma_1^2-2\delta\beta_1\beta_2\sigma_1^2\sigma_2^2 = \beta_2\sigma_2^2(\mu_1^2-2\delta \beta_1\sigma_1^2)+\beta_1\mu_2^2\sigma_1^2 = \beta_1\sigma_1^2(\mu_2^2-2\delta\beta_2\sigma_2^2)+\beta_2\mu_1^2\sigma_2^2,
    \end{equation}
    which implies that $\mu_1^2<2\delta \beta_1\sigma_1^2$ and $\mu_2^2<2\delta \beta_2\sigma_2^2$ must both hold. We then have $\psi(0)<0$ and $\psi(1)<0$. It suffices to show that $\psi(z)<0$ for all $z\in(0,1)$. We can rewrite $\psi$ as 
    \begin{equation}
        \psi(z)=h_1z^4+h_2z^3(z-1)+h_3(z-1)^4+h_4z(z-1)^3+h_5z^2(z-1)^2,
    \end{equation}
    where
        \begin{align}
            h_1&=\beta_1\beta_2(\beta_2\mu_1^2\sigma_2^2+\beta_1\mu_2^2\sigma_1^2-2\delta\beta_1\beta_2\sigma_1^2\sigma_2^2) \leq 0,\\
            h_2&=-\beta_2\mu_1^2\sigma_2^2(2\beta_1+\beta_2)-\beta_1\mu_2^2\sigma_1^2(\beta_1+2\beta_2)+2\beta_1\beta_2\mu_1\mu_2\sigma_1\sigma_2+4\delta\beta_1\beta_2(\beta_1+\beta_2)\sigma_1^2\sigma_2^2\\
            &=-(\beta_1+\beta_2)(\beta_2\mu_1^2\sigma_2^2+\beta_1\mu_2^2\sigma_1^2-2\delta\beta_1\beta_2\sigma_1^2\sigma_2^2)+2\beta_1\beta_2\mu_1\mu_2\sigma_1\sigma_2\\
            &\qquad+\beta_1\beta_2[\sigma_2^2(2\delta\beta_1\sigma_1^2-\mu_1^2)+\sigma_1^2(2\delta\beta_2\sigma_2^2-\mu_2^2)]>0,\\
            h_3&=-2\delta\sigma_1^2\sigma_2^2(1-\rho^2)^2<0,\\
            h_4&=(\mu_1^2\sigma_2^2+\mu_2^2\sigma_1^2)(3\rho+1)(\rho-1)+2\mu_1\mu_2\sigma_1\sigma_2(1-\rho)(2\rho^2+\rho+1)+4\delta(1-\rho^2)(\beta_1+\beta_2)\sigma_1^2\sigma_2^2\\
            &>(\mu_1^2\sigma_2^2+\mu_2^2\sigma_1^2)(3\rho+1)(\rho-1)+2\mu_1\mu_2\sigma_1\sigma_2(1-\rho)(2\rho^2+\rho+1)+2(1-\rho^2)(\mu_1^2\sigma_2^2+\mu_2^2\sigma_1^2)\\
            &=(\mu_1^2\sigma_2^2+\mu_2^2\sigma_1^2)(\rho-1)^2+2\mu_1\mu_2\sigma_1\sigma_2(1-\rho)(2\rho^2+\rho+1)>0,\\
            h_5&=\mu_1^2\sigma_2^2(\beta_1-3\rho^2\beta_2+2\rho\beta_2+2\beta_2)+\mu_2^2\sigma_1^2(\beta_2-3\rho^2\beta_1+2\rho\beta_1+2\beta_1)-2\mu_1\mu_2\sigma_1\sigma_2(\beta_1+\beta_2)\\
            &\qquad -2\delta\sigma_1^2\sigma_2^2(\beta_1^2+2\beta_1\beta_2+\beta_2^2)-4\delta\sigma_1^2\sigma_2^2(1-\rho^2)\beta_1\beta_2\\
            &<\mu_1^2\sigma_2^2(\beta_1-3\rho^2\beta_2+2\rho\beta_2+2\beta_2)+\mu_2^2\sigma_1^2(\beta_2-3\rho^2\beta_1+2\rho\beta_1+2\beta_1)-2\mu_1\mu_2\sigma_1\sigma_2(\beta_1+\beta_2)\\
            &\qquad -(\beta_1+\beta_2)(\mu_1^2\sigma_2^2+\mu_2^2\sigma_1^2)-4\delta\sigma_1^2\sigma_2^2(1-\rho^2)\beta_1\beta_2\\
            &=-\beta_2\mu_1^2\sigma_2^2(\rho-1)^2-\beta_1\mu_2^2\sigma_1^2(\rho-1)^2-2\mu_1\mu_2\sigma_1\sigma_2(\beta_1+\beta_2)\\
            &\qquad +2(1-\rho^2)(\beta_2\mu_1^2\sigma_2^2+\beta_1\mu_2^2\sigma_1^2-2\delta\beta_1\beta_2\sigma_1^2\sigma_2^2)<0.
        \end{align}
    Since $z\in(0,1)$, it holds that $\psi(z)<0$, which completes the proof.

\subsection{Proof of Theorem \ref{theorem 1}}

In this section, we present the key results used to obtain Theorem \ref{theorem 1}. The discussion serves as the proof for Theorem \ref{theorem 1}.

Suppose $\gamma_1\in(0,1)$ exists, \eqref{eqn: correlation bounds} holds, and $N_1\neq N_3$, as required in Theorem \ref{theorem 1}. Recall that $N_1$ and $N_3$ are defined in \eqref{eqn: N1, N2, N3}. Since $\gamma_1$ exists, then $\delta < \frac{\mu_1^2}{2\beta_1\sigma_1^2} + \frac{\mu_2^2}{2\beta_2\sigma_2^2}$. We conjecture that $g$ satisfies $g'(x) > a_2$ for $x< b^*$ and $g'(x) = a_2$ for $x\geq b^*$.

In the region $\{x<w_0\}$, the HJB equation \eqref{eqn:hjborig} becomes
\begin{equation}\label{eqn:hjb x<w0}
    \frac{A(x)}{2\sigma_1^2\sigma_2^2B(x)}-\delta g(x)=0,
\end{equation}
where
\begin{align}
    A(x)&:= \left[\mu_1^2\sigma_2^2(3\rho+1)(\rho-1)+\mu_2^2\sigma_1^2(3\rho+1)(\rho-1)+2\mu_1\mu_2\sigma_1\sigma_2(1-\rho)(2\rho^2+\rho+1) \right]g''(x)^3g'(x)^2g(x)^4\\
        &\qquad + \left[-\beta_2\mu_1^2\sigma_2^2(2\beta_1+\beta_2)-\beta_1\mu_2^2\sigma_1^2(\beta_1+2\beta_2)+2\beta_1\beta_2\mu_1\mu_2\sigma_1\sigma_2 \right] g''(x)g'(x)^6g(x)^2\\
        &\qquad+ \left[\mu_1^2\sigma_2^2(\beta_1-3\rho^2\beta_2+2\rho\beta_2+2\beta_2)+\mu_2^2\sigma_1^2(\beta_2-3\rho^2\beta_1+2\rho\beta_1+2\beta_1) \right.\\
        &\qquad\qquad \left. -2\mu_1\mu_2\sigma_1\sigma_2(\beta_1+\beta_2) \right]g''(x)^2g'(x)^4g(x)^3 + \left[\beta_1\beta_2^2\mu_1^2\sigma_2^2+\beta_1^2\beta_2\mu_2^2\sigma_1^2\right]g'(x) ^8g(x),\\
        B(x)&:=[(\beta_1+\beta_2)^2+2(1-\rho^2)\beta_1\beta_2]g''(x)^2g'(x)^4g(x)^2+(1-\rho^2)^2g''(x)^4g(x)^4+\beta_1^2\beta_2^2g'(x)^8\\
        &\qquad-2(1-\rho^2)(\beta_1+\beta_2)g''(x)^3g'(x)^2g(x)^3-2\beta_1\beta_2(\beta_1+\beta_2)g''(x)g'(x)^6g(x).
\end{align}
We use the ansatz $g_1(x)=K_1x^{\gamma_1}$, with $K_1>0$ and $\gamma_1\in(0,1)$, on \eqref{eqn:hjb x<w0} and obtain $\frac{A(x)-2K_1\delta\sigma_1^2\sigma_2^2B(x)x^{\gamma_1}}{2\sigma_1^2\sigma_2^2B(x)}=0,$
which is equivalent to
\begin{equation}\label{eqn: A-B gamma1}
    \frac{\widetilde A(\gamma_1) -2\delta\sigma_1^2\sigma_2^2\widetilde B(\gamma_1)}{2\sigma_1^2\sigma_2^2\widetilde B(\gamma_1)}x^{\gamma_1}=0,
\end{equation}
where $\widetilde A(z)$ and $\widetilde B(z)$ are defined in \eqref{eqn: defn of psi}. Since $g$ must be (strictly) increasing in $x$ and $g(0)=0$, it follows that $g_1(x)>0$ in the region $\{0<x<b^*\}$. Moreover, since $\gamma_1\in(0,1)$ and $\beta_i>0$, we have $\widetilde B(\gamma_1)>0$. Hence, \eqref{eqn: A-B gamma1} is also equivalent to
\begin{equation}\label{eqn: gamma1 equation}
    \widetilde A(\gamma_1) -2\delta\sigma_1^2\sigma_2^2\widetilde B(\gamma_1) = 0.
\end{equation}

We now determine the corresponding optimal reinsurance levels in the region $\{x<w_0\}$ with the constraint $\pi_i\in[0,1]$. From \eqref{eqn: optimal pi} and the ansatz $g_1(x)=K_1x^{\gamma_1}$, we have $\bar\pi_1(x)=\frac{x}{w_1}$ and $\bar\pi_2(x)=\frac{x}{w_2}$.
It is clear that $w_0=\min\{w_1,w_2\}$. From \eqref{eqn:theta candidates}, the optimal distortion levels are then given by $\bar\theta_1(x)=\frac{\beta_1\sigma_1\gamma_1}{w_1}$ and $\bar\theta_2(x)=\frac{\beta_2\sigma_2\gamma_1}{w_2}$.
Due to the constraint $\pi_i\in[0,1]$, at least one line must retain all of its risk when the aggregate reserve level exceeds the threshold $w_0$. The following lemma states that the retention level of \emph{both} lines remains constant when $x\geq w_0$.
\begin{lemma}\label{lemma: pi remain constant}
    Write $\tilde w_0:=\min\{w_1,w_2\}$. For $x\geq \tilde w_0$, $(\bar\pi_1,\bar\pi_2)(x)=\left(\frac{\tilde w_0}{w_1},\frac{\tilde w_0}{w_2}\right).$
\end{lemma}
\begin{proof}
Suppose $w_1\leq w_2$. It follows that $\tilde w_0=w_1$ and $\bar\pi_1(x) = 1$ for $x\geq \tilde w_0$.  Since $ \bar\pi_2(x) = \frac{x}{w_2}= \frac{\bar\pi_1(x) w_1}{w_2}=\frac{\tilde w_0}{w_2}$, the result then follows. The proof for the case where $w_2\leq w_1$ is similar.
\end{proof}
Using the above lemma, in the region $\{w_0<x<b^*\}$, the HJB equation \eqref{eqn:hjborig} becomes
\begin{equation}\label{eqn: ode 2}
    N_1g''(x)+N_2g'(x)-N_3\frac{g'(x)^2}{g(x)}-\delta g(x) =0,
\end{equation}
where $N_1$, $N_2$, and $N_3$ are defined in \eqref{eqn: N1, N2, N3}. Define $v(x)$ such that it satisfies
\begin{equation}\label{eqn: defn of v}
    v(x):=\frac{g_2'(x)}{g_2(x)},
\end{equation}
where $g_2$ satisfies \eqref{eqn: ode 2}. We can then rewrite \eqref{eqn: ode 2} as
\begin{equation}\label{eqn: ode for v}
    v'(x)+\left(1-\frac{N_3}{N_1}\right)v(x)^2+\frac{N_2}{N_1}v(x)-\frac{\delta}{N_1}=0,
\end{equation}
which is a Riccati equation since we assume that $N_1\neq N_3$. Define $u(x)$ such that it satisfies $\left(1-\frac{N_3}{N_1}\right)v(x)=\frac{u'(x)}{u(x)}$.
Then, \eqref{eqn: ode for v} is equivalent to
\begin{equation}\label{eqn: ode for u}
    u''(x) + \frac{N_2}{N_1}u'(x) - \left(1-\frac{N_3}{N_1}\right)\frac{\delta}{N_1}u(x)=0,
\end{equation}
whose solution is given by $u(x)=K_{3+}e^{\gamma_{2+}x}+K_{3-}e^{\gamma_{2-}x},$
where $\gamma_{2\pm}$ are defined in \eqref{eqn: v(x), K3, gamma2pm thm 1},
provided that $N_2^2+4\delta(N_1-N_3)>0$. The following lemma proves that $N_2^2+4\delta(N_1-N_3)>0$ indeed holds.
\begin{lemma}
    Suppose $\delta < \frac{\mu_1^2}{2\beta_1\sigma_1^2} + \frac{\mu_2^2}{2\beta_2\sigma_2^2}$. It holds that $N_2^2+4\delta(N_1-N_3)>0$.
\end{lemma}
\begin{proof}
    Suppose $N_1>N_3$. The result immediately follows. Suppose now that $N_1<N_3$. Define $h(\delta):=N_2^2+4\delta(N_1-N_3). $
Write $\overline\delta:=\frac{\mu_1^2}{2\beta_1\sigma_1^2}+\frac{\mu_2^2}{2\beta_2\sigma_2^2}.$ From Lemma \ref{lemma: gamma1 suff cond}, it follows that $\gamma_1=1$ if $\delta=\overline \delta$. Then, $\frac{w_1}{w_2}|_{\gamma_1=1}=\frac{\beta_1\mu_2\sigma_1^2}{\beta_2\mu_1\sigma_2^2}.$
Hence,
\begin{align}
    h(\overline \delta)&=N_2^2+4\left(\frac{\mu_1^2}{2\beta_1\sigma_1^2}+\frac{\mu_2^2}{2\beta_2\sigma_2^2}\right)\left(\frac{1}{2}\sigma_1^2(1-\beta_1)+\frac{1}{2}\sigma_2^2(1-\beta_2)\frac{\beta_1^2\mu_2^2\sigma_1^4}{\beta_2^2\mu_1^2\sigma_2^4}+\rho\sigma_1\sigma_2\frac{\beta_1\mu_2\sigma_1^2}{\beta_2\mu_1\sigma_2^2}\right)\\
    &=N_2^2+\left(\frac{2\mu_1^2}{\beta_1\sigma_1^2}+\frac{2\mu_2^2}{\beta_2\sigma_2^2}\right)\left(\frac{1}{2}\sigma_1^2+\frac{1}{2}\sigma_2^2\frac{\beta_1^2\mu_2^2\sigma_1^4}{\beta_2^2\mu_1^2\sigma_2^4}+\rho\sigma_1\sigma_2\frac{\beta_1\mu_2\sigma_1^2}{\beta_2\mu_1\sigma_2^2}\right)\\
    &\quad + \left(\frac{2\mu_1^2}{\beta_1\sigma_1^2}+\frac{2\mu_2^2}{\beta_2\sigma_2^2}\right)\left(-\frac{1}{2}\beta_1\sigma_1^2-\frac{1}{2}\beta_2\sigma_2^2\frac{\beta_1^2\mu_2^2\sigma_1^4}{\beta_2^2\mu_1^2\sigma_2^4}\right)\\
    &=N_2^2+\left(\frac{2\mu_1^2}{\beta_1\sigma_1^2}+\frac{2\mu_2^2}{\beta_2\sigma_2^2}\right)\left(\frac{1}{2}\sigma_1^2+\frac{1}{2}\sigma_2^2\frac{\beta_1^2\mu_2^2\sigma_1^4}{\beta_2^2\mu_1^2\sigma_2^4}+\rho\sigma_1\sigma_2\frac{\beta_1\mu_2\sigma_1^2}{\beta_2\mu_1\sigma_2^2}\right)- N_2^2>0.
    \end{align}
Since $h'(\delta) =4(N_1-N_3)<0$, we have $h(\delta)>h(\overline\delta)>0$ for all $\delta\in(0,\overline \delta)$.
\end{proof}

From the definition of $v$ in \eqref{eqn: defn of v}, we then have
\begin{equation}
    \frac{\dd }{\dd x}\left[\ln g_2(x)\right]=\frac{g_2'(x)}{g_2(x)}=v(x)=\frac{N_1}{N_1-N_3}\frac{u'(x)}{u(x)}=\frac{N_1}{N_1-N_3}\frac{\dd }{\dd x}\left[\ln u(x)\right].
\end{equation}
Hence,
\begin{equation}
    g_2(x)=K_2\exp\left[\int_{w_0}^x v(y) \dd y\right]=K_2\exp\left[\frac{N_1}{N_1-N_3} \ln \left(\frac{u(x)}{u(w_0)}\right) \right]=K_2\left[\frac{u(x)}{u(w_0)}\right]^{\frac{N_1}{N_1-N_3}},
\end{equation}
where $K_2>0$ is an undetermined constant. It must be noted that $\frac{u(x)}{u(w_0)}=       \frac{K_3e^{\gamma_{2+}x}+e^{\gamma_{2-}x}}{K_3e^{\gamma_{2+}w_0}+e^{\gamma_{2-}w_0}},$
where $K_3:=\frac{K_{3+}}{K_{3-}}$. This implies that $g_2$ implicitly depends on the constant $K_3$.

In the region $\{x>b^*\}$, since we conjecture that $g'(x) = a_2$, it must then hold that $g_3(x)=a_2\left(x-b^*+K_4\right).$
We then conjecture the following solution:
\begin{align}\label{eqn: g conjecture}
    g(x)=
    \begin{cases}
    K_1x^{\gamma_1} & \mbox{if $x<w_0$,}\\
    K_2 e^{\int_{w_0}^x v(y) \dd y} & \mbox{if $w_0\leq x<b^*$,}\\
    a_2\left(x-b^*+K_4\right) & \mbox{if $x\geq b^*$.}
    \end{cases}
\end{align}

To ensure that $g$ is twice continuously differentiable, we require $g$, $g'$, and $g''$ to be continuous at the switching points $w_0$ and $b^*$. Since in the neighborhood of $w_0$ the function $g$ satisfies \eqref{eqn:hjborig}, it suffices to show that $g$ and $g'$ are continuous at $x=w_0$. We have the following system of equations:
\begin{equation}\label{eqn: system at w0}
    \begin{aligned}
        K_1w_0^{\gamma_1}&= K_2\\
        K_1\gamma_1w_0^{\gamma_1-1}&=K_2 v(w_0),
    \end{aligned}
\end{equation}
which yields
\begin{equation}\label{eqn: K1 in terms of K2}
    K_1=K_2w_0^{-\gamma_1}.
\end{equation}
Moreover, $K_3$ must satisfy
\begin{equation}\label{eqn: v(w0) for K3}
    v(w_0)=\frac{\gamma_1}{w_0}.
\end{equation}
We then obtain $K_3$ defined in \eqref{eqn: v(x), K3, gamma2pm thm 1}.

In the neighborhood of $b^*$, we have the following system of equations:
\begin{equation}\label{eqn: system at b*}
    \begin{aligned}
        K_2 \exp\left[\int_{w_0}^{b^*}v(y)\, \dd y\right]&= a_2 K_4\\
        K_2v(b^*) \exp\left[\int_{w_0}^{b^*}v(y)\, \dd y\right]&= a_2 \\
        K_2\left[v(b^*)^2+v'(b^*)\right] \exp\left[\int_{w_0}^{b^*}v(y)\, \dd y\right]&= 0.
    \end{aligned}
\end{equation}
From the first and second equations, we obtain
\begin{equation}\label{eqn: K2 and K4 from system}
    K_2=\frac{a_2}{v(b^*)}\exp\left[-\int_{w_0}^{b^*}v(y)\, \dd y\right] \quad \mbox{and} \quad
    K_4=\frac{1}{v(b^*)}.
\end{equation}
From the third equation, $b^*$ must satisfy
\begin{equation}\label{eqn: v2+v' =0 for b*}
    v(b^*)^2+v'(b^*)=0.
\end{equation}
The following lemma guarantees the existence of $b^*$.

\begin{lemma}\label{lemma on u1 exist case 1}
    Suppose that $\delta < \frac{\mu_1^2}{2\beta_1\sigma_1^2} + \frac{\mu_2^2}{2\beta_2\sigma_2^2}$ and $N_1\neq N_3$. Then, $b^*$ exists and satisfies \eqref{eqn: b^* theorem 1}.
\end{lemma}
\begin{proof}
    From \eqref{eqn: ode for v}, we have $v'(x)+v(x)^2=\frac{1}{N_1}\left[N_3v(x)^2-N_2v(x)+\delta\right].$
    Using \eqref{eqn: v(w0) for K3} yields
    \begin{equation}
        \begin{aligned}
            v'(w_0)+v(w_0)^2
            &=\frac{1}{N_1}\left[N_3v(w_0)^2-N_2v(w_0)+\delta\right]\\
            &=\frac{1}{N_1}\left[N_3\frac{\gamma_1^2}{w_0^2}-N_2\frac{\gamma_1}{w_0}+\delta\right]\\
            &=\frac{-\psi(\gamma_1)+\widetilde f(\gamma_1)}{2\sigma_1^2\sigma_2^2\left[(\beta_1+\beta_2)\gamma_1(\gamma_1-1)-(1-\rho^2)(\gamma_1-1)^2-\beta_1\beta_2\gamma_1^2\right]^2},
        \end{aligned}
    \end{equation}
    where
    \begin{equation}
    \begin{aligned}
        \widetilde f(\gamma_1)&:=(\beta_2\mu_1\sigma_2+\beta_1\mu_2\sigma_1)^2\gamma_1^3(\gamma_1-1)+(1-\rho)^2(\mu_1^2\sigma_2^2+\mu_2^2\sigma_1^2)\gamma_1(\gamma_1-1)^3\\
        &\qquad+2(\rho-1)\left[\beta_2\mu_1^2\sigma_2^2+\beta_1\mu_2^2\sigma_1^2+(\beta_1+\beta_2)\mu_1\mu_2\sigma_1\sigma_2\right]\gamma_1^2(\gamma_1-1)^2.
        \end{aligned}
    \end{equation}
    Since $\gamma_1$ is a solution of $\psi(z)=0$ on $(0,1)$, we have
    \begin{equation}
        \begin{aligned}
            v'(w_0)+v(w_0)^2
            &=\frac{\widetilde f(\gamma_1)}{2\sigma_1^2\sigma_2^2\left[(\beta_1+\beta_2)\gamma_1(\gamma_1-1)-(1-\rho^2)(\gamma_1-1)^2-\beta_1\beta_2\gamma_1^2\right]^2}<0.
        \end{aligned}
    \end{equation}
    
    Suppose $b^*$ does not exist. Since $v(x)^2+v'(x)$ is continuous in $x$, we have
    \begin{equation}\label{eqn: v' + v^2<0}
        v'(x)+v(x)^2 <0 \quad \mbox{for $x>w_0$.}
    \end{equation}
    It must be noted that $\gamma_{2+}>0$ if and only if $N_1>N_3$. Moreover, $\gamma_{2+}>\gamma_{2-}$. Since $v(x)=\frac{N_1}{N_1-N_3}\cdot\frac{u'(x)}{u(x)}$, it holds that
    \begin{equation}\label{eqn: limit of v}
        \lim_{x\to\infty} v(x)=\frac{N_1}{N_1-N_3}\cdot\lim_{x\to\infty} \frac{K_3\gamma_{2+}e^{\gamma_{2+}x}+\gamma_{2-}e^{\gamma_{2-}x}}{K_3e^{\gamma_{2+}x}+e^{\gamma_{2-}x}} = \frac{N_1}{N_1-N_3}\gamma_{2+}.
    \end{equation}
     Moreover, $\gamma_{2+}^2+\frac{N_2}{N_1}\gamma_{2+}-\left(1-\frac{N_1}{N_3}\right)\frac{\delta}{N_1}=0$ since $\gamma_{2+}$ is a solution to the characteristic polynomial associated with \eqref{eqn: ode for u}. We then have
    \begin{align*}
        \lim_{x\to\infty}[v'(x)+v(x)^2]
        &=\frac{1}{N_1}\lim_{x\to\infty}\left[N_3v(x)^2-N_2v(x)+\delta\right]\\
        &=\frac{1}{N_1}\left[\frac{N_3N_1^2}{(N_1-N_3)^2}\gamma_{2+}^2-\frac{N_2N_1}{N_1-N_3}\gamma_{2+}+\delta\right]\\
        &=\frac{N_3N_1}{(N_1-N_3)^2}\left[\gamma_{2+}^2-\frac{N_2}{N_1}\left(\frac{N_1}{N_3}-1\right)\gamma_{2+}+\frac{\delta(N_1-N_3)}{N_1^2}\left(\frac{N_1}{N_3}-1\right)\right]\\
        &=\frac{N_3N_1}{(N_1-N_3)^2}\left[\left(\gamma_{2+}^2+\frac{N_2}{N_1}\gamma_{2+}-\left(1-\frac{N_1}{N_3}\right)\frac{\delta}{N_1}\right)-\frac{N_2}{N_3}\gamma_{2+}+\frac{\delta(N_1-N_3)}{N_1N_3}\right]\\
        &=\frac{N_3N_1}{(N_1-N_3)^2}\left[0-\frac{N_2}{N_3}\gamma_{2+}+\frac{\delta(N_1-N_3)}{N_1N_3}\right]\\
        &=\frac{N_1}{(N_1-N_3)^2}\left[-N_2\gamma_{2+}+\left(1-\frac{N_3}{N_1}\right)\delta\right].
    \end{align*}
    Suppose $-N_2\gamma_{2+}+\left(1-\frac{N_3}{N_1}\right)\delta\leq 0$. It is equivalent to the following inequality:
    \begin{equation}\label{eqn:r1 ineq}
        N_2^2+2\delta (N_1-N_3) \leq  N_2\sqrt{N_2^2+4\delta (N_1-N_3)}.
    \end{equation}
    If $N_1>N_3$, then $N_2^2+2\delta (N_1-N_3)>0$. If $N_1<N_3$, then $N_2^2+2\delta (N_1-N_3)>N_2^2+4\delta (N_1-N_3)>0$. Squaring both sides of \eqref{eqn:r1 ineq} then yields $4\delta^2(N_1-N_3)^2\leq 0,$
    which is a contradiction. Hence, $-N_2\gamma_{2+}+\left(1-\frac{N_3}{N_1}\right)\delta>0$, and, consequently, $\lim_{x\to\infty}[v'(x)+v(x)^2] >0.$
    This contradicts \eqref{eqn: v' + v^2<0}, which proves the existence of $b^*$.
\end{proof}
Thus, we have obtained the form of the value function $g$ defined in \eqref{eqn: g thm 1}. The corresponding optimal reinsurance and distortion levels are given by \eqref{eqn: pi-theta thm 1}.

\begin{remark}\label{remark on g' and g'' thm 1}
    The first derivative and second derivative of $g$ are given as follows:
    \small\begin{equation}
        \begin{aligned}
            g'(x)&=
    \begin{cases}
    \frac{a_2\gamma_1}{v(b^*)e^{\int_{w_0}^{b^*}v(y) \dd y}}\left(\frac{x}{w_0}\right)^{\gamma_1-1} & \mbox{if $x<w_0$,}\\
    \frac{a_2v(x)}{v(b^*)e^{\int_{w_0}^{b^*}v(y) \dd y}} e^{\int_{w_0}^x v(y) \dd y} & \mbox{if $w_0\leq x<b^*$,}\\
    a_2 & \mbox{if $x\geq b^*$,}
    \end{cases}\quad
    g''(x)&=
    \begin{cases}
    \frac{a_2\gamma_1(\gamma_1-1)}{v(b^*)e^{\int_{w_0}^{b^*}v(y) \dd y}}\left(\frac{x}{w_0}\right)^{\gamma_1-2} & \mbox{if $x<w_0$,}\\
    \frac{a_2[v(x)^2+v'(x)]}{v(b^*)e^{\int_{w_0}^{b^*}v(y) \dd y}} e^{\int_{w_0}^x v(y) \dd y} & \mbox{if $w_0\leq x<b^*$,}\\
    0 & \mbox{if $x\geq b^*$.}
    \end{cases}
        \end{aligned}
    \end{equation}\normalsize
    From \eqref{eqn: v' + v^2<0}, we can say that $v'(x) < 0$ for $x>w_0$, which implies that $v$ is decreasing for $x>w_0$. Moreover, from \eqref{eqn: limit of v}, we have $\lim_{x\to\infty} v(x) >0$, which then implies that $v(b^*)>0$. It is then clear that $g'(x)>0$ for $x>0$, which implies that $g$ is strictly increasing for $x>0$. Moreover, since $\gamma_1\in(0,1)$ and $v(x)^2+v'(x)<0$ in the region $\{w_0 < x < b^*\}$, it holds that $g''(x) \leq 0 $ for $x>0$, which implies that $g$ is concave for $x>0$. 
\end{remark}

\subsection{Proof of Theorem \ref{theorem 2}}

In this section, we present the key results used to obtain Theorem \ref{theorem 2}. The discussion serves as the proof for Theorem \ref{theorem 2}.

Suppose $\gamma_1\in(0,1)$ exists, \eqref{eqn: correlation bounds} holds, and $N_1= N_3$, as required in Theorem \ref{theorem 2}. Since $\gamma_1$ exists, then $\delta < \frac{\mu_1^2}{2\beta_1\sigma_1^2} + \frac{\mu_2^2}{2\beta_2\sigma_2^2}$. In the region $\{x < w_0\}$, we obtain \eqref{eqn:hjb x<w0} and still obtain its solution $g_1(x)=K_1x^{\gamma_1}$, where $\gamma_1$ is a solution of $\psi(z)=0$ on $(0,1)$ and $K_1>0$ is an unknown constant. In the region $\{w_0<x<b^*\}$, we obtain \eqref{eqn: ode 2}. Using the same definition of $v$ in \eqref{eqn: defn of v} yields \eqref{eqn: ode for v}. Since $N_1=N_3$, \eqref{eqn: ode for v} has a solution given by 
\begin{equation}
\begin{aligned}
    v(x)
    &=e^{-\frac{N_2}{N_1}x}\left[\int_{w_0}^x\frac{\delta}{N_1}e^{\frac{N_2}{N_1}z} \dd z + K_3\right]=\frac{\delta}{N_2}\left[1-e^{-\frac{N_2}{N_1}(x-w_0)}\right]+K_3e^{-\frac{N_2}{N_1}x}.
\end{aligned}
\end{equation}
We can then conjecture the same solution given by \eqref{eqn: g conjecture}.

Similar to the previous section, we require $g$, $g'$, and $g''$ to be continuous at the switching points $w_0$ and $b^*$. At $x=w_0$, we have the system of equations obtained in \eqref{eqn: system at w0}, which also yields \eqref{eqn: K1 in terms of K2} and \eqref{eqn: v(w0) for K3}. Using \eqref{eqn: v(w0) for K3}, we obtain $K_3=\frac{\gamma_1}{w_0}e^{\frac{N_2}{N_1}w_0}.$
In the neighborhood of $b^*$, we still obtain the system of equations in \eqref{eqn: system at b*}, which yields \eqref{eqn: K2 and K4 from system} and \eqref{eqn: v2+v' =0 for b*}. We then have the following lemma:
\begin{lemma}
    Suppose that $\delta < \frac{\mu_1^2}{2\beta_1\sigma_1^2} + \frac{\mu_2^2}{2\beta_2\sigma_2^2}$ and $N_1= N_3$. Then, $b^*$ exists and satisfies \eqref{eqn: b^* theorem 2}.
\end{lemma}
\begin{proof}
    From \eqref{eqn: ode for v}, we have $v'(x)+v(x)^2=\frac{\delta}{N_1}-\frac{N_2}{N_1}v(x) +v(x)^2.$
    Similar to Lemma \ref{lemma on u1 exist case 1}, we obtain $v'(w_0)+v(w_0)^2<0.$
    Suppose $b^*$ does not exist. Since $v(x)^2+v'(x)$ is continuous in $x$, we have $v'(x)+v(x)^2 <0$ for $x>w_0$,
    which implies that $0> v'(x)= \frac{1}{N_1}\left(\delta-N_2v(x)\right).$
    Hence, we obtain $v(x) > \frac{\delta}{N_2}$ and 
    \begin{align}
        v'(x)+v(x)^2 
        & > \frac{\delta}{N_1}-\frac{N_2}{N_1}v(x)+\frac{\delta^2}{N_2^2}\\
        & = \frac{\delta}{N_1}-\frac{N_2}{N_1}\left[\frac{\delta}{N_2}\left[1-e^{-\frac{N_2}{N_1}(x-w_0)}\right]+\frac{\gamma_1}{w_0}e^{-\frac{N_2}{N_1}(x-w_0)}\right]+\frac{\delta^2}{N_2^2}\\
        & = \frac{\delta^2}{N_2^2}-\frac{N_2}{N_1}\left(\frac{\gamma_1}{w_0}-\frac{\delta}{N_2}\right)e^{-\frac{N_2}{N_1}(x-w_0)}.
        \end{align}
    Thus, $\lim_{x\to\infty}[v'(x)+v(x)^2]= \frac{\delta^2}{N_2^2}>0,$
    which is a contradiction. This proves the result.
\end{proof}
Thus, we have obtained the form of the value function $g$ defined in \eqref{eqn: g thm 2}. The corresponding optimal reinsurance and distortion levels are of the same form as in \eqref{eqn: pi-theta thm 1}. The discussion on $g$ defined in \eqref{eqn: g thm 2} being increasing and concave is similar to that in Remark \ref{remark on g' and g'' thm 1}.

\subsection{Proof of Theorem \ref{theorem 3}}

In this section, we present the key results used to obtain Theorem \ref{theorem 3}. The discussion serves as the proof for Theorem \ref{theorem 3}. 

Suppose that $\psi(z)=0$ does not have a solution, or, equivalently, $\delta \geq \frac{\mu_1^2}{2\beta_1\sigma_1^2} + \frac{\mu_2^2}{2\beta_2\sigma_2^2}$. We conjecture that $g(x)=a_2x$ satisfies the HJB equation \eqref{eqn:hjborig}. It is clear that $g'(x) = a_2$ for all $x$. Hence, $a_1 \leq g'(x)$ and $a_2\leq g'(x)$ are immediately satisfied. Moreover, by the definition of \eqref{eqn: defn of b*}, we have $b^*=0$.

Write $\tilde w_0:=\min\{w_1,w_2\}$. In the region $\{x<\tilde w_0\}$, we have
\begin{equation}
    \begin{aligned}
        \sup_{\pi_i\in[0,1]}\inf_{\theta_i\in\mathbb R}\left[\Lc^{\vartheta}(g)(x)+g(x)\sum_{i=1}^2\frac{\theta_i^2}{2\beta_i}\right]&=\frac{A(x)}{2\sigma_1^2\sigma_2^2B(x)}-\delta g(x)\\
        &=\frac{\beta_2\mu_1^2\sigma_2^2+\beta_1\mu_2^2\sigma_1^2-2\delta\beta_1\beta_2\sigma_1^2\sigma_2^2}{2\beta_1\beta_2\sigma_1^2\sigma_2^2}\leq 0.
    \end{aligned}
\end{equation}
In the region $\{x > \tilde w_0\}$, we have, by Lemma \ref{lemma: pi remain constant}, $(\bar\pi_1,\bar\pi_2)(x)=\left(\frac{\tilde w_0}{w_1},\frac{\tilde w_0}{w_2}\right)$. Then,
    \begin{align}
        \sup_{\pi_i\in[0,1]}\inf_{\theta_i\in\mathbb R}\left[\Lc^{\vartheta}(g)(x)+g(x)\sum_{i=1}^2\frac{\theta_i^2}{2\beta_i}\right]
        &=N_1g''(x)+N_2g'(x)-N_3\frac{g'(x)^2}{g(x)}-\delta g(x)\\
        &=a_2\left(N_2-\frac{N_3}{x}-\delta x\right)\leq a_2\left(N_2-2\sqrt{\delta N_3}\right)\\
        &\leq a_2\left(N_2-2\sqrt{N_3\left(\frac{\mu_1^2}{2\beta_1\sigma_1^2}+\frac{\mu_2^2}{2\beta_2\sigma_2^2}\right)}\right)\\
        &= a_2\left(N_2-\sqrt{\mu_1^2\frac{\tilde w_0^2}{w_1^2}+\mu_2^2\frac{\tilde w_0^2}{w_2^2}+\frac{\beta_2\sigma_2^2\tilde w_0^2}{\beta_1\sigma_1^2w_2^2}\mu_1^2+\frac{\beta_1\sigma_1^2\tilde w_0^2}{\beta_2\sigma_2^2w_1^2}\mu_2^2}\right)\\
        &\leq a_2\left(N_2-\sqrt{\mu_1^2\frac{\tilde w_0^2}{w_1^2}+\mu_2^2\frac{\tilde w_0^2}{w_2^2}+2\mu_1\mu_2\frac{\tilde w_0^2}{w_1w_2}}\right)=0,
    \end{align}
where the first and third inequalities follow by the arithmetic mean-geometric mean (AM-GM) inequality while the second inequality follows by $\beta_2\mu_1^2\sigma_2^2+\beta_1\mu_2^2\sigma_1^2-2\delta\beta_1\beta_2\sigma_1^2\sigma_2^2\leq 0$.

\section{Conclusion}\label{sec:conclusion}

We study an optimal dividend payout, reinsurance, and capital injection problem for an insurer with collaborating business lines under model uncertainty. By incorporating ambiguity and a relative entropy-based penalty, we characterize the optimal value function and the corresponding optimal strategies in closed form across several parameter regimes. Our results show that the optimal dividend-capital injection strategy is of barrier type, while the optimal proportional reinsurance coverage and the deviation of the worst-case model from the reference model are decreasing with respect to the aggregate reserve level. The numerical illustrations further show how ambiguity aversion influences the optimal strategies.

{
\setstretch{1}
\bibliographystyle{apalike} 
\bibliography{References}
}	
%
%

\appendix

\section{Proof of Theorem \ref{thm: verification}}\label{app: verification proof}
    \emph{Part I:} Since $C_i$ and $L_i$ are c\`{a}dl\`{a}g processes, we have $dC_i(t)=dC_i^{c}(t)+C_i(t)-C_i(t-)$ and $dL_i(t)=dL_i^{c}(t)+L_i(t)-L_i(t-)$,
    where $C_i^c(t)$ and $L_i^c(t)$ represent the continuous parts of $C_i(t)$ and $L_i(t)$, respectively. Write $(\pi,\theta):=(\pi_1,\pi_2,\theta_1,\theta_2)$. For any finite time $t>0$, using It\^{o}'s formula to $e^{-\delta(t\wedge\tau^u)}w(X^u(t\wedge \tau^u))$ yields
    \begin{align}
        &e^{-\delta(t\wedge\tau^u)}w(X^u(t\wedge \tau^u))\\
        &=w(x)+\int_0^{t\wedge\tau^u} e^{-\delta s}\Lc^{(\pi(s-),\theta(s-))}(w)(X^u(s-)) \dd s-\sum_{i=1}^2 \int_0^{t\wedge\tau^u}e^{-\delta s}\sigma_i\pi_i(s-)w'(X^u(s-)) \dd W^{\mathbb{Q}^{\theta}}_i(s)\\
        &\qquad - \sum_{i=1}^2 \int_0^{t\wedge\tau^u}e^{-\delta s}w'(X^u(s-)) \dd C^c_i(s) + \sum_{i=1}^2 \int_0^{t\wedge\tau^u}e^{-\delta s}w'(X^u(s-)) [\dd L^c_i(s) - \dd L^c_{3-i}(s)] \\
        &\qquad + \sum_{0\leq s \leq t\wedge \tau^u}e^{-\delta s} \left(w(X^u(s))-w(X^u(s-))\right)\\
        &=w(x)+\int_0^{t\wedge\tau^u} e^{-\delta s}\Lc^{(\pi(s-),\theta(s-))}(w)(X^u(s-)) \dd s-\sum_{i=1}^2 \int_0^{t\wedge\tau^u}e^{-\delta s}\sigma_i\pi_i(s-)w'(X^u(s-)) \dd W^{\mathbb{Q}^{\theta}}_i(s)\\
        &\qquad - \sum_{i=1}^2 \int_0^{t\wedge\tau^u}e^{-\delta s}w'(X^u(s-)) \dd C^c_i(s) + \sum_{0\leq s \leq t\wedge \tau^u}e^{-\delta s} \left(w(X^u(s))-w(X^u(s-))\right), 
    \end{align}
    where $\theta_1$ and $\theta_2$ satisfy the Novikov condition. Since $w(x)$ is concave, $w'(x) \geq a_2$, and $\pi_i(s-)\in[0,1]$, it holds that the process $\left\{\sum_{i=1}^2 \int_0^{t\wedge\tau^u}e^{-\delta s}\sigma_i\pi_i(s-)w'(X^u(s-)) \dd W^{\mathbb{Q}^{\theta}}_i(s)\right\}_{t\geq 0}$ is a true martingale under the alternative probability measure $\mathbb{Q}^{\theta}$. Taking expectations yields 
    \begin{align}
        &\mathbb E^{\mathbb Q^{\theta}}\left[e^{-\delta(t\wedge\tau^u)}w(X^u(t\wedge \tau^u))\right]\\
        &=w(x)+\mathbb E^{\mathbb Q^{\theta}}\Bigg[\int_0^{t\wedge\tau^u} e^{-\delta s}\Lc^{(\pi(s-),\theta(s-))}(w)(X^u(s-)) \dd s - \sum_{i=1}^2 \int_0^{t\wedge\tau^u}e^{-\delta s}w'(X^u(s-)) \dd C^c_i(s) \\
        &\qquad + \sum_{0\leq s \leq t\wedge \tau^u}e^{-\delta s} \left(w(X^u(s))-w(X^u(s-))\right)\Bigg]. \label{proof eq 14}
    \end{align}
    Since $w'(x)\geq a_2$ and $X^u(s-)\geq X^u(s)$, we have
    \begin{align}
        w(X^u(s-))-w(X^u(s)) &\geq a_2 \left(X^u(s-)-X^u(s)\right)\\
        &=a_2\left(X_1^u(s-)+X_2^u(s-)-X_1^u(s)-X_2^u(s)\right) \\
        &=a_2\left( C_1(s) + C_2(s) - C_1(s-) - C_2(s-) \right).
    \end{align}
    Since $C_1$ and $C_2$ are nondecreasing, we have $\dd C^c_1(s) , \dd C^c_2(s) \geq 0$. Hence,
    \begin{align}
        &\mathbb E^{\mathbb Q^{\theta}}\left[\sum_{i=1}^2 \int_0^{t\wedge\tau^u}e^{-\delta s}w'(X^u(s-)) \dd C^c_i(s) - \sum_{0\leq s \leq t\wedge \tau^u}e^{-\delta s} \left(w(X^u(s))-w(X^u(s-))\right)\right]\\
        &\geq \mathbb E^{\mathbb Q^{\theta}}\left[a_2\sum_{i=1}^2 \int_0^{t\wedge\tau^u}e^{-\delta s}  \dd C^c_i(s) + a_2\sum_{0\leq s \leq t\wedge \tau^u}e^{-\delta s} \left(C_1(s) + C_2(s) - C_1(s-) - C_2(s-)\right)\right]\\
        &= \mathbb E^{\mathbb Q^{\theta}}\left[a_2\sum_{i=1}^2 \int_0^{t\wedge\tau^u}e^{-\delta s}  \dd C_i(s) \right]\geq  \mathbb E^{\mathbb Q^{\theta}}\left[\sum_{i=1}^2 a_i \int_0^{t\wedge\tau^u}e^{-\delta s}  \dd C_i(s) \right], \label{proof eq 16}
    \end{align}
    where the last inequality is due to $a_1\leq a_2$. Combining \eqref{proof eq 14} and \eqref{proof eq 16} and replacing $\theta$ with $\theta^*$ yield
    \small
    \begin{align}
        &w(x)\\
        &\geq \mathbb E^{\mathbb Q^{\theta^*}} \left[e^{-\delta(t\wedge\tau^u)}w(X^u(t\wedge \tau^u))-\int_0^{t\wedge\tau^u} e^{-\delta s}\Lc^{(\pi(s-),\theta^*(s-))}(w)(X^u(s-)) \dd s + \sum_{i=1}^2 a_i \int_0^{t\wedge\tau^u}e^{-\delta s}  \dd C_i(s)\right] \\
        &\geq \mathbb E^{\mathbb Q^{\theta^*}} \Bigg[e^{-\delta(t\wedge\tau^u)}w(X^u(t\wedge \tau^u))-\int_0^{t\wedge\tau^u} e^{-\delta s}\left(\Lc^{(\pi(s-),\theta^*(s-))}(w)(X^u(s-))+\sum_{i=1}^2a_i\frac{(\theta_i^*(s-))^2}{2\widetilde \beta_i}w(X^u(s-)) \right)\dd s\\
        &\qquad\qquad+ \sum_{i=1}^2 a_i \left(\int_0^{t\wedge\tau^u}e^{-\delta s}  \dd C_i(s)+\int_0^{t\wedge \tau^u}e^{-\delta s}\frac{(\theta_i^*(s-))^2}{2\widetilde \beta_i}w(X^u(s-)) \dd s\right)\Bigg] \\
        &\geq  \mathbb E^{\mathbb Q^{\theta^*}}\left[\sum_{i=1}^2 a_i \left(\int_0^{t\wedge\tau^u}e^{-\delta s}  \dd C_i(s)+\int_0^{t\wedge \tau^u}e^{-\delta s}\frac{(\theta_i^*(s-))^2}{2\widetilde \beta_i}w(X^u(s-)) \dd s\right)\right],
    \end{align}\normalsize
    where the last inequality is due to \eqref{proof eqn 9} and the nonnegativity of $e^{-\delta(t\wedge\tau^u)}w(X^u(t\wedge \tau^u))$. Letting $t\to\infty$ and using the monotone convergence theorem yield
    \begin{align}
        w(x)
        &\geq  \mathbb E^{\mathbb Q^{\theta^*}}\left[\sum_{i=1}^2 a_i \left(\int_0^{\tau^u}e^{-\delta s}  \dd C_i(s)+\int_0^{\tau^u}e^{-\delta s}\frac{(\theta_i^*(s-))^2}{2\widetilde \beta_i}w(X^u(s-)) \dd s\right)\right]\label{proof eqn 19}\\
        & \geq \inf_{\theta\in\Theta} \mathbb E^{\mathbb Q^{\theta^*}}\left[\sum_{i=1}^2 a_i \left(\int_0^{\tau^u}e^{-\delta s}  \dd C_i(s)+\int_0^{\tau^u}e^{-\delta s}\frac{(\theta_i(s-))^2}{2\widetilde \beta_i}w(X^u(s-)) \dd s\right)\right].
    \end{align}
    Since the above inequalities hold for any admissible $u$, we have
    \begin{equation}\label{proof eqn 21}
        w(x) \geq \sup_{u\in\mathcal U}\inf_{\theta\in\Theta} \mathbb E^{\mathbb Q^{\theta^*}}\left[\sum_{i=1}^2 a_i \left(\int_0^{\tau^u}e^{-\delta s}  \dd C_i(s)+\int_0^{\tau^u}e^{-\delta s}\frac{(\theta_i(s-))^2}{2\widetilde \beta_i}w(X^u(s-)) \dd s\right)\right],\quad x\geq 0.
    \end{equation}

    \emph{Part II}: Consider the strategy $u^*(X^{u^*}(t))=(\pi_1^*(X^{u^*}(t)),\pi_2^*(X^{u^*}(t)),C_{1,b^*},C_{2,b^*},L_{1,b^*},L_{2,b^*})$. Under the barrier strategy $(C_{1,b^*},C_{2,b^*},L_{1,b^*},L_{2,b^*})$, at time $t=0$, Line 2 pays the amount $X^{u^{*}}(0-) - b^*$ as dividends if and only if $X^{u^{*}}(0-) > b^*$. Since Line 1 does not pay dividends, it holds that $C_{1,b^*}(t)=0$ for all $t\geq 0$. At any time $t>0$, the aggregate reserve will not exceed $b^*$ and there will be no dividend payments. The dividend payout is made as continuous payments at time $t>0$ if and only if $X^{u^*}(t-)=b^*$; that is, $\dd C_{i,b^*}(t)=\mathds 1_{\{X^{u^*}(t-)=b^*\}} \dd C_{i,b^*}(t)$ and $C_{i,b^*}(t)-C_{i,b^*}(t-)=0$ for all $t>0$. 

    It follows that for any $x\in[0,b^*]$, given that $X^{u^*}(0-)=x$, $X^{u^*}(t)$ is always continuous until ruin time. For $x\in[0,b^*]$ and $t>0$, applying It\^{o}'s lemma to $e^{-\delta(t\wedge \tau^{u^*})}w(X^{u^*}(t\wedge \tau^{u^*}))$ and taking expectations yield
   \begin{align}
        &\mathbb E^{\mathbb Q^{\theta}} \left[e^{-\delta(t\wedge \tau^{u^*})}w(X^{u^*}(t\wedge \tau^{u^*}))\right]\\
        &=w(x) + \mathbb E^{\mathbb Q^{\theta}} \left[ \int_0^{t\wedge \tau^{u^*}} e^{-\delta s} \Lc ^{(\pi^*(s-),\theta(s-))}(w)(X^{u^*}(s-)) \dd s - \sum_{i=1}^2 \int_0^{t\wedge \tau^{u^*}} e^{-\delta s} w'(X^{u^*}(s-)) \dd C^{c}_{i,b^*}(s) \right],
        \label{proof eqn 26}
    \end{align}
    where $C^{c}_{i,b^*}(t)$ is the continuous part of $C_{i,b^*}(t)$. Moreover, for $x\in[0,b^*]$, we have
    \begin{align}
        \mathbb E^{\mathbb Q^{\theta}} \left[\sum_{i=1}^2 \int_0^{t\wedge \tau^{u^*}} e^{-\delta s} w'(X^{u^*}(s-)) \dd C^{c}_{i,b^*}(s) \right]
        &=\mathbb E^{\mathbb Q^{\theta}} \left[\sum_{i=1}^2 \int_0^{t\wedge \tau^{u^*}} e^{-\delta s} w'(b^*) \mathds 1_{\{X^{u^*}(s-)=b^*\}} \dd C^{c}_{i,b^*}(s) \right]\\
        &=\mathbb E^{\mathbb Q^{\theta}} \left[\sum_{i=1}^2 \int_0^{t\wedge \tau^{u^*}} e^{-\delta s} a_2  \dd C^{c}_{i,b^*}(s) \right]\\
        &=\mathbb E^{\mathbb Q^{\theta}} \left[\sum_{i=1}^2 \int_0^{t\wedge \tau^{u^*}} e^{-\delta s} a_2  \dd C_{i,b^*}(s) \right]\\
        &=\mathbb E^{\mathbb Q^{\theta}} \left[\sum_{i=1}^2 \int_0^{t\wedge \tau^{u^*}} e^{-\delta s} a_i  \dd C_{i,b^*}(s) \right].\label{proof eqn 27}
    \end{align}
    Substituting \eqref{proof eqn 27} into \eqref{proof eqn 26} and using \eqref{proof eqn 10} yield
    \small\begin{align}
        w(x)&=\mathbb E^{\mathbb Q^{\theta}} \Bigg[e^{-\delta(t\wedge \tau^{u^*})}w(X^{u^*}(t\wedge \tau^{u^*}))-\int_0^{t\wedge \tau^{u^*}}e^{-\delta s}\left(\Lc ^{(\pi^*(s-),\theta(s-))}(w)(X^{u^*}(s-))+\sum_{i=1}^2a_i\frac{(\theta_i(s-))^2}{2\widetilde \beta_i}\right)\dd s\\
        &\qquad\qquad + \sum_{i=1}^2 a_i \left(\int_0^{t\wedge\tau^{u^*}}e^{-\delta s}  \dd C_{i,b^*}(s)+\int_0^{t\wedge \tau^{u^*}}e^{-\delta s}\frac{(\theta_i(s-))^2}{2\widetilde \beta_i}w(X^{u^*}(s-)) \dd s\right) \Bigg]\\
        &\leq \mathbb E^{\mathbb Q^{\theta}} \Bigg[e^{-\delta(t\wedge \tau^{u^*})}w(X^{u^*}(t\wedge \tau^{u^*}))\\
        &\qquad\qquad + \sum_{i=1}^2 a_i \left(\int_0^{t\wedge\tau^{u^*}}e^{-\delta s}  \dd C_{i,b^*}(s)+\int_0^{t\wedge \tau^{u^*}}e^{-\delta s}\frac{(\theta_i(s-))^2}{2\widetilde \beta_i}w(X^{u^*}(s-)) \dd s\right) \Bigg]\label{proof eqn 28}.
    \end{align}\normalsize
    Using the dominated convergence theorem and $X^{u^*}(\tau^{u^*})=0$ yields 
    \begin{equation}
        \lim_{t\to\infty} \mathbb E^{\mathbb Q^{\theta}} \left[e^{-\delta(t\wedge \tau^{u^*})}w(X^{u^*}(t\wedge \tau^{u^*}))\right]=\mathbb E^{\mathbb Q^{\theta}} \left[e^{-\delta\tau^{u^*}}w(X^{u^*}(\tau^{u^*}))\right]=0.
    \end{equation}
    Using the monotone convergence theorem on \eqref{proof eqn 28} yields
    \begin{align}\label{proof eqn 29}
        w(x)\leq \mathbb E^{\mathbb Q^{\theta}}\left[\sum_{i=1}^2 a_i \left(\int_0^{\tau^{u^*}}e^{-\delta s}  \dd C_{i,b^*}(s)+\int_0^{\tau^{u^*}}e^{-\delta s}\frac{(\theta_i(s-))^2}{2\widetilde \beta_i}w(X^{u^*}(s-)) \dd s\right)\right].
    \end{align}
    Due to the arbitrariness of $\theta\in\Theta$, we have, for $x\in[0,b^*]$,
    \begin{align}
        w(x)&\leq \inf_{\theta \in \Theta }\mathbb E^{\mathbb Q^{\theta}}\left[\sum_{i=1}^2 a_i \left(\int_0^{\tau^{u^*}}e^{-\delta s}  \dd C_{i,b^*}(s)+\int_0^{\tau^{u^*}}e^{-\delta s}\frac{(\theta_i(s-))^2}{2\widetilde \beta_i}w(X^{u^*}(s-)) \dd s\right)\right]\\
        &\leq \sup_{u\in \Uc}\inf_{\theta \in \Theta }\mathbb E^{\mathbb Q^{\theta}}\left[\sum_{i=1}^2 a_i \left(\int_0^{\tau^{u}}e^{-\delta s}  \dd C_{i}(s)+\int_0^{\tau^{u}}e^{-\delta s}\frac{(\theta_i(s-))^2}{2\widetilde \beta_i}w(X^{u}(s-)) \dd s\right)\right]\label{proof eqn 30}.
    \end{align}

    Recall that $w'(x)=a_2$ for $x\geq b^*$. By the continuity of $w(x)$ and \eqref{proof eqn 29}, we obtain
    \begin{align}
        w(x)
        &=w(b^*) + a_2 (x-b^*)\\
        &\leq \mathbb E_{b^*}^{\mathbb Q^{\theta}}\left[\sum_{i=1}^2 a_i \left(\int_0^{\tau^{u^*}}e^{-\delta s}  \dd C_{i,b^*}(s)+\int_0^{\tau^{u^*}}e^{-\delta s}\frac{(\theta_i(s-))^2}{2\widetilde \beta_i}w(X^{u^*}(s-)) \dd s\right)\right] + a_2 (x-b^*) \\
        \label{proof eqn 31}\\
        &= \mathbb E^{\mathbb Q^{\theta}}\left[\sum_{i=1}^2 a_i \left(\int_0^{\tau^{u^*}}e^{-\delta s}  \dd C_{i,b^*}(s)+\int_0^{\tau^{u^*}}e^{-\delta s}\frac{(\theta_i(s-))^2}{2\widetilde \beta_i}w(X^{u^*}(s-)) \dd s\right)\right], \label{proof eqn 32}
    \end{align}
    where $\mathbb E_{b^*}^{\mathbb Q^{\theta}}[\cdot]$ is the conditional expectation given $X^{u^*}(0)=x_1+x_2=b^*$. Similar to \eqref{proof eqn 30}, we have, for $x>b^*$,
    \begin{equation}\label{proof eqn 33}
        w(x)\leq \sup_{u\in \Uc} \inf_{\theta \in \Theta} \mathbb E^{\mathbb Q^{\theta}} \left[\sum_{i=1}^2 a_i \left(\int_0^{\tau^{u}}e^{-\delta s}  \dd C_{i}(s)+\int_0^{\tau^{u}}e^{-\delta s}\frac{(\theta_i(s-))^2}{2\widetilde \beta_i}w(X^{u}(s-)) \dd s\right) \right].
    \end{equation}
    Combining \eqref{proof eqn 21} with \eqref{proof eqn 30} and \eqref{proof eqn 33} yields
    \begin{equation}
        w(x) =\sup_{u\in \Uc} \inf_{\theta \in \Theta} \mathbb E^{\mathbb Q^{\theta}} \left[\sum_{i=1}^2 a_i \left(\int_0^{\tau^{u}}e^{-\delta s}  \dd C_{i}(s)+\int_0^{\tau^{u}}e^{-\delta s}\frac{(\theta_i(s-))^2}{2\widetilde \beta_i}w(X^{u}(s-)) \dd s\right) \right],\quad x\geq 0.
    \end{equation}

    \emph{Part III:} Replacing $\theta$ in \eqref{proof eqn 28}, \eqref{proof eqn 29}, \eqref{proof eqn 31}, and \eqref{proof eqn 32} with $\theta^*$ and using \eqref{proof eqn 11} make \eqref{proof eqn 28}, \eqref{proof eqn 29}, \eqref{proof eqn 31}, and \eqref{proof eqn 32} equalities. Moreover, we have
    \begin{align}
        w(x) = \mathbb E^{\mathbb Q^{\theta^*}}\left[\sum_{i=1}^2 a_i \left(\int_0^{\tau^{u^*}}e^{-\delta s}  \dd C_{i,b^*}(s)+\int_0^{\tau^{u^*}}e^{-\delta s}\frac{(\theta^*_i(s-))^2}{2\widetilde \beta_i}w(X^{u^*}(s-)) \dd s\right)\right], \quad x\geq 0.
    \end{align}
    Since \eqref{proof eqn 19} holds for any admissible $u$, it follows that
    \begin{align}
        w(x)&\geq \sup_{u\in \Uc} \mathbb E^{\mathbb Q^{\theta^*}}\left[\sum_{i=1}^2 a_i \left(\int_0^{\tau^u}e^{-\delta s}  \dd C_i(s)+\int_0^{\tau^u}e^{-\delta s}\frac{(\theta_i^*(s-))^2}{2\widetilde \beta_i}w(X^u(s-)) \dd s\right)\right]\\
        &\geq \inf_{\theta\in\Theta}\sup_{u\in \Uc} \mathbb E^{\mathbb Q^{\theta}}\left[\sum_{i=1}^2 a_i \left(\int_0^{\tau^u}e^{-\delta s}  \dd C_i(s)+\int_0^{\tau^u}e^{-\delta s}\frac{(\theta_i(s-))^2}{2\widetilde \beta_i}w(X^u(s-)) \dd s\right)\right],\label{proof eqn 35}
    \end{align} for $x\geq 0$.
    Combining \eqref{proof eqn 35} with \eqref{proof eqn 30} and \eqref{proof eqn 33} leads to
    \small\begin{align}
        &\inf_{\theta\in\Theta}\sup_{u\in \Uc} \mathbb E^{\mathbb Q^{\theta}}\left[\sum_{i=1}^2 a_i \left(\int_0^{\tau^u}e^{-\delta s}  \dd C_i(s)+\int_0^{\tau^u}e^{-\delta s}\frac{(\theta_i(s-))^2}{2\widetilde \beta_i}w(X^u(s-)) \dd s\right)\right] \\
        &\leq w(x)\leq  \sup_{u\in \Uc}\inf_{\theta\in\Theta}\mathbb E^{\mathbb Q^{\theta}}\left[\sum_{i=1}^2 a_i \left(\int_0^{\tau^u}e^{-\delta s}  \dd C_i(s)+\int_0^{\tau^u}e^{-\delta s}\frac{(\theta_i(s-))^2}{2\widetilde \beta_i}w(X^u(s-)) \dd s\right)\right].
    \end{align}\normalsize
    Since it always holds that
    \begin{align}
        &\sup_{u\in \Uc}\inf_{\theta\in\Theta}\mathbb E^{\mathbb Q^{\theta}}\left[\sum_{i=1}^2 a_i \left(\int_0^{\tau^u}e^{-\delta s}  \dd C_i(s)+\int_0^{\tau^u}e^{-\delta s}\frac{(\theta_i(s-))^2}{2\widetilde \beta_i}w(X^u(s-)) \dd s\right)\right] \\
        &\leq \inf_{\theta\in\Theta}\sup_{u\in \Uc} \mathbb E^{\mathbb Q^{\theta}}\left[\sum_{i=1}^2 a_i \left(\int_0^{\tau^u}e^{-\delta s}  \dd C_i(s)+\int_0^{\tau^u}e^{-\delta s}\frac{(\theta_i(s-))^2}{2\widetilde \beta_i}w(X^u(s-)) \dd s\right)\right],
    \end{align}
    the result is proved.

\section{Proof of Theorem \ref{theorem no uncertainty}}\label{app: theorem no model unc}

Suppose $0<\rho\leq  \frac{\mu_1/\sigma_1}{\mu_2/\sigma_2}\leq\frac{1}{\rho}$ or $-1< \rho \leq 0$. The optimization over $\pi_i$ (reinsurance decision) can then be simplified to:
\begin{align}
        \sup_{\pi_1\in[0,1],\pi_2\in[0,1]} \left\{\left(\frac{1}{2}\sigma_1^2\pi_1^2+\frac{1}{2}\sigma_2^2\pi_2^2+\rho\sigma_1\sigma_2\pi_1\pi_2\right)g''(x) + \left(\mu_1\pi_1+\mu_2\pi_2\right)g'(x)\right\},
\end{align}
and, ignoring the constraints over $[0,1]$, we obtain the following candidate maximizers
\begin{align}
    \widehat \pi_1(x)=-\frac{\mu_1\sigma_2-\rho\mu_2\sigma_1}{\sigma_1^2\sigma_2(1-\rho^2)}\cdot\frac{g'(x)}{g''(x)}\quad\mbox{and}\quad \widehat \pi_2(x)=-\frac{\mu_2\sigma_1-\rho\mu_1\sigma_2}{\sigma_1\sigma_2^2(1-\rho^2)}\cdot\frac{g'(x)}{g''(x)}.
\end{align}

In the region $\{x<w_0\}$, the HJB equation \eqref{eqn: hjb no model unc} becomes
\begin{equation}
    -\frac{(\mu_1\sigma_2-\mu_2\sigma_1)^2+2(1-\rho)\mu_1\mu_2\sigma_1\sigma_2}{2\sigma_1^2\sigma_2^2(1-\rho^2)}\cdot\frac{g'(x)^2}{g''(x)}-\delta g(x)=0.
\end{equation}
Using the ansatz $g_1(x)=K_1x^{\gamma_1}$ yields $\gamma_1$ defined in \eqref{eqn: gamma1 no mod.unc.}. Hence, we obtain $\bar\pi_1(x)=\frac{x}{w_1}$ and $\bar\pi_2(x)=\frac{x}{w_2}$ as candidate maximizers for the reinsurance levels, where $w_1$ and $w_2$ are defined in \eqref{eqn: w1 and w2}. It follows that $w_0=\min\{w_1,w_2\}$. 

In the region $\{w_0<x<b^*\}$, the HJB equation \eqref{eqn: hjb no model unc} becomes $N_1g''(x)+N_2g'(x)-\delta g(x)=0$, where $N_1$ and $N_2$ are defined in \eqref{eqn: N1, N2, N3}. Its solution is given by
\begin{equation}\label{eqn: g2 no mod unc}
    g_2(x)=K_{2+}e^{\gamma_{2+} x} + K_{2-} e^{\gamma_{2-} x},
\end{equation}
where $\gamma_{2\pm}$ is defined in \eqref{eqn: v(x), K3, gamma2pm thm 1}. In the region $\{x> b^*\}$, $g$ must satisfy $g'(x)=a_2$. Hence, we must have $g_3(x)=a_2(x-b^*+K_3).$

To ensure that $g$ is twice continuously differentiable, we require $g$, $g'$, and $g''$ to be continuous at the switching points $w_0$ and $b^*$. Since in the neighborhood of $w_0$ the function $g$ satisfies \eqref{eqn:hjborig}, it suffices to show that $g$ and $g'$ are continuous at $x=w_0$. Let $\alpha_{2\pm:}=\frac{K_{2\pm}}{K_1}e^{\gamma_{2\pm}w_0}$. We then have the following system of equations:
\begin{equation}
    \begin{aligned}
        w_0^{\gamma_1}&= \alpha_{2+} + \alpha_{2-},\\
        \gamma_1w_0^{\gamma_1-1}&= \gamma_{2+}\alpha_{2+} +  \gamma_{2-}\alpha_{2-},
    \end{aligned}
\end{equation}
whose solution is given by $\alpha_{2+}=\frac{w_0^{\gamma_1-1}(\gamma_1-\gamma_{2-}w_0)}{\gamma_{2+}-\gamma_{2-}}$ and $\alpha_{2-}=\frac{w_0^{\gamma_1-1}(\gamma_{2+}w_0-\gamma_1)}{\gamma_{2+}-\gamma_{2-}}$.
It must be noted that the following equations hold: (i) $\gamma_{2\pm} =\frac{1-\gamma_1}{w_0}\left(-1\pm \frac{1}{\sqrt{1-\gamma_1}}\right)$; (ii) $\gamma_{2+}-\gamma_{2-}=\frac{2\sqrt{1-\gamma_1}}{w_0}$; (iii) $\gamma_1-\gamma_{2-}w_0=1+\sqrt{1-\gamma_1}$; and (iv) $\gamma_{2+}w_0-\gamma_{1}=\sqrt{1-\gamma_1}-1$.
Hence, $\alpha_{2+}=-\frac{w_0^{\gamma_1+1}}{2(1-\gamma_1)}\gamma_{2-}>0$ and $\alpha_{2-}=-\frac{w_0^{\gamma_1+1}}{2(1-\gamma_1)}\gamma_{2+}<0.$
We can then rewrite $g_2$ in \eqref{eqn: g2 no mod unc} as $g_2(x)=-\lambda\left[\gamma_{2-}e^{\gamma_{2+}(x-w_0)}+\gamma_{2+}e^{\gamma_{2-}(x-w_0)}\right],$
where $\lambda:=\frac{K_1 w_0^{\gamma_1+1}}{2(1-\gamma_1)}$.

At $x=b^*$, we have the following system of equations:
\begin{equation}\label{eqn: system at u2 no modunc}
    \begin{aligned}
        -\lambda\left[\gamma_{2-}e^{\gamma_{2+}(b^*-w_0)}+\gamma_{2+}e^{\gamma_{2-}(b^*-w_0)}\right]&= a_2 K_3,\\
        -\lambda\gamma_{2+}\gamma_{2-}\left[e^{\gamma_{2+}(b^*-w_0)}+e^{\gamma_{2-}(b^*-w_0)}\right]&= a_2 ,\\
        -\lambda\gamma_{2+}\gamma_{2-}\left[\gamma_{2+}e^{\gamma_{2+}(b^*-w_0)}+\gamma_{2-}e^{\gamma_{2-}(b^*-w_0)}\right]&= 0.
    \end{aligned}
\end{equation}
From the third equation in \eqref{eqn: system at u2 no modunc}, we must have $\gamma_{2+}e^{\gamma_{2+}(b^*-w_0)}+\gamma_{2-}e^{\gamma_{2-}(b^*-w_0)}=0,$
which is equivalent to \eqref{eqn: b^* no mod.unc.}.
Since $\gamma_{2+}+\gamma_{2-}<0$, we have $b^*>w_0$.

From the second equation in \eqref{eqn: system at u2 no modunc}, we obtain \eqref{eqn: lambda no mod.unc.}.
Dividing the first equation in \eqref{eqn: system at u2 no modunc} by the second equation yields
    \begin{align}
        K_3
        &=\frac{\gamma_{2-}\left(-\frac{\gamma_{2-}}{\gamma_{2+}}\right)^{\frac{\gamma_{2+}}{\gamma_{2+}-\gamma_{2-}}}+\gamma_{2+}\left(-\frac{\gamma_{2-}}{\gamma_{2+}}\right)^{\frac{\gamma_{2-}}{\gamma_{2+}-\gamma_{2-}}}}{\gamma_{2+}\gamma_{2-}\left[\left(-\frac{\gamma_{2-}}{\gamma_{2+}}\right)^{\frac{\gamma_{2+}}{\gamma_{2+}-\gamma_{2-}}}+\left(-\frac{\gamma_{2-}}{\gamma_{2+}}\right)^{\frac{\gamma_{2+}}{\gamma_{2+}-\gamma_{2-}}}\right]}=\frac{\gamma_{2+}+\gamma_{2-}}{\gamma_{2+}\gamma_{2-}}=\frac{1}{\delta}\left(\mu_1\frac{w_0}{w_1}+\mu_2\frac{w_0}{w_2}\right)=\frac{N_2}{\delta}.
    \end{align}Thus, we have obtained the value function given in \eqref{eqn: g no mod.unc.}. The corresponding optimal reinsurance levels are given by \eqref{eqn: pi no mod.unc.}.

\section{Additional Numerical Results}
\label{sec:MU comp b}

Figures \ref{fig:contour_b} and \ref{fig:MUCompb} and Table \ref{tab:MUComp w0 bstar b} illustrate the impact of model uncertainty on the insurer's optimal strategies when the reserve processes of both lines have the same drift and diffusion coefficients. The key result in these analyses is the symmetry in the effect of ambiguity aversion of each line on the zero-reinsurance threshold $w_0$ and the optimal barrier $b^*$. Furthermore, if only Line $i$ is ambiguity averse, then it is optimal for the insurer to internalize all the risk of Line $i$ when the underlying reserve value reaches the threshold $w_0$. The behavior of the value function in response to ambiguity aversion, shown in Figure \ref{fig:MUCompb} is consistent with observations in \citet{feng2021}, where the value function shifts downward when at least one line is ambiguity averse due to making decisions in the worst-case scenario.

\begin{figure}[h]
    \centering
        \begin{subfigure}[H]{0.25\textwidth}
            \centering
            \includegraphics[width=\linewidth]{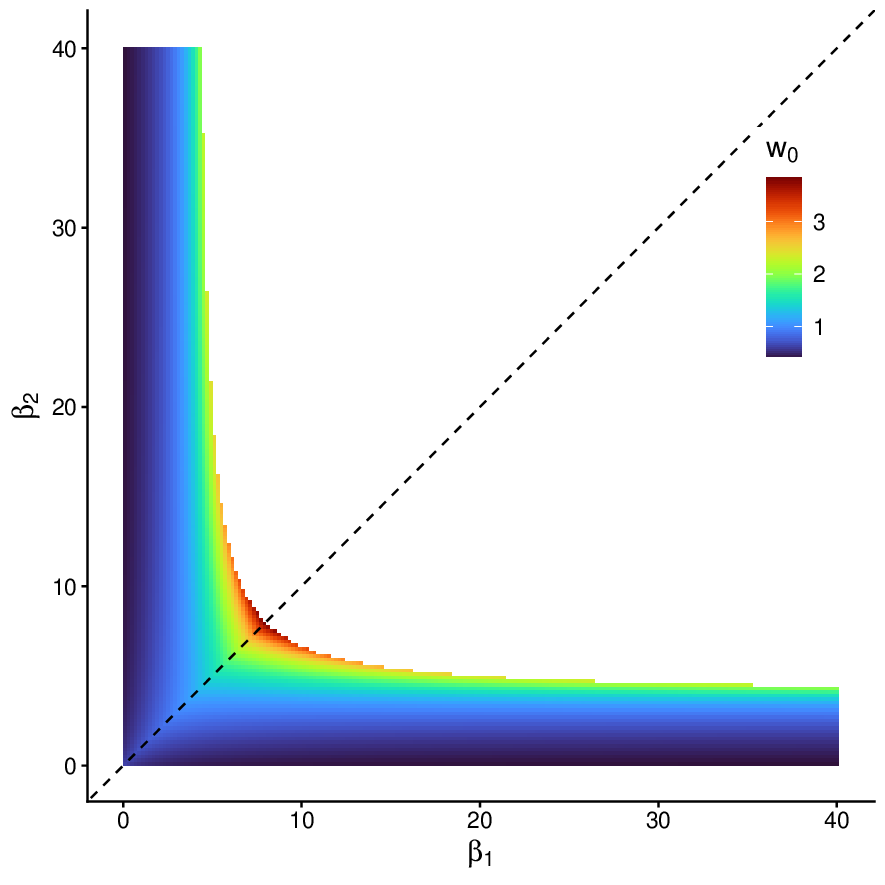} 
            \caption{$w_0$}
        \end{subfigure}
        \qquad
        \begin{subfigure}[H]{0.25\textwidth}
            \centering
            \includegraphics[width=\linewidth]{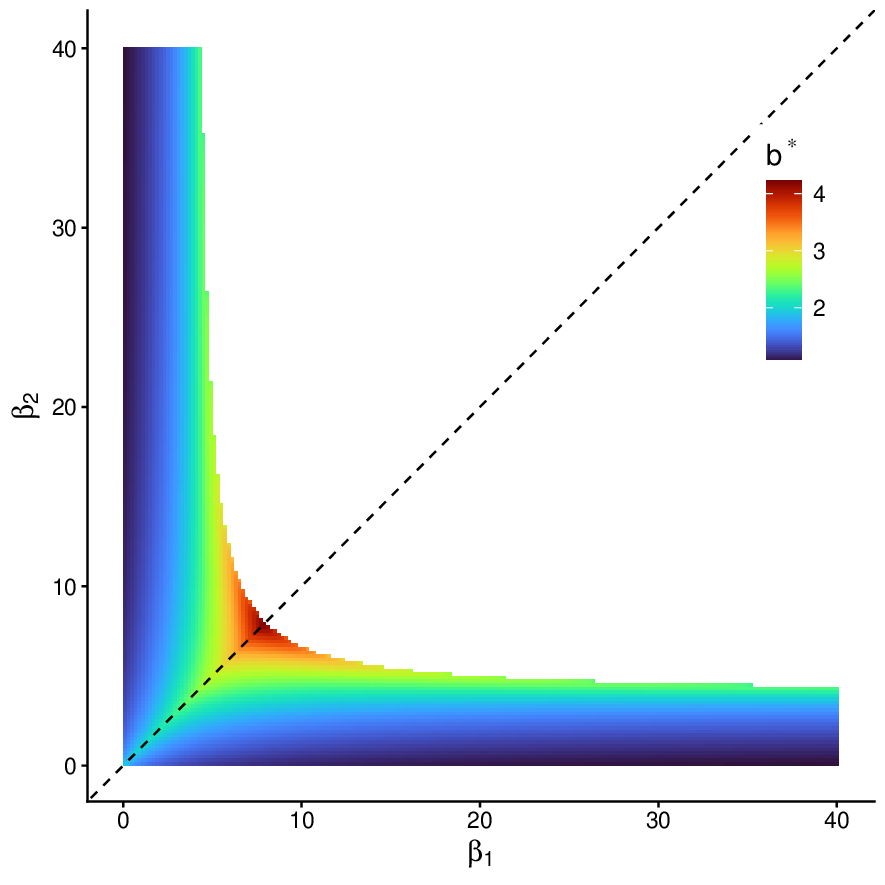}
            \caption{$b^*$}
        \end{subfigure}
        \\
        \caption{Contour plot of the reinsurance threshold $w_0$ and the optimal barrier $b^*$ as a function of the ambiguity-aversion parameters $\beta_1$ and $\beta_2$ when $\mu_1 = \mu_2 = 2$ and $\sigma_1 = \sigma_2 = 1$. 
        }
    \label{fig:contour_b}
\end{figure}

\begin{figure}[h]
    \centering
    \begin{subfigure}[]{0.25\textwidth}
        \centering
        \includegraphics[width=\linewidth]{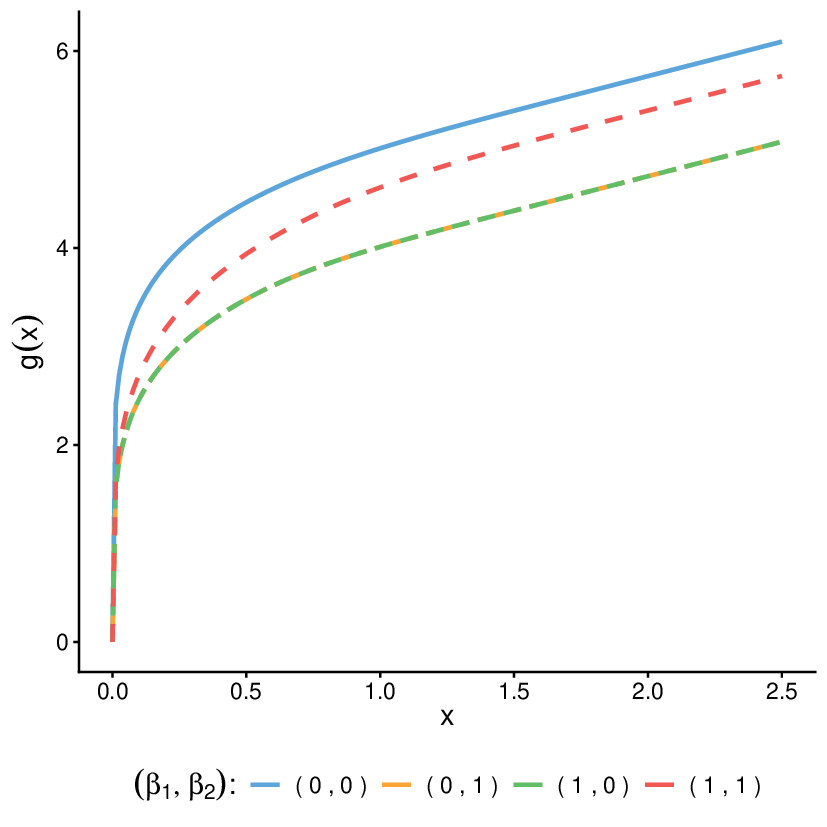} 
    \end{subfigure}
    \qquad
    \begin{subfigure}[]{0.25\textwidth}
        \centering
        \includegraphics[width=\linewidth]{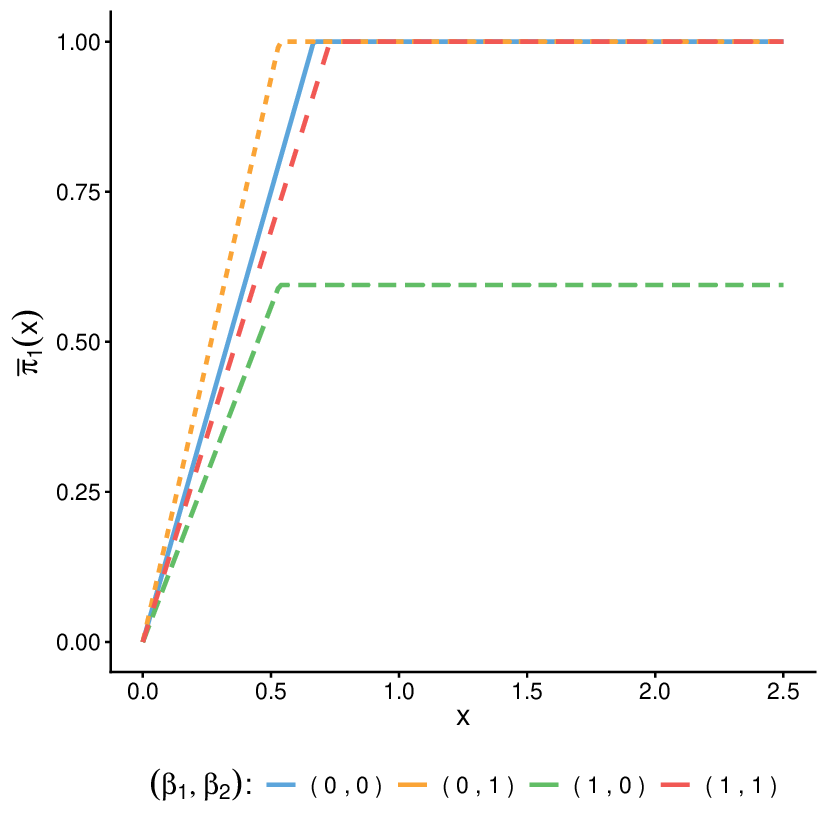}
    \end{subfigure}
    \qquad
    \begin{subfigure}[]{0.25\textwidth}
        \centering
        \includegraphics[width=\linewidth]{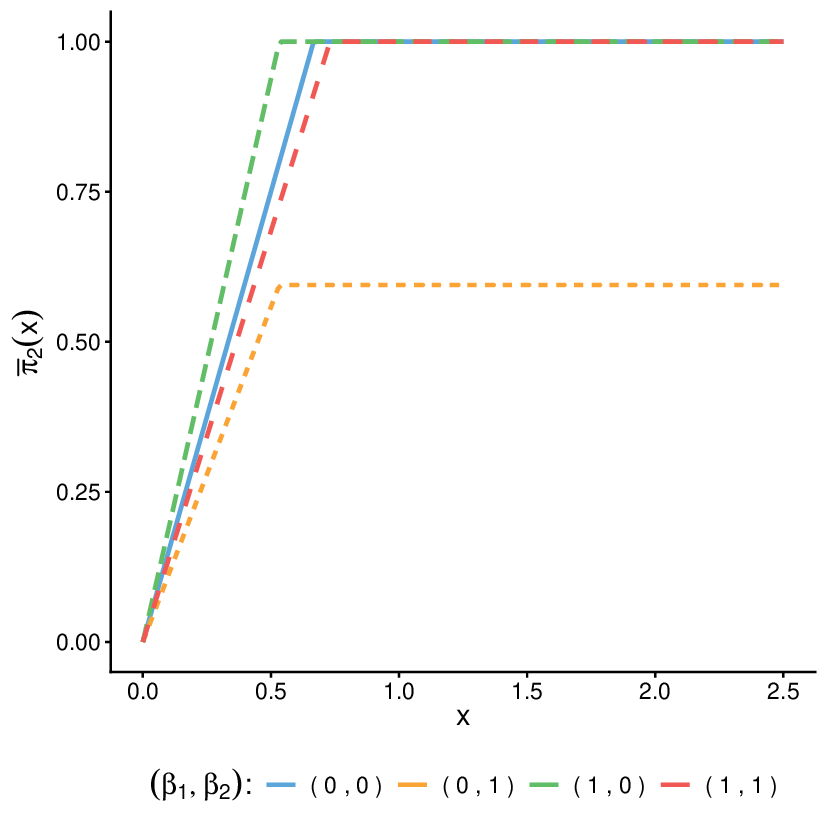}
    \end{subfigure}
    \caption{Value function $g(x)$ and optimal reinsurance strategies $\bar{\pi}_1(x), \bar{\pi}_2(x)$ for various values of $\beta_1$ and $\beta_2$ when $\mu_1 = \mu_2 = 2$ and $\sigma_1 = \sigma_2 = 1$. The value functions are identical when $(\beta_1, \beta_2) \in \{(0,1), (1,0)\}$.}
    \label{fig:MUCompb}
\end{figure}

\begin{table}[h]
    \caption{Impact of the ambiguity aversion parameter (and model uncertainty) on the reinsurance threshold $w_0$ and the optimal barrier $b^*$ when $\mu_1 = \mu_2 = 2$ and $\sigma_1 = \sigma_2 = 1$.}
    \label{tab:MUComp w0 bstar b}
    \centering
    \begin{tabular}{cccc}
    \toprule
    $\beta_1$ & $\beta_2$ & $w_0$ & $b^*$\\
    \midrule
    0 & 0 & 0.6666667 & 1.794599\\
    1 & 0 & 0.5330534 & 1.568814\\
    0 & 1 & 0.5330534 & 1.568814\\
    1 & 1 & 0.7306020 & 2.034350\\
    \bottomrule
    \end{tabular}
\end{table}

\end{document}